\documentclass[11pt]{article}
\usepackage{latexsym,amssymb,amsmath,amsthm}
\newtheorem{thm}{Theorem}[section]
\newtheorem{lem}[thm]{Lemma}

\newtheorem{rem}[thm]{Remark}
\newtheorem{dfi}[thm]{Definition}

\newtheorem{pro}[thm]{Proposition}
\newtheorem{assumption}[thm]{Assumption}

\newcommand{\RR}{{\mathbb R}}

\newcommand{\FC}{{\mathfrak FC}^{\infty}_b}
\newcommand{\ep}{\varepsilon}

\newcommand{\Casimir}{{\mathrm C}}
\newcommand{\Var}{{\rm Var}}
\newcommand{\calh}{{\cal H}}
\numberwithin{equation}{section}
\begin{document}
\newcounter{aaa}
\newcounter{bbb}
\newcounter{ccc}
\newcounter{ddd}
\newcounter{eee}
\newcounter{xxx}
\newcounter{xvi}
\newcounter{x}
\setcounter{aaa}{1}
\setcounter{bbb}{2}
\setcounter{ccc}{3}
\setcounter{ddd}{4}
\setcounter{eee}{32}
\setcounter{xxx}{10}
\setcounter{xvi}{16}
\setcounter{x}{38}
\title
{Vanishing of one dimensional $L^2$-cohomologies of loop groups
\footnote{This research was partially supported by
Grant-in-Aid for Scientific Research (A) No.~21244009.}}
\author{Shigeki Aida\\
Mathematical Institute\\
Tohoku University,
Sendai, 980-8578, JAPAN\\
e-mail: aida@math.tohoku.ac.jp}
\date{}
\maketitle
\begin{abstract}
Let $G$ be a simply connected 
compact Lie group.
Let $L_e(G)$ be the based loop group with the base point
$e$ which is the identity element. 
Let $\nu_e$ be the pinned Brownian motion measure on 
$L_e(G)$ and let $\alpha\in 
L^2(\wedge^1T^{\ast}L_e(G),\nu_e)\cap 
{\mathbb D}^{\infty,p}(\wedge^1T^{\ast}L_e(G),\nu_e)$~$(1<p<2)$ be a closed 
$1$-form on
$L_e(G)$. Using results in rough path analysis, 
we prove that there exists a measurable function
$f$ on $L_e(G)$ such that $df=\alpha$.
Moreover we prove that $\dim\ker \square=0$ for
the Hodge-Kodaira type operator $\square$ 
acting on $1$-forms on $L_e(G)$.
\end{abstract}

\section{Introduction}

Let $(M,g)$ be a
compact Riemannian manifold.
Let $d$ be the exterior differential operator
on $M$.
Let $d^{\ast}$ be the adjoint operator of $d$ in
the $L^2$ space of differential forms with respect to
the Riemannian volume.
Let $\square=dd^{\ast}+d^{\ast}d$.
Celebrated Hodge-Kodaira theorem asserts that
$\dim\ker\square|_p=b_p$.
Here $\square|_p$ denotes the Hodge-Kodaira operator 
on the space of $p$-forms 
and $b_p$ is the (real coefficient)
Betti number of $M$.
This theorem does not hold any more in non-compact
Riemannian manifold.
On the other hand,
in infinite dimension, there exist natural measures,
such as (pinned) Brownian motion measures,
on spaces of paths over a Riemannian
manifold.
Several researchers have been trying to 
establish a differential geometry
and analysis including Hodge-Kodaira type theorem
based on Brownian motion measures.
Since the path space $P_x(M)=C([0,1]\to M~|~\gamma(0)=x)$
has trivial topology, one natural guess is that
there are no harmonic forms on $P_x(M)$ except $0$-dimension.
When $M$ is a Euclidean space and $x=0$,
the path space with the Brownian motion measure 
is the Wiener space.
The notion of $H$-derivative fits in with the differential calculus
based on the Wiener measure and Sobolev spaces are
defined according to the $H$-derivative.
However the vanishing of $L^2$ cohomologies in the Sobolev space
category
is not trivial
because smooth functions in the sense of $H$-derivative
need not to be smooth in the sense of Fr\'echet.
The vanishing theorem on Wiener space was proved by
Shigekawa~\cite{shigekawa0} in the setting of Sobolev spaces.

When $M$ is a general Riemannian manifold, the Bismut tangent space is
used to define a vector field and $H$-derivative on $P_x(M)$.
The Bismut tangent space appeared naturally 
in the study of integration by parts formula and 
the quasi-invariance of (pinned) Brownian motion measures
\cite{driver}.
This tangent space depends on the choice of the 
metric connection on $M$ and if the curvature does not vanish, then
the Lie bracket of the vector fields do not
belong to the Bismut tangent space.
This shows a difficulty to study exterior 
differential operators on $P_x(M)$.
We refer the reader to \cite{elworthy-li, leandre}
for this problem.
Let us consider a special case where $M$ is a compact Lie group $G$.
Since the curvature of the right (or left) invariant connection
of $G$ is 0,
the Bismut tangent space of $P_e(G)$ which is defined by the
right (or left) invariant connection is stable under the Lie bracket
and the exterior differential operator on $P_e(G)$ is well-defined.
Here $e$ is the identity element.
We note that Hodge-Kodaira's theorem on 
$P_e(G)$ was studied in \cite{fang-franchi2} using Shigekawa's result
on a Wiener space.

Now let us consider the pinned case.
Let
$L_x(M)=C([0,1]\to M~|~\gamma(0)=\gamma(1)=x)$.
We have difficulties for the definition of the exterior differential
operator similarly to $P_x(M)$.
Instead of working on $L_x(M)$, some researchers studied
differential calculus over submanifolds in the
Wiener space~\cite{am1, kazumi-shigekawa, aida-ou, malliavin}.
Typical submanifolds are obtained by 
solutions of stochastic differential equations (=SDEs) on
$M$. See (\ref{submanifold}).
The tangent space of the submanifold is defined to be a closed subspace of
the Cameron-Martin subspace of the Wiener space
and the Lie brackets of vector fields on the submanifold
are also vector fields on the submanifold.
That is,
the exterior differential operator
is well-defined.
In a certain case, since the submanifold is isomorphic in some sense
to $L_x(M)$ which has non-trivial topology,
one may expect that the dimension of
harmonic forms on the submanifold coincides
with the Betti number of $L_x(M)$.
Note that solutions of SDE are smooth in the sense of $H$-derivative
(or in the sense of Malliavin calculus) but generally discontinuous functional
of Brownian motions.
Hence these submanifolds are not 
submanifolds in usual sense and
the link between the analysis over the submanifolds
and the ``topology'' of them are very unclear subject.
Nevertheless, Kusuoka succeeded in proving a Hodge-Kodaira theorem
and announced positive results in
\cite{kusuoka}.
See \cite{kusuoka91, kusuoka92} also.
We explain his results in Section 2 briefly.

In the present paper, we study a Hodge-Kodaira theorem 
for $1$-forms on the based loop group $L_e(G)$, where $G$ is a
compact Lie group.
The exterior differential operator $d$ on $L_e(G)$
is defined using
the right (or left) invariant connection
in the similar manner to $P_e(G)$.
When $G$ is simply connected, $\pi_2(G)=0$
and so $\pi_1(L_e(G))=0$ and the first Betti number is
$0$.
Therefore one may conjecture a vanishing theorem of
``the Hodge-Kodaira operator''acting on $1$-forms on $L_e(G)$.
Indeed, this is one of the main results of this paper.
Our proof of vanishing theorem is different from
Kusuoka's ones.
Here we explain the outline of our proof.
First, we show that if $\alpha$ is a closed $1$-form on
$L_e(G)$, then there exists a function $f$ on $L_e(G)$
such that $df=\alpha$.
To show this,
using a map from a Wiener space to
$L_e(G)$, we change the problem to a problem on an ``open subset'' 
${\cal D}_{\ep}$ of
the Wiener space.
The map is given by a solution of
an SDE on
$G$
and a ``retraction map''
on the Wiener space.
The ``open subset'' ${\cal D}_{\ep}$ is
homotopy equivalent to $L_e(G)$ in some sense.
The property of ``open'' should be understood 
in the sense of rough path analysis.
The topology in the rough path analysis
is finer than the usual uniform convergence 
topology of the Wiener space and the solution of
SDE can be viewed as a continuous functional with 
respect to the topology.
The most important next step is to establish a Poincar\'e's lemma
on a ball-like set $U_{r}(\varphi)$ in the sense of rough path analysis.
That is, we prove that a closed $1$-form on $U_{r}(\varphi)$
is exact.
Note that ${\cal D}_{\ep}$ has a countable cover by 
the ball-like sets.
In the third step, using the topological property of
$\pi_1(L_e(G))=0$, we prove that
a closed $1$-form on ${\cal D}_{\ep}$ is exact
putting together the locally established 
Poincar\'e's lemma on $U_{r}(\varphi)$.
Applying this, for any closed $1$-form on $L_e(G)$,
we can show the existence of $f$
such that $df=\alpha$.
Finally, using this result, hypoellipticity
of Bochner Laplacian and essential self-adjointness of
Hodge-Kodaira operator on $L_e(G)$,
we can get our vanishing theorem.

The paper is organized as follows.
In Section 2, we state main results in this paper
and make some remarks.
In Section 3, 
we recall the necessary results in rough path analysis.
We fix a subset $\Omega$ of $d$-dimensional 
Wiener space $W^d$ on which
Brownian rough path is defined.
Then a version of the solution of SDE on a compact
Lie group $G$
can be defined for all $w\in \Omega$. 
Also we give necessary estimates for iterated integrals
and Wiener integrals which will be used in Section 4.
In Section 4, we introduce subsets $U_{r,\varphi}$,
$U_{r}(\varphi)$ 
and prove a Poincar\'e's lemma for closed $1$-forms
on the subsets in
Theorem~\ref{a vanishing theorem 0} and Theorem~\ref{a vanishing theorem}.
This kind of Poincar\'e lemma was studied by Kusuoka~
\cite{kusuoka92}.
Also Shigekawa~\cite{shigekawa3} studied Hodge-Kodaira operator with
absolute boundary condition on convex domains in Wiener spaces.
We note that $U_r(\varphi)$ is not an $H$-convex domain
and the Poincar\'e lemma is non-trivial.
To prove Theorem~\ref{a vanishing theorem 0} and 
Theorem~\ref{a vanishing theorem}, 
we prove Poincar\'e's inequalities on finite dimensional approximation
of $U_{r,\varphi}$ in Claim~2 in the proof of Theorem~\ref{a vanishing theorem 0}.
The point is that the Poincar\'e constant is
independent of the dimensions.
At the end of this section, we introduce subsets $S$,
${\mathcal D}_{\ep}$ of $\Omega$.
$S$ is a ``submanifold'' of $\Omega$ and 
isomorphic to $L_e(G)$ by the solution 
of the SDE on $G$.
Note that $\Omega$ is not a linear space and
$S$ is not a submanifold in usual sense.
The subset ${\mathcal D}_{\ep}$ is a kind of
``tubular neighborhood'' of $S$ in $\Omega$.
In Section 5, we prove that ${\mathcal D}_{\ep}$ is covered by 
a countable family of $U_r(\varphi)$.
In Section 6,
we introduce notions of 
$H$-connectedness and $H$-simply connectedness.
We prove that ${\mathcal D}_{\ep}$ is an $H$-connected and 
$H$-simply connected set
when $G$ is simply connected.
This and Stokes theorem (Lemma~\ref{stokes}) are used to
prove the existence of a function $F$
such that $dF=\beta$ for a closed $1$-form $\beta$ on
${\cal D}_{\ep}$.
In Section 7, we prove several results which are necessary for
reducing the problem on $L_e(G)$ to that on ${\mathcal D}_{\ep}$.
First, we state relations between Sobolev spaces on $S$ and $L_e(G)$.
Next, we define a retraction map
from ${\mathcal D}_{\ep}$ onto $S$.
This kind of retraction map are used in \cite{am2,gross,aida-ou}.
We obtain a closed form on ${\mathcal D}_{\ep}$
by the pull-back of a closed form on $L_e(G)$ using the retraction map.
We apply results in Section 4 to this closed form.
In Section 8, we prove our main theorems.

\section{Statement of results and remarks}
Let $W^d$ be the set of continuous paths on $\RR^d$
defined on $[0,1]$ 
starting at $0$.
We denote by $\mu$ the Wiener measure on $W^d$
whose Cameron-Martin subspace is 
$
H=H^1([0,1]\to\RR^d~|~h_0=0).
$
We recall the definition of Sobolev spaces (\cite{ikeda-watanabe}) over
the Wiener space $(W^d,H,\mu)$.
Let $\FC(W^d,E)$ be the set of all smooth cylindrical
functions with values in a separable Hilbert space $E$.
When $E=\RR$, we may omit $E$.
We denote by ${\mathbb D}^{k,p}(W^d,E)$ the 
set of $L^p$ functions with respect to $\mu$
on $W^d$ with values in $E$ 
which are $k$-times $H$-differentiable and
all their derivatives are also in $L^p(\mu)$.
We write ${\mathbb D}^{\infty}(W^d,E)=\cap_{k\ge 0,p>1}
{\mathbb D}^{k,p}(W^d,E)$.
Let $G$ be a compact Lie group and consider a
bi-invariant Riemannian metric on $G$.
Let $P_e(G)$ be the set of continuous paths which are defined on 
the time interval $[0,1]$ and the starting point is $e$.
Let $L_e(G)$ be the subset of $P_e(G)$
which consists of paths whose end points are also $e$.
Let $\nu$, $\nu_e$ be the Brownian motion measure on $P_e(G)$ and
the pinned Brownian motion measure on $L_e(G)$ respectively.
These measures are defined by the diffusion semigroup
$e^{t\Delta/2}$, where $\Delta$ is the Laplace-Beltrami operator
which is defined by the bi-invariant Riemannian metric.
Let $T_e(G)={\mathfrak g}$ be the Lie algebra of $G$.
We identify it as the set of right invariant vector fields.
The bi-invariant Riemannian metric defines
an inner product on ${\mathfrak g}$.
We fix an orthonormal basis $\{\ep_1,\ldots,\ep_d\}$ which enables us
to identify ${\mathfrak g}$ and $\RR^d$,
where $d=\dim G$.
Therefore we identify $H$ and
a set of $H^1$-paths over ${\mathfrak g}$ starting at $0$
in this way.
Set $H_0=\{h\in H~|~h_1=0\}$.
We recall the definition of $H$-derivative on $P_e(G)$ and $L_e(G)$.
For a smooth cylindrical function $F(\gamma)$ on $P_e(G) (\mbox{or}~L_e(G))$,
we define the $H$-derivative of $F$ to be a measurable map
$G=G(\gamma)$ (actually smooth map in this case) from 
$P_e(G) (\mbox{or}~L_e(G))$ to
$H^{\ast} (\mbox{or}~H_0^{\ast})$ which satisfies 
$$
\left(G(\gamma), h \right)=
\lim_{\ep\to 0}
\frac{F(e^{\ep h}\gamma)-F(\gamma)}{\ep}
$$
for all $h\in H$~$(\mbox{or}~h\in H_0)$,
where $\left(\cdot,\cdot\right)$ is the pairing of
the elements of 
$H^{\ast} (\mbox{or}~H_0^{\ast})$ and $H (\mbox{or}~H_0)$.
We denote $G(\gamma)$ by $dF(\gamma)$.
This derivative corresponds to the derivative which is defined
by a right-invariant vector field $X_h$ on 
$L_e(G)$.
The tangent space $T_{\gamma}L_e(G)$ is defined to be the set
of all continuous mappings $h$ from
$[0,1]$ to $TG$ with $h(t)\in T_{\gamma(t)}G$
and $(R_{\gamma(\cdot)})_{\ast}^{-1}h(\cdot)\in H_0$.
Here $R_ab=ba$ for $a,b\in G$.
Naturally, $T_{\gamma}L_e(G)$ can be identified with
$H_0$.
Therefore $\otimes^pT_{\gamma}^{\ast}L_e(G), 
\wedge^pT_{\gamma}^{\ast}L_e(G)$ can be identified with
$\otimes^pH_0^{\ast}, \wedge^pH_0^{\ast}$
respectively.
Accordingly, measurable covariant tensor fields, differential forms on
$L_e(G)$ are defined to be measurable maps from
$L_e(G)$ to $\otimes^pH_0^{\ast}, \wedge^pH_0^{\ast}$
respectively.
The set $L_e(G)$ is a Banach manifold and there is a natural definition
of the (co)tangent bundle.
In this paper, we do not use the structure
but use the derivative in the $H$-direction and
the notation $T^{\ast}L_e(G)$ should be understood in such a sense.

To define Sobolev spaces of tensors over $L_e(G)$,
we use the Levi-Civita covariant derivative
$\nabla$ which is defined using the right invariant Riemannian metric.
The covariant derivative $\nabla$ is a mapping on
the smooth cylindrical
tensor fields such that
$\nabla T\in \FC(\otimes^{p+1}T^{\ast}L_e(G))$ for
$T\in \FC(\otimes^pT^{\ast}L_e(G))$~$(p=0,1,2,\ldots)$.
The Sobolev space ${\mathbb D}^{k,q}(\otimes^pT^{\ast}L_e(G),\nu_e)$
~$(k\in {\mathbb N}\cup \{0\}, q\ge 1)$
is the completion of $\FC(\otimes^pT^{\ast}L_e(G))$
by the norm $\|~\|_{k,q}$ such that
$$
\|T\|_{k,q}=
\left(\sum_{i=0}^k\|\nabla^iT\|_{L^q(\nu_e)}^q\right)^{1/q}.
$$
Also we have
$\nabla$ maps ${\mathbb D}^{k,q}(\otimes^pT^{\ast}L_e(G),\nu_e)$
to ${\mathbb D}^{k-1,q}(\otimes^{p+1}T^{\ast}L_e(G),\nu_e)$.
Let $X_{h_1},X_{h_2}$ be the vector field corresponding to
$h_i\in H_0$.
Then an easy calculation shows that
$[X_{h_1},X_{h_2}]F:=X_{h_1}\left(X_{h_2}F\right)
-X_{h_2}\left(X_{h_1}F\right)$
is equal to $X_{[h_2,h_1]}F$
for any smooth cylindrical function $F$.
Here $[h_2,h_1](t):=[h_2(t),h_1(t)]$.
Thus the exterior differential operator
$d$ is well-defined.
We refer the reader to \cite{aida-sobolev, fang-franchi} 
for the notion of tensor fields, covariant derivatives and
Sobolev spaces on $L_e(G)$.
We introduce a submanifold which is isomorphic to
$L_e(G)$ by the solution of the stochastic differential equation
in the sense of Stratonovich
on $G$:
\begin{eqnarray}
dX(t,a,w)&=&(L_{X(t,a,w)})_{\ast}\circ dw_t,\label{sde}\\
X(0,a,w)&=&a\in G.\nonumber
\end{eqnarray}
Here $L_ab=ab$ for $a,b\in G$
and $w_t$ is the $d$-dimensional standard Brownian
motion on ${\mathbb R}^d\cong {\mathfrak g}$ whose starting point is $0$.
That is, $w=(w_t)\in W^d$.
We fix an $\infty$-quasi-continuous version of 
$X(t,e,w)$ which is defined on a subset $\Omega$ of $W^d$.
See Theorem~\ref{Omega} and Proposition~\ref{solution of sde}.
Let 
\begin{equation}
S=\left\{w\in \Omega~|~X(1,e,w)=e\right\}.\label{submanifold}
\end{equation}
There exists a probability measure $\mu_e$ on $S$ 
which is given by
$$
d\mu_e(w)=p(1,e,e)^{-1}\delta_e(X(1,e,w))d\mu(w)
$$
where $\delta_e(X(1,e,w))$ is a positive generalized 
Wiener function \cite{sugita}.
Note that $\mu_e$ has no mass on any Borel measurable subset
$A$ with $C^s_q(A)=0$, where $C^s_q$ denotes the
$(q,s)$-capacity of $A$ and $q$ (the parameter of integrability)
is any number which is greater than $1$
and $s$ (the parameter of differentiability) 
is a sufficiently large positive number
which depends on the dimension of $G$.
Recall that a function $f$ on $W^d$ is said to be
$(q,s)$-quasi-continuous if for any $\ep>0$, there exists
a Borel measurable subset $A_{\ep}$ of $W^d$ such that
$C^s_q(A_{\ep})<\ep$ and $f|_{A_{\ep}^c}$ is continuous 
with respect to the topology of $W^d$.
Hence, for sufficiently large $s$, 
$(q,s)$-quasi-continuous function is a
$\mu_e$-almost everywhere defined Borel
measurable function.
Also $f$ is said to be $\infty$-quasi-continuous when
$f$ is $(q,s)$-quasi-continuous for all $(q,s)$.
We refer the reader to \cite{sugita, malliavin,ikeda-watanabe}
for these notions and results.
It is well-known that $X_{\ast}\mu_e=\nu_e$.
In fact, the map $X : (S,\mu_e)\to (L_e(G),\nu_e)$
is isomorphism in the sense of
Proposition~\ref{isomorphism}.
The covariant 
derivative $\nabla_S$ and the exterior differential
operator $d_S$
is defined on $S$ using the $H$-derivative on $W^d$
as in finite dimensions.
These differential operators are defined on 
Sobolev spaces of covariant tensor fields ${\mathbb D}^{k,q}(\otimes^pT^{\ast}S)$
and the space of $p$-forms 
${\mathbb D}^{k,q}(\wedge^pT^{\ast}S)$.
We denote by $\|~\|_{k,q}$ the Sobolev norm.
See \cite{kazumi-shigekawa, aida-ou} for these notions.
Here we present a first main theorem which
shows that
any closed $1$-form is exact on $S$.

\begin{thm}\label{main theorem 1}
Let $G$ be a simply connected compact Lie group.
There exists a sequence of
$\infty$-quasi-continuous functions
$\rho_n\in {\mathbb D}^{\infty}(W^d)$~$(n\in {\mathbb N})$
for which the following statements hold.

\noindent
$(1)$~
For any $n, w$, $0\le \rho_n(w)\le 1$ holds.
Moreover for any $r>1$, $k\in {\mathbb N}$,\\
$\lim_{n\to\infty}C^k_r\left(\{w\in W^d~|~\rho_n(w)=1\}^c\right)=0$ and
$\lim_{n\to\infty}\|\rho_n-1\|_{r,k}=0$.

\noindent
$(2)$ Let $1<p<2$.
Let $\theta\in L^2(\wedge^1T^{\ast}S, d\mu_e)
\cap {\mathbb D}^{\infty,p}(\wedge^1T^{\ast}S,d\mu_e)$ 
and assume that
$d_S\theta=0$~$\mu_e-a.s.$ on $S$.
Let $1<q<p$ and $k$ be a sufficiently large positive integer.
Then there exist $f$ and $f_n$
which satisfy {\rm (i)-(v)} below.

\begin{itemize}
\item[{\rm (i)}]~The function $f$ is a $\mu_e$-almost everywhere defined
measurable function on $S$.
Also $f_n$ is a $(q,k)$-quasi-continuous function on $W^d$
and $f_n\in {\mathbb D}^{k,q}(W^d)$.
\item[{\rm (ii)}]
For any $n$, 
$f_n(w)=f(w)$~$\mu_e$-almost everywhere on
$\{\rho_n(w)\ne 0\}\cap S$ and 
$d_Sf_n$ is equal to
$\theta$ for $\mu_e$-almost all elements of
$\{\rho_n(w)\ne 0\}\cap S$.

\item[{\rm (iii)}]
Let $\eta\in {\mathbb D}^{\infty}(W^d)$
be an $\infty$-quasi-continuous function.
Then it holds that
$f\rho_n\eta\in L^1(S,\mu_e)$.

\item[{\rm (iv)}]
For any $n$ and $\infty$-quasi-continuous map
$\eta\in {\mathbb D}^{\infty}(W^d,H^{\ast})$,
\begin{eqnarray*}
\lefteqn{\int_{S}
f(w)\rho_n(w)
\left(-(d_S\rho_n(w),\eta(w))+
\rho_n(w)\widetilde{d_S^{\ast}\eta}(w)\right)d\mu_e(w)}\\
& &=\int_{S}
\left(\tilde{\theta}(w)\rho_n(w)+
f(w)d_S\rho_n(w),\rho_n(w)\eta(w)\right)d\mu_e(w),
\end{eqnarray*}
where 
$\widetilde{d_S^{\ast}\eta}$ is an 
$\infty$-quasi-continuous modification of
$d_S^{\ast}\eta$ and so on.

\item[{\rm (v)}]~
Let $K>0$ and $\psi_K$ be a smooth function on $\RR$
such that $\psi_K(u)=u$ $(|u|\le K)$,
$\psi_K(u)=-K-1$ $(u\le -K-1)$,
$\psi_K(u)=K+1$  $(u\ge K+1)$ and set
$f^K=\psi_K(f)$.
Then $f^K\in {\mathbb D}^{1,2}(S,\mu_e)$
and $d_Sf^K=\psi_K'(f)\theta$ holds.
\end{itemize}
\end{thm}

The theorem above says that
$f$ is differentiable and $d_Sf=\theta$ holds on $S$
in the theorem's sense.
The function $\rho_n$ can be chosen independent of
$\theta$ and
actually they can be given more explicitly
using the iterated integrals of the Brownian motion
$w$.
On $L_e(G)$, we can state a corresponding theorem
to the above
in a very simple form.

\begin{thm}\label{main theorem on loop}
Let $1<p<2$.
Let $\alpha\in 
L^2(\wedge^1T^{\ast}L_e(G),\nu_e)\cap 
{\mathbb D}^{\infty,p}(\wedge^1T^{\ast}L_e(G),\nu_e)$ and assume that
$d\alpha=0$ on $L_e(G)$.
Then there exists a measurable function $f$ on $L_e(G)$ such that
the following hold.

\noindent
$(1)$~Let $\psi_K$ be the function which is defined in 
Theorem~$\ref{main theorem 1}$.
Set $f^K=\psi_K(f)$.
Then $f^K\in {\mathbb D}^{1,2}(L_e(G),\nu_{e})$ and
$df^K=\psi_K'(f)\alpha$.

\noindent
$(2)$~For any $h\in H_0$ and $\ep\ge 0$, we have
\begin{equation}
f(e^{\ep h}\gamma)-f(\gamma)=\int_0^{\ep}
\left(\alpha(e^{sh}\gamma),h\right)ds
\qquad \mbox{$\nu_e$-almost all $\gamma$}.
\end{equation}

\noindent
$(3)$~For any $h\in H_0$ and $q<p$,
\begin{equation}
\lim_{\ep\to 0}\left\|
\frac{f(e^{\ep h}\gamma)-f(\gamma)}{\ep}-
\left(\alpha(\gamma),h \right)\right\|_{L^q(L_e(G),\nu_e)}=0.
\end{equation}
\end{thm}

Using the above results,
we have a vanishing theorem for
the Hodge-Kodaira operator acting on
$1$-forms.
First we give the definition of the Hodge-Kodaira
operator.

\begin{dfi}
Let $d$ be the exterior differential operator
acting on $1$-forms on $L_e(G)$.
Let $d^{\ast}$ be the adjoint operator of
$d$.
We consider the closable form on $L^2(\wedge^1T^{\ast}L_e(G),\nu_e)$.
\begin{equation*}
{\cal E}(\alpha,\alpha)=
\left(d\alpha,d\alpha\right)_{L^2(\wedge^2T^{\ast}L_e(G))}+
\left(d^{\ast}\alpha,d^{\ast}\alpha\right)_{L^2(L_e(G))},
\end{equation*}
which is defined on $\FC(\wedge^1T^{\ast}L_e(G))$.
The Hodge-Kodaira operator $\square$ acting on $1$-forms is the
non-negative generator of the closed form of
the closure of the above.
\end{dfi}

We note that $\left(dd^{\ast}+d^{\ast}d, \FC(\wedge^1T^{\ast}L_e(G))\right)$
is essentially self-adjoint.
See \cite{shigekawa2}.
The statement in \cite{shigekawa2} is concerning 
Hodge-Kodaira operators on submanifolds in Wiener spaces.
However it can be applied to the case of $L_e(G)$ noting
Proposition~\ref{isomorphism}.
The following is our vanishing theorem.

\begin{thm}\label{main theorem 2}
Let $G$ be a simply connected compact Lie group.
Then
$\ker \square=\{0\}$.
Also it holds that
\begin{equation}
L^2(\wedge^1T^{\ast}L_e(G))=
\overline{\left\{df~|~f\in \FC(L_e(G))\right\}}\oplus
\overline{\left\{
d^{\ast}\alpha~|~\alpha\in \FC(\wedge^2T^{\ast}L_e(G))\right\}}.
\label{decomposition}
\end{equation}
\end{thm}

\bigskip

Finally, we make further remarks.

\noindent
(1)~
As noted in the introduction,
there are some difficulties to define a de Rham complex of
differential forms in the 
Sobolev space category on the general path spaces $P_x(M)$, $L_x(M)$.
However, we can define them on submanifolds in Wiener spaces.
See \cite{kazumi-shigekawa, kusuoka, aida-ou, am1}.
The proof in this paper can be applied to prove the vanishing 
of the $1$-dimensional $L^2$ cohomology of 
the submanifold which is isomorphic to $L_x(M)$ in the case
where $\pi_2(M)=0$ which is equivalent to $\pi_1(L_x(M))=0$.

\bigskip

\noindent
(2)~
We mention the works of Kusuoka in the introduction.
We explain Kusuoka's results.
Kusuoka defined a local Sobolev spaces
$
{\mathcal D}^{\infty,q}_{loc}(U,d\mu)
$
where $U$ is a subset of $W^d$ and
$q$ is the index of the integrability.
Based on these Sobolev spaces and several results on
the capacity which he introduced, Kusuoka announced the following
theorems in \cite{kusuoka}.
Let $M$ be a compact Riemannian manifold which is isometrically embedded
in $\RR^d$.
Let $P(x) : \RR^d\to T_xM$ be the projection operator and
consider a stochastic differential equation:
\begin{eqnarray*}
dX(t,x,w)&=&P(X(t,x,w))\circ dw_t,\\
X(0,x,w)&=&x\in M.
\end{eqnarray*}
There exists a probability measure 
$d\mu_x=p(1,x,x)^{-1}\delta_x(X(1,x,w))d\mu$ on the submanifold:
$$
S=\{w\in W^d~|~X(1,x,w)=x\}\subset W^d.
$$
Kusuoka proved that

\begin{thm}\label{kusuoka1}
There exists an isomorphism:
\begin{eqnarray*}
\Bigl\{\alpha\in {\mathcal D}_{loc}^{\infty,q}
(\wedge^pT^{\ast}S)~|~d_S\alpha=0\Bigr\}
/\Bigl\{d_S\beta~|~{\mathcal D}_{loc}^{\infty,q}
(\wedge^{p-1}T^{\ast}S)\Bigr\}
&\simeq& H^p({\cal M}_x,{\mathbb R}),
\end{eqnarray*}
where 
\begin{eqnarray*}
{\cal M}_x&=&\Bigl\{h\in H~|~\xi(1,x,h)=x, 
\mbox{where $\xi(t,x,h)$ is the solution to}\\
& &\dot{\xi}(t,x,h)=P(\xi(t,x,h))\dot{h}(t), \xi(0,x,h)=x,
~t\ge 0
\Bigr\}
\end{eqnarray*}
and $H^p({\cal M}_x,\RR)$ is the de Rham cohomology of
${\cal M}_x$.
\end{thm}

The subset ${\cal M}_x$ is a Hilbert manifold in usual sense.
Let $H^1\cap L_x(M)$ be the subset of $H^1$-paths of
$L_x(M)$.
Noting that $H^1\cap L_x(M)$ and ${\cal M}_x$ is $C^{\infty}$-homotopy
equivalent, the conclusion of Theorem~\ref{kusuoka1}
is natural.
Let $\square=d_S^{\ast}d_S+d_sd_S^{\ast}$
and $\square|_p$ be the restriction on $p$-forms.
They are defined as the Friedrichs extension of
them on some cores.
Another Kusuoka's result is as follows.

\begin{thm}\label{kusuoka2}
There exists a mapping $j_p : \ker\square|_p\to H^p({\cal M}_x,{\mathbb R})$
such that

$(1)$~$j_p$ is surjective for $p=0,1,2,\ldots$.

$(2)$~$j_p$ is injective for $p=0,1$.
\end{thm}

Therefore our results give another proof to some special cases
of his results.
We may prove a vanishing theorem on
a ``contractible domain'' of $S$ using the method in our paper.
Moreover, combining the usage of the ${\check {\rm C}}$ech cohomology,
we may prove the isomorphism between 
$H_1(H^1\cap L_x(M), \RR)$ and $\ker \square|_1$ based on our proof.
However we do not pursue this direction in this paper.

\section{Preliminary from rough path analysis}

The solutions of It\^o's stochastic differential equations
are measurable functions on $W^d$, but, they
are not continuous in the uniform convergence topology of $W^d$ 
in general.
The reason of the discontinuity is clarified by the rough path
analysis \cite{lyons98, lq, fv2}. 
In rough path analysis, we need to consider objects
which consist of the path and the iterated integrals.
To explain the iterated integrals, 
we take two continuous paths 
$x=x_t=(x^1_t,\ldots,x^d_t)$,
$y=y_t=(y^1_t,\ldots,y^d_t)$~$(0\le t\le 1)$ on $\RR^d$.
Suppose that $x$ or $y$ is a bounded variation path.
Then we can define for $0\le s\le t\le 1$
\begin{eqnarray}
C(x,y)_{s,t}&=&\int_s^t(x_u-x_s)\otimes dy_u\nonumber\\
&=&\sum_{1\le i,j\le d}\left(\int_s^t(x^i_u-x^i_s)dy^j_u\right)
e_i\otimes e_j\in \RR^d\otimes \RR^d
\label{bdd-cont-2nd-level}
\end{eqnarray}
as a Stieltjes integral.
Here $e_i={}^t(0,\ldots,\stackrel{i}{1},\ldots,0)$.
We introduce a function spaces for these iterated integrals.
Let $\Delta=\{(s,t)\in \RR^2~|~0\le s\le t\le 1\}$.
Let $V$ be a normed linear space.
For a Borel measurable mapping $\phi : \Delta\to V$, define
\begin{eqnarray*}
\|\phi\|_{m,\theta}&=&\left\{\int_0^1\int_0^t
\frac{|\phi(s,t)|^m}{(t-s)^{2+m\theta}}dsdt\right\}^{1/m},
\end{eqnarray*}
where, $m$ is a positive even integer and $0<\theta<1$. 
We denote the set of all measurable mappings $\phi$ from $\Delta$ to
$V$ satisfying $\|\phi\|_{m,\theta}<\infty$ by
$L_{m,\theta}(\Delta\to V)$.
Also we define $W_{m,\theta}(\Delta\to V)=L_{m,\theta}(\Delta\to V)\cap
C(\Delta\to V)$, where
$C(\Delta\to V)$ is the set of all continuous mappings 
from $\Delta$ to $V$.
Note that $L_{m,\theta}(\Delta\to V)$ is a separable Banach space.
Also for a measurable mapping $\phi : \Delta\to V$,
define
$$
\|\phi\|_{H,\theta}=\sup_{0\le s<t\le 1}\frac{|\phi(s,t)|}{|t-s|^{\theta}}.
$$
For $w\in W^d$, define $\bar{w}_{s,t}=w_t-w_s$~$((s,t)\in \Delta)$.
We denote by $W_{m,\theta}(\RR^d)$ all $w\in W^d$ with
$\|\bar{w}\|_{m,\theta}<\infty$.
We write $\|w\|_{m,\theta}$ instead of $\|\bar{w}\|_{m,\theta}$.
Note that the H\"older norm 
$\|w\|_{H,\theta}:=\|\bar{w}\|_{H,\theta}$
is weaker than the norm of $\|~\|_{m,\theta}$ by 
a result of \cite{grr}.
However this kind of statement does not hold for
general $\phi\in W_{m,\theta}(\Delta\to V)$ without additional
assumptions.
See Lemma~\ref{Hoelder and Besov}.
Let $M_{m,\theta}=\sup_{x\ne 0, x\in W_{m,\theta/2}({\mathbb R})}
\frac{\|x\|_{H,\theta/2}}{\|x\|_{m,\theta/2}}$.
Wiener measure $\mu$ satisfies that
$\mu(W_{m,\theta/2}({\mathbb R}^d))=1$ for all $0<\theta<1$.
Note that $W_{m,\theta}(\RR^d)$ is a separable 
Banach space.
If $x$ and $y$ are Lipschitz continuous paths, then $C(x,y)\in 
W_{m,\theta}(\Delta\to \RR^d\otimes \RR^d)$ for all $(m,\theta)$
with $m(1-\theta)>2$.
See Lemma~\ref{pointwise estimate}.

Let $w=w_t=(w^1_t,\ldots,w^d_t)\in W^d$
and 
$w(N)_t$ be the dyadic polygonal approximation of $w$.
Namely, $w(N)_t=w_t$ for $t=\frac{k}{2^N}$~$(k=0,1,\ldots,2^N)$
and $t\mapsto w(N)_t$~$(\frac{k}{2^N}\le t\le \frac{k+1}{2^N},
0\le k\le 2^N-1)$ are linear functions.
Also let
$w(N)^i=(w(N),e_i)$ and define $w(N)^{\perp,i}=w^i-w(N)^i$,
$w(N)^{\perp}=w-w(N)$.
We need a probabilistic argument to define the 
integrals $C(w^i,w^j)_{s,t}$, $C(w,w)_{s,t}$ in contrast with
$C(w(N),w)$, $C(w(N)^i,w^j)$.
Indeed, they are Stratonovich integrals and
we fix a version of them below.

\begin{thm}\label{Omega}
Let $\Omega$ be the subset of $W^d$
which consists of $w$
satisfying the following {\rm (i)-(iii)}.
\begin{itemize}
\item[{\rm (i)}] $\lim_{N\to\infty}w(N)$ converges in
$W_{m,\theta}(\RR^d)$ for all $(m,\theta)$
with $m(1-\theta)>2$.
\item[{\rm (ii)}] $\lim_{N\to\infty}C(w(N),w(N))$ converges in 
$W_{m,\theta}(\Delta\to \RR^d\otimes \RR^d)$ for all 
$(m,\theta)$ with $m(1-\theta)>2$.
Moreover these converge with respect to all norms
$\|~\|_{H,\theta}$~$(0<\theta<1)$.
\item[{\rm (iii)}] $\lim_{N\to\infty}C(w(N)^{\perp},w(N))$
and $\lim_{N\to\infty}C(w(N),w(N)^{\perp})$ converge to $0$ in
$W_{m,\theta}(\Delta\to\RR^d\otimes \RR^d)$ for all
$(m,\theta)$ with $m(1-\theta)>2$.
Moreover these converge to $0$ with respect to all norms
$\|~\|_{H,\theta}$~$(0<\theta<1)$.
\end{itemize}
Then $\Omega^c$ is a slim set and it holds that
$H\subset \Omega$ and $\Omega+H\subset \Omega$.
\end{thm}

A subset $A$ of $W^d$ is called a slim set if
$C^{s}_q(A)=0$ for all $s>0$ and $q>1$.
See \cite{malliavin}.
We note that 
$C(w^i,z^j)$ is meaningless even if 
both $w=(w^i)$ and $z=(z^j)$ belong to $\Omega$ generally.
In rough path analysis, it is proved in many papers that
the Wiener measure of the total set of paths which satisfy 
(i), (ii) above is 1.
We need the property (iii) for our applications. 
The property (iii) is essential in \cite{aida-wpoincare} also.
The fact that $\Omega^c$ is a slim set is proved in \cite{higuchi}.
We give the proof of Theorem~\ref{Omega} for the sake of completeness,
together with that of Theorem~\ref{quasi-cont-version}.

We use the following notation.
For $w\in \Omega$, we define
\begin{eqnarray}
C(w,w)_{s,t}&=&\lim_{N\to\infty}C(w(N),w(N))_{s,t}\label{convergence1}\\
C(w^i,w^j)_{s,t}&=&\lim_{N\to\infty}C(w(N)^i,w(N)^j)_{s,t}
\label{convergence2}
\end{eqnarray}
where $1\le i,j\le d$.
Then it holds that for any $w=(w^i)\in \Omega$
and $0\le s\le t\le 1$,
\begin{equation}
C(w^i,w^j)_{s,t}=
(w^i_t-w^i_s)(w^j_t-w^j_s)-C(w^j,w^i)_{s,t}
\label{ibp for rough path}
\end{equation}
and $\|C(w(N)^{\perp,i},w(N)^{\perp,j})\|_{m,\theta}$ converges to
$0$ for all $1\le i,j\le d$ and
$(m,\theta)$ with $m(1-\theta)>2$.
For later use, we define
$
\Omega_N=\{w(N)~|~w\in \Omega\}
$
and
$
\Omega_N^{\perp}=\{w-w(N)~|~w\in \Omega\}.
$
We denote the laws of $w(N)$ and $w(N)^{\perp}$ 
by $\mu_N$ and $\mu_N^{\perp}$ respectively.
Note that $\Omega_N$ is the same as the set of all piecewise linear
continuous paths $w$ such that 
$t\mapsto w_t$~$(\frac{k}{2^N}\le t\le \frac{k+1}{2^N},
0\le k\le 2^N-1)$ is a linear function and this space 
is isomorphic to $\RR^{2^Nd}$.
Also $w\in\Omega_N^{\perp}$  
is equivalent to $w\in \Omega$ and
$w(k/2^N)=0$ for all integers with $0\le k\le 2^N$.
For simplicity, we may use the notation $\xi=(\xi^1,\ldots,\xi^d)$ and 
$\eta=(\eta^1,\ldots,\eta^d)$ to denote the element of 
$\Omega_N$ and $\Omega_N^{\perp}$ respectively.

\begin{thm}\label{quasi-cont-version}
Let us fix a positive even integer $m$ and a positive number $\theta$ with
$m(1-\theta)>2$.
Let ${\mathfrak T}$ be the weakest topology such that
$w(\in W^d)\mapsto w(k/2^N)$ are continuous mappings
for all $k,N$.
The mappings $w(\in \Omega)\mapsto C(w^i,w^j)\in W_{m,\theta}(\Delta\to\RR)$
and $w(\in \Omega)\mapsto w\in W_{m,\theta/2}$
are $\infty$-quasi-continuous for all $i,j$ with respect to the
topology ${\mathfrak T}$.
\end{thm}

To prove these theorems, we use the following lemmas.

\begin{lem}\label{tchebycheff}
Let $u\in {\mathbb D}^{s,q}(W^d)$ and
$\tilde{u}$ be the $(q,s)$-quasi-continuous version of $u$.
Then there exists a positive number $C_{s,q}$ which is independent of
$u$ such that for all $R>0$, the $(q,s)$-capacity satisfies 
$$
C^s_q\left(\{w\in W^d~|~|\tilde{u}(w)|>R\}\right)
\le R^{-1} C_{s,q}\|u\|_{s,q}.
$$
\end{lem}

We refer the proof of Lemma~\ref{tchebycheff} to
\cite{malliavin}.
In Lemma~\ref{pointwise estimate}~(2),
the estimates (\ref{varphi}), (\ref{Cwvarphi}), (\ref{Cvarphiw})
hold with different constants under the weaker
assumption $m(1-\theta)>2$.
This is checked 
by the same proof as given below.
Under the stronger assumption 
$m(1-\theta)>4$, the constants
in the estimates 
(\ref{varphi}), (\ref{Cwvarphi}), (\ref{Cvarphiw})
are simpler.
We use this lemma in the proof of Lemma~\ref{U and V} too
and the simpleness of the constants make the calculation
simpler.
Therefore we consider the stronger assumption.
In the calculation below, 
constants $C$ may change line by line.

\begin{lem}\label{pointwise estimate}
$(1)$~
Let $x, y\in W_{m,\theta/2}(\RR)$
and set $(\bar{x}\cdot \bar{y})_{s,t}=(x_t-x_s)(y_t-y_s)$~
$(0\le s\le t\le 1)$.
Then
\begin{equation}
\|\bar{x}\cdot \bar{y}\|_{m,\theta}\le M_{m,\theta}\|x\|_{m,\theta/2}
\|y\|_{m,\theta/2},\label{product of bar}
\end{equation}
where $M_{m,\theta}=\sup_{x\ne 0, x\in W_{m,\theta/2}({\mathbb R})}
\frac{\|x\|_{H,\theta/2}}{\|x\|_{m,\theta/2}}$.

\noindent
$(2)$~
Let $w\in W_{m,\theta/2}(\RR)$ and $\varphi\in H$.
Suppose that $m(1-\theta)>4$.
Then 
\begin{eqnarray}
\|\varphi\|_{m,\theta/2}&\le&\|\varphi\|_H,\label{varphi}\\
\|C(w,\varphi)\|_{m,\theta}&\le&\|w\|_{m,\theta/2}\|\varphi\|_H,
\label{Cwvarphi}\\
\|C(\varphi,w)\|_{m,\theta}&\le&2\|w\|_{m,\theta/2}\|\varphi\|_H,
\label{Cvarphiw}\\
\|D\|C(w,\varphi)\|_{m,\theta}^m\|_H&\le& C_{m,\theta}
\|C(w,\varphi)\|_{m,\theta}^{m-1}
\|\varphi\|_{m,\theta/2},\label{DCwvarphi}\\
\|D\|C(\varphi,w)\|_{m,\theta}^m\|_H&\le&
C_{m,\theta}\|C(\varphi,w)\|_{m,\theta}^{m-1}\|\varphi\|_{m,\theta/2},
\label{DCvarphiw}
\end{eqnarray}
where $D$ denotes the $H$-derivative and $\|~\|_H$
stands for the norm of the Cameron-Martin subspace
$H$.
\end{lem}

\begin{proof}
(1)~
We have
\begin{eqnarray*}
\|\bar{x}\cdot \bar{y}\|_{m,\theta}^m&=&
\int_0^1\int_0^t\frac{|(x_t-x_s)(y_t-y_s)|^m}{(t-s)^{2+m\theta}}dsdt\nonumber\\
&\le&
\int_0^1\int_0^t\frac{|(x_t-x_s)|^m(M_{m,\theta}\|y\|_{m,\theta/2})^m}
{(t-s)^{2+m\theta/2}}dsdt
=M_{m,\theta}^m\|x\|_{m,\theta/2}^m\|y\|_{m,\theta/2}^m.
\end{eqnarray*}

\noindent
(2)~The estimate (\ref{varphi}) follows from 
\begin{equation}
|\varphi_t-\varphi_s|\le \|\varphi\|_H(t-s)^{1/2}.\label{Holder varphi}
\end{equation}
We prove (\ref{Cwvarphi}).
Using the H\"older inequality, we have
\begin{eqnarray*}
\lefteqn{\frac{\left|\int_s^t(w(u)-w(s))\dot{\varphi}(u)du\right|^m}
{(t-s)^{2+m\theta}}}\nonumber\\
& &\le
\frac{1}{(t-s)^{m\theta/2}}
\left(\int_s^t\frac{|w(u)-w(s)|}{|u-s|^{(2+m\theta/2)/m}}
|\dot{\varphi}(u)|du
\right)^m\nonumber\\
& &\le
\int_s^t\frac{|w(u)-w(s)|^m}{|u-s|^{2+m\theta/2}}du
\frac{1}{(t-s)^{m\theta/2}}
\left(\int_s^t|\dot{\varphi}(u)|^{m/(m-1)}du\right)^{m-1},
\end{eqnarray*}
\begin{eqnarray*}
\frac{1}{(t-s)^{m\theta/2}}
\left(\int_s^t|\dot{\varphi}(u)|^{m/(m-1)}du\right)^{m-1}&\le&
\frac{1}{(t-s)^{m\theta/2}}
\left(\int_s^t|\dot{\varphi}(u)|^2du\right)^{m/2}
(t-s)^{\frac{m-2}{2}}\nonumber\\
&\le&
(t-s)^{\frac{(m-2)-m\theta}{2}}\|\varphi\|_{H}^m\nonumber\\
&\le&\|\varphi\|_H^m.
\end{eqnarray*}
Hence
\begin{eqnarray*}
\|C(w,\varphi)\|_{m,\theta}^m&\le&
\int_0^1\int_0^t
\left(\int_s^t\frac{|w(u)-w(s)|^m}{|u-s|^{2+m\theta/2}}du\right)dsdt
\|\varphi\|_H^m\nonumber\\
&=&\int_0^1\int_0^t
\left(\int_0^u\frac{|w(u)-w(s)|^m}
{|u-s|^{2+m\theta/2}}ds\right)dudt
\|\varphi\|_H^m\nonumber\\
&\le&\|w\|_{m,\theta/2}^m\|\varphi\|_H^m.
\end{eqnarray*}
We prove (\ref{Cvarphiw}).
Noting 
that for
$1$-dimensional paths $x,y$,
\begin{equation}
C(x,y)_{s,t}=(x_t-x_s)(y_t-y_s)-C(y,x)_{s,t},\label{ibp for x and y}
\end{equation}
we have
\begin{eqnarray*}
\|C(\varphi,w)\|_{m,\theta}&\le&
\|C(w,\varphi)\|_{m,\theta}
+\left(\int_0^1\int_0^t
\frac{|(w_t-w_s)(\varphi_t-\varphi_s)|^m}
{(t-s)^{2+m\theta}}dsdt\right)^{1/m}
\nonumber\\
&\le&
\|C(w,\varphi)\|_{m,\theta}+
\|w\|_{m,\theta/2}\|\varphi\|_H,
\end{eqnarray*}
where we have used (\ref{Holder varphi}).
This and (\ref{Cwvarphi}) prove (\ref{Cvarphiw}).
We consider (\ref{DCwvarphi}).
Let $h\in H$.
We have
\begin{eqnarray*}
D_{h}\left(\int_s^t(w_u-w_s)d\varphi_u\right)
&=&
(\varphi_t-\varphi_s)(h_t-h_s)
-\int_s^t(\varphi_u-\varphi_s)\dot{h}_udu.
\end{eqnarray*}
Therefore
\begin{eqnarray*}
\lefteqn{D_{h}\left(\|C(w,\varphi)\|_{m,\theta}^m\right)}\nonumber\\
& &=
m\int_0^1\int_0^t
\frac{\left((\varphi_t-\varphi_s)(h_t-h_s)-
C(\varphi,h)_{s,t}\right)C(w,\varphi)_{s,t}^{m-1}}{(t-s)^{2+m\theta}}
dsdt.
\end{eqnarray*}
Using the H\"older inequality, (\ref{Cwvarphi})
and (\ref{Holder varphi}), we get
\begin{eqnarray*}
D_{h}(\|C(w,\varphi)\|_{m,\theta}^m)&\le&
C_{m,\theta}\left(\|\varphi\|_{m,\theta/2}\|h\|_H
+\|C(\varphi,h)\|_{m,\theta}\right)\|C(w,\varphi)
\|_{m,\theta}^{m-1}\nonumber\\
&\le&2C_{m,\theta}
\|\varphi\|_{m,\theta/2}\|h\|_H\|C(w,\varphi)\|_{m,\theta}^{m-1}
\end{eqnarray*}
which proves (\ref{DCwvarphi}).
As for (\ref{DCvarphiw}),
noting that
$$
D_{h}\left(\int_s^t(\varphi_u-\varphi_s)dw_u\right)
=\int_s^t(\varphi_u-\varphi_s)\dot{h}_udu,
$$
we can prove (\ref{DCvarphiw}) similarly to (\ref{DCwvarphi}).
\end{proof}

\begin{lem}\label{Hoelder and Besov}
Let $0<\theta<1$ and $m$ be a positive even integer.
There exists a positive constant $N_{m,\theta}$ such that
for all $x,y\in H$, we have
\begin{equation}
\|C(x,y)\|_{H,\theta}\le
N_{m,\theta}\left(\|C(x,y)\|_{m,\theta}+\|x\|_{m,\theta/2}\|y\|_{m,\theta/2}\right).
\end{equation}
\end{lem}

\begin{proof}
It suffices to prove the case where
$\|y\|_{m,\theta/2}=1$.
In this case, the proof is almost similar to
\cite{grr} noting Chen's identity:
$
C(x,y)_{s,t}=C(x,y)_{s,r}+C(x,y)_{r,t}+(x(r)-x(s))
\otimes (y(t)-y(r))\quad 0<s<r<t<1.
$
See also \cite{fv}
\end{proof}

\begin{proof}
[Proof of Theorem~$\ref{Omega}$ and Theorem~$\ref{quasi-cont-version}$]

Let $z(N)=w(N)-w(N-1)$~$(N=1,2,\ldots)$, where $w(0)=0$.
Then $\{z(N) ; N=1,2,\ldots\}$ are independent random variables 
with values in the set of piecewise linear functions.
Using explicit form of $z(N)$, we have
\begin{eqnarray}
E[|w(N)_t-w(N)_s|^2]&\le& d|t-s|\\
E[|z(N)_t-z(N)_s|^2]&\le& C_d\min\left(|t-s|,2^{-N}\right)\\
E[|w(N)^{\perp}_t-w(N)^{\perp}_s|^2]&\le&C_d
\min\left(|t-s|,2^{-N}\right).\label{estimate for wNperp}
\end{eqnarray}
We estimate $L^2$-norm of $\|z(N)^i\|_{m,\theta/2}^m$.
\begin{eqnarray*}
\lefteqn{
\left\|\|z(N)^i\|_{m,\theta/2}^{m}\right\|_{L^2(\mu)}}\nonumber\\
& &=\left\{\int_{W^d}d\mu
\iint_{(s,t)\in \Delta, (s',t')\in \Delta}
\frac{(z(N)^i_t-z(N)^i_s)^m(z(N)^i_{t'}-z(N)^i_{s'})^m}
{|t-s|^{2+m\theta/2}|t'-s'|^{2+m\theta/2}}dsdt~ds'dt'\right\}^{1/2}
\nonumber\\
& &\le
\iint_{(s,t)\in \Delta}
\frac{E[(z(N)^i_t-z(N)^i_s)^{2m}]^{1/2}}{|t-s|^{2+m\theta/2}}dsdt
\nonumber\\
& &= C_m\iint_{(s,t)\in \Delta}
\frac{E[(z(N)^i_t-z(N)^i_s)^{2}]^{m/2}}{|t-s|^{2+m\theta/2}}dsdt
\nonumber\\
& &\le C_m\iint_{(s,t)\in \Delta}
\frac{\min(|t-s|, 2^{-N})^{m/2}}{|t-s|^{2+m\theta/2}}dsdt
\nonumber\\
& &\le
C_m\iint_{(s,t)\in \Delta}
|t-s|^{\frac{m}{2}(1-\ep-\theta)-2}2^{-\ep mN/2}dsdt.
\end{eqnarray*}
Thus if $m(1-\theta)>2$, choosing 
an appropriate $\ep>0$, there exists a positive number
$C_{m,\theta,\ep}$
\begin{equation}
\left\|\|z(N)^i\|_{m,\theta/2}^{m}\right\|_{L^2(\mu)}
\le C_{m,\theta,\ep}2^{-\ep mN/2}.\label{zN}
\end{equation}
Noting $E[|w(N)^i_t-w(N)^i_s|^{2m}]\le
E[|w^i_t-w^i_s|^{2m}]\le C_m|t-s|^m$ 
and by the calculation similar to the above, if $m(1-\theta)>2$,
\begin{eqnarray}
\left\|\|w(N)\|_{m,\theta/2}^m\right\|_{L^2(\mu)}&\le& C_{m,\theta}\label{wN}\\
\left\|\|w(N)^{\perp,i}\|_{m,\theta/2}^{m}\right\|_{L^2(\mu)}
&\le& C_{m,\theta,\ep}2^{-\ep mN/2}.\label{wNperp}
\end{eqnarray}
Hence by (\ref{product of bar}),
\begin{equation*}
\left\|
\left\|\overline{w(N)^i}\cdot\overline{z(N+1)^j}\right\|_{m,\theta}^m
\right\|_{L^2(\mu)}
\le C_{m,\theta,\ep}2^{-\ep m(N+1)/2},
\end{equation*}
where $\left(\overline{w(N)^i}\cdot\overline{z(N+1)^j}\right)_{s,t}
=\left(w(N)^i_t-w(N)^i_s\right)\left(z(N+1)^j_t-z(N+1)^j_s\right)$.
Similarly,
\begin{equation*}
\left\|
\left\|\overline{w(N)^{\perp,i}}\cdot\overline{w(N)^{j}}\right\|_{m,\theta}
\right\|_{L^2(\mu)}
\le C_{m,\theta,\ep}2^{-\ep m(N+1)/2}.
\end{equation*}
We estimate $C(z(N+1)^{i},w(N)^j)_{s,t}$.
By the independence of $z(N+1)^{i}$ and $w(N)^j$,
\begin{eqnarray*}
E\left[C(z(N+1)^{i},w(N)^j)_{s,t}^m\right]&=&
C_mE\left[\left(\int_s^t(z(N+1)_u^i-z(N+1)_s^i)^2du\right)^{m/2}\right]
\nonumber\\
&\le&C_mE\left[\int_s^t\left(z(N+1)^i_u-z(N+1)^i_s\right)^mdu\right]
\left(\int_s^t1du\right)^{(m-2)/2}\nonumber\\
&\le&
C_m\min\left(|t-s|^m,2^{-(N+1)m/2}\right).
\end{eqnarray*}
Using this,
\begin{eqnarray*}
\left\|\|C(z(N+1)^i,w(N)^j)\|_{m,\theta}^m\right\|_{L^2(\mu)}
&\le&
\iint_{(s,t)\in
\Delta}\frac{E\left[C(z(N+1)^{i},w(N)^j)_{s,t}^{2m}\right]^{1/2}}
{|t-s|^{2+m\theta}}dsdt\nonumber\\
&\le&2^{-(N+1)m\ep/2}\iint_{(s,t)\in \Delta}
|t-s|^{m(1-\ep-\theta)-2}dsdt.
\end{eqnarray*}
Hence if $m(1-\theta)>1$, then we have
\begin{equation*}
\left\|\|C(z(N+1)^i,w(N)^j)\|_{m,\theta}^m\right\|_{L^2(\mu)}
\le C_{m,\theta,\ep}2^{-(N+1)m\ep/2}.
\end{equation*}
Similarly if $m(1-\theta)>1$,
\begin{eqnarray*}
\left\|\|C(w(N)^{\perp,i},w(N)^j)\|_{m,\theta}^m\right\|_{L^2(\mu)}
&\le& C_{m,\theta,\ep}2^{-Nm\ep/2},\\
\left\|\|C(z(N)^{i},z(N)^j)\|_{m,\theta}^m\right\|_{L^2(\mu)}
&\le& C_{m,\theta,\ep}2^{-Nm\ep/2}\qquad (i\ne j).
\end{eqnarray*}
When $i=j$, under the assumption that
$m(1-\theta)>2$,
\begin{eqnarray*}
\|\|C(z(N)^i,z(N)^i)\|_{m,\theta}^m\|_{L^2(\mu)}
&=&\left(\frac{1}{2}\right)^m
\|\|\overline{z(N)^i}\cdot\overline{z(N)^i}\|_{m,\theta}^m\|_{L^2(\mu)}\\
&\le&
\left(\frac{M_{m,\theta}}{2}\right)^m
\|\|z(N)^i\|_{m,\theta/2}^{2m}\|_{L^2(\mu)}.
\end{eqnarray*}
Let 
\begin{eqnarray*}
A_{N,i}&=&\left\{w~\Big |~\|z(N+1)^i\|_{m,\theta/2}>N^{-2}\right\},\\
B_{N,i,j}&=&
\left\{w~\Big |~\|C(w(N+1)^i,w(N+1)^j)-C(w(N)^i,w(N)^j)\|_{m,\theta}
>N^{-2}\right\},\\
C_{N,i,j}&=&
\left\{w~\Big |~
\|\overline{w(N)^{\perp,i}}\cdot \overline{w(N)^j}\|_{m,\theta}>N^{-2}
\right\}\\
D_{N,i,j}&=&\left\{w~\Big |~\|C(w(N)^{\perp,i},w(N)^j)\|_{m,\theta}
>N^{-2}
\right\}.
\end{eqnarray*}
Note that $\|z(N+1)^i\|_{m,\theta/2}^m$, 
$\|C(w(N+1)^i,w(N+1)^j)-C(w(N)^i,w(N)^j)\|_{m,\theta}^m$,
$\|\overline{w(N)^{\perp,i}}\cdot \overline{w(N)^j}\|_{m,\theta}^m$,
$\|C(w(N)^{\perp,i},w(N)^j)\|_{m,\theta}^m$,
are Wiener chaos of order at most $2m$.
Hence by the hypercontractivity of the Ornstein-Uhlenbeck semi-group,
their $L^2$-norms and the $(q,s)$-Sobolev norms are equivalent for
any
$q\ge 2, s>0$.
By Lemma~\ref{tchebycheff} and the above estimates, we obtain
\begin{eqnarray}
\max\left(C^{s}_q(A_{N,i}), C^s_q(C_{N,i,j}), C^s_q(D_{N,i,j})\right)&\le&
C_{s,q,m,\theta,\ep}N^{2m}2^{-\ep mN/2}.\label{ACD}
\end{eqnarray}
Since
\begin{eqnarray}
\lefteqn{C(w(N+1)^i,w(N+1)^j)-C(w(N)^i,w(N)^j)}\nonumber\\
& &=
\left(w(N)^i_t-w(N)^i_s\right)
\left(z(N+1)^j_t-z(N+1)^j_s\right)-
C(z(N+1)^j,w(N)^i)_{s,t}\nonumber\\
& &+C(z(N+1)^i,w(N)^j)_{s,t}+C(z(N+1)^i,z(N+1)^j)_{s,t},\label{1 step difference
of w(N)}
\end{eqnarray}
using the subadditivity of the capacity,
we have
\begin{equation}
C^s_q(B_{N,i,j})\le C_{s,q,m,\theta,\ep}N^{2m}
2^{-\ep mN/2}.\label{B}
\end{equation}
Here we note that $A_{N,i}, B_{N,i,j}, C_{N,i,j}, D_{N,i,j}$
depend on $(m,\theta)$ satisfying $m(1-\theta)>2$.
Let 
\begin{equation*}
E=\cup_{1\le i,j\le d, m,\theta\in {\mathbb Q}}\left\{
(\limsup_{N\to\infty}A_{N,i})
\cup(\limsup_{N\to\infty}B_{N,i,j})
\cup(\limsup_{N\to\infty}C_{N,i,j})
\cup(\limsup_{N\to\infty}D_{N,i,j})
\right\}.
\end{equation*}
By (\ref{ACD}) and (\ref{B}), $E$ is a slim set.
Since $E^c\subset \Omega$, $\Omega^c$ is a slim set.
The properties that $H\subset \Omega$ and $\Omega+H\subset \Omega$
follows from the estimates in Lemma~\ref{pointwise estimate}.
To complete the proof of Theorem~\ref{Omega}, we need
to show 
\begin{itemize}
\item[{\rm (a)}] the sequences of iterated integrals converge with respect to
$\|~\|_{H,\theta}$,
\item[{\rm (b)}] the limit is continuous with respect to $(s,t)\in \Delta$.
\end{itemize}
The item (a) follows from Lemma~\ref{Hoelder and Besov}
and the convergences in $L_{m,\theta}$.
The item (b) follows from (a).
Now we prove Theorem~\ref{quasi-cont-version}.
Let $E_{K,m,\theta}=\cap_{1\le i,j\le d}\left\{\cap_{N=K}^{\infty}
(A_{N,i}^c\cap B_{N,i,j}^c\cap C_{N,i,j}^c\cap D_{N,i,j}^c)
\right\}$.
Then $w(N)$, $C(w(N),w(N))$ converges
uniformly with respect to $\|~\|_{m,\theta/2}$
on $E_{K,m,\theta}$.
Therefore $C(w,w), w$ is continuous with respect to 
${\mathfrak T}$ on $E_{K,m,\theta}\cap \Omega$.
For any $(s,q)$ and $\ep>0$, 
we have $C^{s}_q(E_{K,m,\theta}^c)<\ep$ for sufficiently large $K$.
This completes the proof of Theorem~\ref{quasi-cont-version}.
\end{proof}

We fix a version of the solution of SDE
~(\ref{sde}) using Theorem~\ref{Omega}.
To this end, we introduce a distance function on
$\Omega$.

\begin{dfi}Let $(2/3)<\theta<\theta'<1$
and assume $m(1-\theta')>2$.
For $w,z\in \Omega$, let
\begin{eqnarray}
d_{\Omega}(w,z)&=&\max\left\{
\max_{i,j}\|C(w^i,w^j)-C(z^i,z^j)\|_{H,\theta},
\max_i\|w^i-z^i\|_{m,\theta'/2}
\right\}.
\end{eqnarray}
\end{dfi}

We note that $(\Omega,d_{\Omega})$ is a separable metric space.
For $h\in H$, let $X(t,a,h)$ be the solution to the following
ODE:
\begin{eqnarray*}
\dot{X}(t,a,h)&=&\left(L_{X(t,a,h)}\right)_{\ast}\dot{h}_t,\\
X(0,a,h)&=&a\in G.
\end{eqnarray*}
By the assumption that $\frac{2}{3}<\theta<1$,
the topology by the distance $d_{\Omega}$ is stronger than
the $p$-variation topology with $p>\frac{2}{\theta}$.
Hence by Theorem~\ref{Omega} and the universal limit theorem
~\cite{lyons98, lq, fv2}, 
for any $w\in \Omega$, $t\ge 0, a\in G$, the limit
\begin{equation}
\lim_{N\to\infty}X(t,a,w(N))\label{approximation}
\end{equation}
exists.
We denote the limit by
$X(t,a,w)$.
For this limit, we have the following.

\begin{pro}\label{solution of sde}
The measurable mapping $X : [0,\infty)\times G\times \Omega\to G$
satisfies the following.

\noindent
$(1)$~$X(t,a,w)$ is a version of the solution to the 
SDE~$(\ref{sde})$.

\noindent
$(2)$~
For any $a$, the mapping $w\mapsto X(\cdot,a,w)\in C([0,1]\to G)$ is 
continuous in the sense that
there exists an increasing function $F$ on
$\RR$ such that for all $w,z\in \Omega$,
$$
\sup_{0\le t\le 1}d(X(t,a,w),X(t,a,z))\le 
F(\max\{d_{\Omega}(0,w),d_{\Omega}(0,z)\})
d_{\Omega}(w,z).
$$
Moreover the mapping $w\mapsto X(\cdot,a,w)$ is $\infty$-quasi-continuous 
with respect to the supremum norm of $W^d$
for any
$a$.

\noindent
$(3)$~For all $t,a,w$, $X(t,a,w)=aX(t,e,w)$.
In particular, the mapping 
$a\mapsto X(t,a,w)$ is a $C^{\infty}$-diffeomorphism.

\noindent
$(4)$~For any $\phi \in H^1([0,1]\to G~|~\phi_0=e)$, 
it holds that
\begin{equation}
X(t,\phi_t,w)=X(t,e,w+\zeta(\phi,w)),
\end{equation}
where $\zeta(\phi,w)$ is the solution to 
\begin{eqnarray}
\dot{\zeta}(\phi,w)_t&=&
Ad\left(X(t,e,w)^{-1}\right)\left(\phi_t^{-1}\dot{\phi}_t\right)
\qquad\quad t>0\\
\zeta(\phi)_0&=&0.
\end{eqnarray}

\noindent
$(5)$~For $h\in H$,
let $Z(t,h,w)$ be the $H^1$-path on $G$ which satisfies the
ODE:
\begin{eqnarray}
Z(t,h,w)^{-1}\dot{Z}(t,h,w)&=&
Ad\left(X(t,e,w)\right)\dot{h}_t\qquad \quad t>0\\
Z(0,h,w)&=&e.
\end{eqnarray}
Then it holds that $X(t,Z(t,h,w),w)=X(t,e,w+h)$.

\noindent
$(6)$~For any $h\in H$
\begin{equation}
\zeta\left(Z(\cdot,h,w),w\right)=h.
\end{equation}
\end{pro}

\begin{proof}
Part (1) is a standard result in stochastic analysis.
Part (2) is a consequence of rough path analysis.
The claim that (3),(4),(5),(6) hold for almost all $w$
is also standard in stochastic analysis.
However, these identities hold for all $w\in \Omega$.
This follows from the fact:
\begin{itemize}
\item[(i)]~the claims (3),(4),(5),(6) hold for all $w\in H$,
\item[(ii)]~The Cameron-Martin subspace 
$H$ is a dense subset in $\Omega$ with respect to the topology defined by
$d_{\Omega}$,
\item[(iii)]~Part (2).
\end{itemize}
\end{proof}

The following 
will be used 
in the next section.

\begin{lem}\label{H-derivative}
Suppose that $m(1-\theta)>2$.
Let $(x,y)=(w(N)^i,w(N)^j), (w^i,w^j)$~for $i\ne j$ or
$(x,y)=(w(N)^i,w(N)^{\perp,j}), (w(N)^{\perp,i},w(N)^{j})$ for
any $i,j$.
Then the following estimates hold for almost all $w$.
\begin{eqnarray}
\|D^k\|x\|_{m,\theta/2}^m\|_H&\le&C_{m,\theta,k}
\|x\|_{m,\theta/2}^{m-k}\qquad \mbox
{for all $1\le k\le m$},\label{Dx}\\
\|D^k\|C(x,y)\|_{m,\theta}^m\|_H&\le&
C_{m,\theta,k}
\sum_{l=0}^{\left[\frac{k}{2}\right]}
\left(\|x\|_{m,\theta/2}^{2}+\|y\|_{m,\theta/2}^{2}\right)^{(k-2l)/2}
\|C(x,y)\|_{m,\theta}^{m+l-k}.\nonumber\\
& &\mbox{for all $1\le k\le 2m$}.\label{DCxy}
\end{eqnarray}
\end{lem}

\begin{proof}
We consider the case where $k=1$ and
$x=w(N)^i$ in (\ref{Dx}).
The proof of other cases are similar to it.
We have
\begin{eqnarray*}
|D_h\|x\|_{m,\theta/2}^m|&=&
\left|
m\int_0^1\int_0^t\frac{(h(N)^i_t-h(N)^i_s)(w(N)^i_t-w(N)^i_s)^{m-1}}
{(t-s)^{2+m\theta/2}}dsdt\right|
\nonumber\\
&=&
m\|h(N)^i\|_{m,\theta/2}\|w(N)^i\|_{m,\theta/2}^{m-1}\nonumber\\
&\le&
C_{m,\theta}\|h\|_H\|w(N)^i\|_{m,\theta/2}^{m-1}
\end{eqnarray*}
which implies (\ref{Dx}).
We prove (\ref{DCxy}) in the case where $k=1$.
Let $(x,y)=(w(N)^i,w(N)^j)$~$(i\ne j)$.
Then 
\begin{eqnarray*}
\lefteqn{|D_h\|C(x,y)\|_{m,\theta}^m|}\nonumber\\
& &=
m\left|\int_0^1\int_0^t
\frac{\left(C(h(N)^i,w(N)^j)_{s,t}+
C(w(N)^i,h(N)^j)_{s,t}
\right)\left(
C(x,y)_{s,t}^{m-1}
\right)}{(t-s)^{2+m\theta}}dsdt\right|\nonumber\\
& &\le m\left(\|C(h(N)^i,w(N)^j)\|_{m,\theta}+
\|C(w(N)^i,h(N)^j)\|_{m,\theta}\right)
\|C(x,y)\|_{m,\theta}^{m-1}\nonumber\\
& &\le C_{m,\theta}\left(\|w(N)^i\|_{m,\theta/2}\|h(N)^j\|_H+
\|w(N)^j\|_{m,\theta/2}\|h(N)^i\|_H\right)
\|C(x,y)\|_{m,\theta}^{m-1},
\end{eqnarray*}
where we have applied Lemma~\ref{pointwise estimate}~(2)
in the case where $m(1-\theta)>2$.
This implies (\ref{DCxy}).
We can check the other cases in similar ways.
\end{proof}

\section{A Poincar\'e's lemma on a certain domain in a Wiener space}

The reader may find the following statement in 
Remark~3.2 in \cite{aida-wpoincare}.
We apply this lemma to Dirichlet forms 
on open subsets in Euclidean spaces.
For the sake of completeness, we give the proof.

\begin{lem}\label{general PI}
Let $(X,\mu)$ and $(Y,\nu)$ be probability spaces.
Let 
$dm=d\mu\otimes d\nu$.
Assume that we are given Dirichlet forms 
$({\cal E}_X, {\rm D}({\cal E}_X))$,
$({\cal E}_Y, {\rm D}({\cal E}_Y))$ on 
$L^2(X,\mu)$ and $L^2(Y,\nu)$.
Moreover we assume that
${\cal E}_X$, ${\cal E}_Y$ has the square field operators
$\Gamma_X$ and $\Gamma_Y$ respectively.
Let $U$ be a measurable subset of $X\times Y$
with $m(U)>0$.
Let $U_x=\{y\in Y~|~(x,y)\in U\}$
and $U^y=\{x\in X~|~(x,y)\in U\}$.
Let $A=\{x\in X~|~\nu(U_x)>0\}$ and
$B=\{y\in Y~|~\mu(U^y)>0\}$.
We assume that

\noindent
$(1)$~There exists $\tilde{A}\subset A$ such that
$\mu(A\setminus \tilde{A})=0$ and
$\delta=\inf_{x,x'\in \tilde{A}}\nu\left(U_x\cap U_{x'}\right)
>0$. Moreover there exists a positive number $C_2$
such that for any $x\in \tilde{A}$ and $g\in {\rm D}({\cal E}_Y)$,
\begin{equation}
\Var(g; U_x)\le
\frac{C_2}{\nu(U_x)}\int_{U_x}\Gamma_Y g(y)d\nu(y).
\end{equation}
Here $\Var(g; U_x)$ denotes the variance of $g$ with respect
to the probability measure $d\nu|_{U_x}/\nu(U_x)$.
In the statement below too, we use $\Var$ in this sense.

\noindent
$(2)$~There exists $\tilde{B}\subset B$ such that
$\nu(B\setminus \tilde{B})=0$ and
there exists a positive number $C_1$
such that for any $y\in \tilde{B}$ and $h\in {\rm D}({\cal E}_X)$
\begin{equation}
\Var(h; U^y)
\le
\frac{C_1}{\mu(U^y)}\int_{U^y}\Gamma_X h(x)d\mu(x).
\end{equation}
Let us denote $z=(x,y)\in X\times Y$.
Then we have for $f=f(z)=f(x,y)$,
\begin{equation}
\Var(f; U)
\le
\frac{3}{\delta m(U)}\int_{U}
\left(\frac{C_1}{m(U)}\Gamma_Xf(x,y)+
C_2\Gamma_Yf(x,y)
\right)
dm(z).
\end{equation}
\end{lem}

\begin{proof}
Let $x,x'\in \tilde{A}$,
$y\in U_x, y'\in U_{x'}, z\in U_x\cap U_{x'}$.
Noting that
\begin{eqnarray}
\lefteqn{\left(f(x,y)-f(x',y')\right)^2}\nonumber\\
& &\le
3\Bigl\{
(f(x,y)-f(x,z))^2+
(f(x,z)-f(x',z))^2+
(f(x',z)-f(x',y'))^2
\Bigr\},
\end{eqnarray}
and $\nu\left(U_x\cap U_{x'}\right)>\delta$,
we have
\begin{eqnarray}
\left(f(x,y)-f(x',y')\right)^2&\le&
\frac{3}{\delta}\int_{U_x\cap U_{x'}}
(f(x,y)-f(x,z))^2d\nu(z)\nonumber\\
& &+\frac{3}{\delta}\int_{U_x\cap U_{x'}}
\left(f(x,z)-f(x',z)\right)^2d\nu(z)\nonumber\\
& &+\frac{3}{\delta}\int_{U_x\cap U_{x'}}
\left(f(x',z)-f(x',y')\right)^2d\nu(z)\nonumber\\
&=&I_1+I_2+I_3.
\end{eqnarray}
We estimate $I_i$.
\begin{eqnarray}
\lefteqn{\int_{x,x'\in\tilde{A},y\in U_x,y'\in U_{x'}}
I_1d\mu(x)d\mu(x')d\nu(y)d\nu(y')}\nonumber\\
& &\le\frac{3}{\delta}
\int_{x\in \tilde{A}, y,z\in
U_x}\left(f(x,y)-f(x,z)\right)^2d\nu(y)d\nu(z)
d\mu(x)m(U)\nonumber\\
& &\le \frac{3C_2m(U)}{\delta}\int_{x\in \tilde{A}, y\in U_x}
2\nu(U_x)\Gamma_Yf(x,y)d\nu(y)d\mu(x).
\end{eqnarray}
\begin{eqnarray}
\lefteqn{\int_{x,x'\in \tilde{A}, y\in U_x,y'\in U_{x'}}
I_2d\mu(x)d\mu(x')d\nu(y)d\nu(y')
}\nonumber\\
& &=
\frac{3}{\delta}\int_{x,x'\in \tilde{A}}
\left(\nu(U_x)\nu(U_{x'})\int_{z\in U_x\cap U_{x'}}
\left(f(x,z)-f(x',z)\right)^2d\nu(z)\right)d\mu(x)d\mu(x')\nonumber\\
& &\le \frac{3}{\delta}
\int_{x,x'\in \tilde{A}\cap U^z, z\in Y}
\left\{\left(f(x,z)-f(x',z)\right)^2d\mu(x)d\mu(x')\right\}
d\nu(z)\nonumber\\
& &=\frac{3}{\delta}\int_{x,x'\in U^z, z\in \tilde{B}}
\left\{\left(f(x,z)-f(x',z)\right)^2d\mu(x)d\mu(x')\right\}
d\nu(z)\nonumber\\
& &\le\frac{3}{\delta}
\int_{\tilde{B}}d\mu(z)2C_1\mu(U^z)\int_{U^z}
\Gamma_Xf(x,z)d\mu(x).
\end{eqnarray}
As to $I_3$, we have the same estimate for $I_1$:
\begin{eqnarray}
\lefteqn{\int_{x,x'\in\tilde{A},y\in U_x,y'\in U_{x'}}
I_3d\mu(x)d\mu(x')d\nu(y)d\nu(y')}\nonumber\\
& &\le
\frac{3C_2m(U)}{\delta}\int_{x\in \tilde{A}, y\in U_x}
2\nu(U_x)\Gamma_Yf(x,y)d\nu(y)d\mu(x).
\end{eqnarray}
Since
\begin{eqnarray}
\lefteqn{\int_{x,x'\in\tilde{A},y\in U_x,y'\in U_{x'}}
(f(x,y)-f(x',y'))^2d\mu(x)d\mu(x')d\nu(y)d\nu(y')}\nonumber\\
& &=2m(U)\int_U\left(f(z)-\frac{1}{m(U)}\int_Uf(z)dm(z)\right)^2
dm(z),
\end{eqnarray}
the above estimates complete the proof.
\end{proof}

To apply the lemma above to $U_{r,\varphi}$ which we will define
later, we need uniform positivity of probabilities 
of intersections of subsets of a Wiener space
(Lemma~\ref{PI on convex domain}~(1)).
First we begin by the following.

\begin{lem}\label{positivity}
Let us consider the case where $d=1$.
That is, $w$ is a real valued continuous path.
Let $0<\theta<\theta'<1$ and $m(1-\theta)>2$.
Let $z^1,\ldots,z^l\in W_{m,\theta/2}(\RR)$ and define
\begin{eqnarray*}
\lefteqn{U_N(z^1,\ldots,z^l ; \ep)}\nonumber\\
& &=
\left\{w\in \Omega~ \Big |~ \max_{1\le i\le l}
\left\{\|w(N)\|_{m,\theta'/2}, \|C(w(N),z^i)\|_{m,\theta}, 
\|C(z^i,w(N))
\|_{m,\theta}\right\}<\ep\right\},
\end{eqnarray*}
where $\ep$ is a positive number.
Then for fixed $l$, $r>0$ and $\ep>0$, we have
\begin{equation}
\inf\left\{\mu\left(U_N(z^1,\ldots,z^l ; \ep)\right)~\Big |~
\max_{1\le i\le l}\|z^i\|_{m,\theta'/2}\le r, N\in {\mathbb N}
\right\}>0.\label{small-ball-N}
\end{equation}
\end{lem}

For later use, we denote the infimum in (\ref{small-ball-N})
by $C(l,\ep,r,m,\theta,\theta')$.

To prove the lemma above, we need a lemma.
Let $x$ be a real-valued continuous function 
on $[0,1]$ and $w$ be the $1$-dimensional Brownian motion. 
Then the stochastic integral (Wiener integral)
$B(x,w)$
is defined for almost all $w$ as continuous functions of
$(s,t)\in \Delta$:
\begin{eqnarray}
B(x,w)_{s,t}&=&\int_s^t(x_u-x_s)dw_u.
\end{eqnarray}
Also we set
$B(w,x)_{s,t}=\left(\bar{x}\cdot\bar{w}\right)_{s,t}-B(x,w)_{s,t}$.
As for the notation $\left(\bar{x}\cdot\bar{w}\right)_{s,t}$,
see Lemma~\ref{pointwise estimate}~(1).
For these stochastic integrals, we have the following estimates.

\begin{lem}\label{L^m estimate}
Assume $m(1-\theta)>2$.
Stochastic integrals $B(x,w), B(w,x)$ take values in $W_{m,\theta/2}$
for almost all $w$ and
\begin{eqnarray}
E\left[\|B(x,w)\|_{m,\theta}^m+
\|B(w,x)\|_{m,\theta}^m\right]&\le&
 C_{m,\theta}\|x\|_{m,\theta/2}^m.\label{estimates for stochastic integrals}
\end{eqnarray}
Also we have
\begin{equation}
\lim_{N\to\infty}
E\left[\|C(x,w(N))-B(x,w)\|_{m,\theta}^m+
\|C(w(N),x)-B(w,x)\|_{m,\theta}^m\right]=0.\label{estimates for remainders}
\end{equation}
\end{lem}

\begin{proof}
We have
\begin{eqnarray*}
E\left[\int_0^1\int_0^t\frac{B(x,w)_{s,t}^m}{
|t-s|^{2+m\theta}}dsdt\right]
&=&
C_m\int_0^1\int_0^t\frac{\left(\int_s^t(x_u-x_s)^2du\right)^{m/2}}
{(t-s)^{2+m\theta}}dsdt\nonumber\\
&\le&
C_m\int_0^1\int_0^t\frac{(t-s)^{\frac{m}{2}-1}\int_s^t(x_u-x_s)^{m}du}
{(t-s)^{2+m\theta}}dsdt\nonumber\\
&\le&
C_m\int_0^1\int_0^t\frac{\int_s^t(x_u-x_s)^{m}du}{(t-s)^{2+m\theta/2}}dsdt
\nonumber\\
&\le&C_m\|x\|_{m,\theta/2}^m.
\end{eqnarray*}
Noting that
$B(w,x)_{s,t}=(w_t-w_s)(x_t-x_s)-B(x,w)_{s,t}$ and
$$
E\left[\int_0^1\int_0^t
\frac{(w_t-w_s)^m}{(t-s)^{2+m\theta/2}}dsdt
\right]<\infty
$$
we complete the proof of (\ref{estimates for stochastic integrals}).
We prove (\ref{estimates for remainders}).
We have
\begin{eqnarray}
\lefteqn{\left\|\|C(x,w(N))-B(x,w)\|_{m,\theta}^m\right\|_{L^2(\mu)}}
\nonumber\\
& &\le
\iint_{\{(s,t)\in \Delta\}}
\frac{E\left[\left(C(x,w(N))_{s,t}-B(x,w)_{s,t}\right)^{2m}\right]^{1/2}}
{(t-s)^{2+m\theta}}dsdt.\label{CxwN}
\end{eqnarray}
Note that
\begin{eqnarray*}
E\left[\left(C(x,w(N))_{s,t}-B(x,w)_{s,t}\right)^{2m}\right]^{1/2}
&\le&C_m\psi_N(s,t)
\end{eqnarray*}
where $\psi_N(s,t)=E
\left[\left(C(x,w(N))_{s,t}-B(x,w)_{s,t}\right)^{2}\right]^{m/2}$.
Also 
\begin{equation*}
\psi_N(s,t)\le E\left[B(x,w)_{s,t}^2\right]^{m/2}=:\psi(s,t).
\end{equation*}
This follows from that $w-w(N)$ and $w(N)$ are independent.
It holds that $\lim_{N\to\infty}\psi_N(s,t)=0$ for all
$(s,t)$ and
$\iint_{\Delta}\frac{\psi(s,t)}{(t-s)^{2+m\theta}}dsdt<\infty$.
Hence the Lebesgue dominated convergence theorem implies 
that the quantity on the right-hand side of (\ref{CxwN}) converges to $0$.
For the other term,
it suffices to note that
$C(w(N),x)-B(w,x)=B(x,w)-C(x,w(N))+
\bar{x}\cdot\bar{w(N)}-\bar{x}\cdot\bar{w}$
and $\lim_{N\to\infty}E[\|w(N)-w\|_{m,\theta/2}^m]=0$.
\end{proof}

\begin{proof}[Proof of Lemma~$\ref{positivity}$]
First we prove that for any $N$,
\begin{equation}
\ep_N:=\inf\left\{\mu\left(U_N(z^1,\ldots,z^l ; \ep)\right)~\Big |~
\max_{1\le i\le l}\|z^i\|_{m,\theta'/2}\le r
\right\}>0.\label{lowerboundN}
\end{equation}
Note that for any $z^1,\ldots,z^l\in W_{m,\theta/2}(\RR)$,
\begin{equation}
\mu\left(U_N(z^1,\ldots,z^l ; \ep)\right)>0.
\end{equation}
If (\ref{lowerboundN}) does not hold, then we can find a sequence
$\{z^{i,n}\}$ such that
$\sup_{i,n}\|z^{i,n}\|_{m,\theta'/2}\le r$
$\lim_{n\to\infty}\mu(U_N(z^{1,n},\ldots,z^{l,n} ; \ep))=0$.
Since the embedding $W_{m,\theta'/2}(\RR)\subset W_{m,\theta/2}(\RR)$
is compact, there exists a subsequence $\{z^{i,n(k)}\}$ and
$\{y^i\}\subset W_{m,\theta/2}(\RR)$ such that
$\lim_{k\to\infty}\|z^{i,n(k)}-y^i\|_{m,\theta/2}=0$.
By Lemma~\ref{L^m estimate} and 
$E[\|C(x,w(N))\|_{m,\theta}^m]\le E[\|B(x,w)\|_{m,\theta}^m]$
and so on,
$$
\lim_{k\to\infty}E\left[
\|C(w(N),z^{i,n(k)})-C(w(N),y^i)\|_{m,\theta}+
\|C(z^{i,n(k)},w(N))-C(y^i,w(N))\|_{m,\theta}\right]=0.
$$
This implies that $\mu\left(U_N(y^1,\ldots,y^l ; \ep/2)\right)=0$
which is a contradiction.
Next we prove that $\liminf_{N\to\infty}\ep_N>0$.
The random variable $(w, B(w,z^i),B(z^i,w))$ defines a Gaussian measure
with mean $0$ on the separable Banach space
$$
W_{m,\theta'/2}(\RR)\times \prod_{i=1}^{2l}
L_{m,\theta}(\Delta\to \RR).
$$
Therefore every ball of positive radius has positive measure.
See \cite{b}.
Thus we obtain for any $\ep>0$ and $\{z^i\}_{i=1}^l\subset 
W_{m,\theta/2}(\RR)$,
\begin{equation}
\mu(U(z^1,\ldots,z^l ; \ep))
>0,\label{small-ball-infty}
\end{equation}
where
\begin{equation}
U(z^1,\ldots,z^l ; \ep)=
\left\{w\in W_{m,\theta'/2}(\RR)~\Big |~ \max_{1\le i\le l}\left\{
\|w\|_{m,\theta'/2}, \|B(w,z^i)\|_{m,\theta}, \|B(z^i,w)
\|_{m,\theta}\right\}< \ep\right\}.
\end{equation}

Now suppose that there exist 
$\{z^{i,N}\}\subset W_{m,\theta'/2}({\mathbb R})$
with $\sup_{i,N}\|z^{i,N}\|_{m,\theta'/2}<r$
and
$$
\lim_{N\to\infty}\mu\left(U_N(z^{1,N},\ldots,z^{l,N} ; \ep)\right)=0.
$$
We may assume that there exists $y^i\in W_{m,\theta/2}(\RR)$ such that
$\lim_{N\to\infty}\|z^{i,N}-y^i\|_{m,\theta/2}=0$.
We have
\begin{eqnarray}
C(w(N),z^{i,N})&=&C(w(N),z^{i,N}-y^i)+B(w,y^i)-
\left(B(w,y^i)-C(w(N),y^i)\right).
\end{eqnarray}
Also the $\|~\|_{m,\theta}$ norms of
$C(w(N), z^{i,N}-y^i)$ and $B(w,y^i)-C(w(N),y^i)$ converge to
$0$ in probability by Lemma~\ref{L^m estimate}.
This shows $\mu(U(y^1,\ldots,y^l;\ep/2))=0$
which is a contradiction and we have proved that $\inf_N\ep_N>0$.
\end{proof}

The following lemma will be applied to
the set $U_k(\xi^{k+1},\ldots,\xi^d,\eta)
_{(\xi^1,\ldots,\xi^{k-1})}$
which is defined in (\ref{convex-section}).

\begin{lem}\label{PI on convex domain}
Let $d=1$.
That is, we consider the case where $w\in \Omega$ and $\xi\in \Omega_N$
are real-valued functions on $[0,1]$.
Let $0<\theta<\theta'<1$ and $m(1-\theta)>2$.
Let 
$x\in W_{m,\theta'/2}(\RR)$,
$y^1,\ldots,y^{2l}\in W_{m,\theta'/2}(\RR)$
and $z^1,\ldots, z^{2l}\in W_{m,\theta}(\Delta\to \RR)$.
Let $r$ be a positive number and $0<\delta<1$.
Suppose that 
$\|x\|_{m,\theta'/2}<\delta r$ and
$\max_{1\le i\le 2l}\|z^{i}\|_{m,\theta}<\delta r$.
Let us consider a bounded open subset of $\Omega_N$,
\begin{eqnarray}
\lefteqn{U_{N}(\{y^i\}_{i=1}^{2l},\{z^i\}_{i=1}^{2l},x)}\nonumber\\
& &=
\Bigl\{
\xi\in \Omega_N~\Big|~\|\xi+x\|_{m,\theta'/2}<r,
\max_{1\le i\le l}\|C(\xi,y^i)+z^i\|_{m,\theta}<r,\nonumber\\
& &\qquad\qquad \qquad \max_{1\le i\le l}\|C(y^{i+l},\xi)+z^{i+l}\|_{m,\theta}<r
\Bigr\}.
\end{eqnarray}

\noindent
$(1)$~
It holds that for any $C>0$
\begin{eqnarray}
\inf\left\{
\mu(U_N(\{y^i\}_{i=1}^{2l}, \{z^i\}_{i=1}^{2l},x))~\Big|~
\max_{1\le i\le 2l}\|y^i\|_{m,\theta'/2}\le C, N\in {\mathbb N}
\right\}>0.
\end{eqnarray}

\noindent
$(2)$~
Let $W^1(U_N(\{y^i\},\{z^i\},x),\mu_N)$ be the Sobolev space
which consists of $L^2$-functions with respect to
$\mu_N$ on
$U_N(\{y^i\},\{z^i\},x)$ whose weak derivatives are
in $L^2(\mu_N)$.
This set coincides with $W^1(U_N(\{y^i\},\{z^i\},x))$
which is usual Sobolev spaces whose derivatives are in $L^2$
with respect to the Lebesgue measure.
Moreover there exists a bounded linear operator 
$(\mbox{extension operator})$
$T : W^1(U_N(\{y^i\},\{z^i\},x),\mu_N)\to
W^1(\Omega_N,\mu_N)$ such that 
$Tf|_{U_N(\{y^i\},\{z^i\},x)}=f$.

\noindent
$(3)$~
It holds that for any $f\in W^1(U_N(\{y^i\},\{z^i\},x),\mu_N)$,
\begin{eqnarray}
\Var(f ; U_N(\{y^i\},\{z^i\},x))&\le&
\int_{U_N(\{y^i\},\{z^i\},x)}|Df(\xi)|_H^2d
\mu_{N,U_N(\{y^i\},\{z^i\},x)}(\xi).
\end{eqnarray}
where $\mu_{N,U_N(\{y^i\},\{z^i\},x)}$
is the normalized probability measure of
$\mu_N$ on $U_N(\{y^i\},\{z^i\},x)$.
\end{lem}

\begin{proof}
Part (1) follows from Lemma~\ref{positivity}.
while (2) follows from the fact that
$U_N(\{y^i\},\{z^i\},x)$ is a bounded convex domain of
$\Omega_N$. 
Then Part (3) follows from the result in (2) and
the Poincar\'e inequality on a convex domain
in a Euclidean space with a Gaussian measure
~(\cite{fu}).
\end{proof}

From now on, 
we fix parameters $m,\theta,\theta'$ as follows.

\begin{assumption}
Let us fix $m,\theta,\theta'$ such that
$m(1-\theta')>4$ and $2/3<\theta<\theta'<1$.
\end{assumption}

Let $\varphi=\varphi_t
=(\varphi^1_t,\ldots,\varphi^d_t)$~$(0\le t\le 1)$
be an element of $H$ and define
\begin{eqnarray}
\lefteqn{U_{r,\varphi}}\nonumber\\
& &=
\Bigl\{
w\in \Omega~\Big |~
\max_{1\le i\le d}\|w^i\|_{m,\theta'/2}< r,
\max_{1\le j<k\le d}\|C(w^j,w^k)\|_{m,\theta}< r,
\max_{1\le i\le j\le d}\|C(\varphi^i,w^j)\|_{m,\theta}<r,\nonumber\\
& &\qquad\sup_{1\le i\le j\le d}\|C(w^i,\varphi^j)\|_{m,\theta}<r
\Bigr\},\label{nbd of 0}
\end{eqnarray}
and
\begin{eqnarray}
\lefteqn{U_r(\varphi)}\nonumber\\
& &=
\Bigl\{
w\in \Omega~\Big |~
\max_{1\le i\le d}\|w^i-\varphi^i\|_{m,\theta'/2}< r,
\max_{1\le j<k\le d}\|C(w^j-\varphi^j,w^k-\varphi^k)
\|_{m,\theta}< r,
\nonumber\\
& &\qquad \max_{1\le i\le j\le d}
\|C(\varphi^i,w^j-\varphi^j)\|_{m,\theta}<r,
\max_{1\le i\le j\le d}
\|C(w^i-\varphi^i,\varphi^j)\|_{m,\theta}<r
\Bigr\}.\label{nbd of varphi}
\end{eqnarray}
Although these sets are different from the metric ball
in the metric space $(\Omega,d_{\Omega})$,
these play a similar kind of
role of the balls in normed linear spaces.
Note that we have the following relation:
\begin{eqnarray}
U_{r}(\varphi)
=\left\{w+\varphi~|~w\in U_{r,\varphi}\right\}.
\end{eqnarray}
The strict positivity of the measures of these subsets
for any $r>0$ and $\varphi\in H$
can be proved by the argument similar to
the proof of Lemma~2.6
in \cite{aida-wpoincare}.
See \cite{lqz} also.

Now we state our Poincar\'e's lemmas.

\begin{thm}\label{a vanishing theorem 0}
Let $\beta\in {\mathbb D}^{\infty,q}(W^d,H^{\ast})\cap L^2(W^d,H^{\ast})$,
where $q>1$.
Suppose that $d\beta=0$ on $U_{r,\varphi}$.
Then for any $r'<r$, there exists $g\in {\mathbb D}^{\infty,q}(W^d,\RR)
\cap {\mathbb D}^{1,2}(W^d,\RR)$ such that
$dg=\beta$ on $U_{r',\varphi}$.
\end{thm}

\begin{thm}\label{a vanishing theorem}
Let $\beta\in {\mathbb D}^{\infty,q}(W^d,H^{\ast})\cap L^2(W^d,H^{\ast})$,
where $q>1$.
We assume that the first derivative of $\varphi$ is 
a bounded variation function.
Suppose that $d\beta=0$ on $U_{r}(\varphi)$.
Then for any $r'<r$, there exists $g\in {\mathbb D}^{\infty,q}(W^d,\RR)
\cap {\mathbb D}^{1,2}(W^d,\RR)$ such that
$dg=\beta$ on $U_{r'}(\varphi)$.
\end{thm}

First we prove Theorem~\ref{a vanishing theorem}
using Theorem~\ref{a vanishing theorem 0}.
After that, we will prove Theorem~\ref{a vanishing theorem 0}.

\begin{proof}
[Proof of Theorem~$\ref{a vanishing theorem}$]
Let $T_{\varphi}w=w+\varphi$.
Then $U_r(\varphi)=\{T_{\varphi}w~|~w\in U_{r,\varphi}\}$.
For a measurable function $u$ on $W^d$, define $T_{\varphi}^{\ast}u(w)
=u(w+\varphi)$.
Let $\chi_R$ be a smooth function on $\RR$ such that
$\chi_R(x)=1$ for $|x|\le R$ and $\chi_R(x)=0$ and
$|x|\ge R+1$.
Let $\hat{\chi}_R(w)=\chi(\|w\|_{m,\theta'/2}^m)$.
Note that $D^l\hat{\chi}_R(w)$ is a bounded function for all $l$.
This follows from Lemma~\ref{H-derivative}.
For any $q>1, k\in {\mathbb N}\cup \{0\}$,
there exist positive constants $C_1,C_2$~$(C_1<C_2)$ such that
for any
$u\in {\mathbb D}^{k,q}(W^d)$ 
$$
C_1\|u\|_{k,q}\le 
\|(T_{\varphi}^{\ast}u) \hat{\chi}_R\|_{k,q}
\le C_2\|u\|_{k,q}.
$$
This can be checked by using the Cameron-Martin formula and
the fact that the stochastic integral 
$\int_0^1(\varphi'(t),dw(t))$ is actually a Riemann-Stieltjes integral
and bounded on $\{w\in \Omega~|~\|w\|_{m,\theta/2}\le R+1\}$.
The same estimates hold for $1$-forms.
Let $\beta$ be the $1$-form which satisfies the assumptions of
the theorem.
Let $R$ be a sufficiently large number and set
$\bar{\beta}=(T_{\varphi}^{\ast}\beta) \hat{\chi}_R$.
Then $\bar{\beta}\in {\mathbb D}^{\infty,q}(W^d,H^{\ast})\cap
L^2(W^d,H^{\ast})$ and $d\bar{\beta}=0$ on $U_{r,\varphi}$.
Therefore by Theorem~\ref{a vanishing theorem 0}, there exists 
$\bar{g}\in {\mathbb D}^{\infty,q}(W^d,H^{\ast})\cap {\mathbb
D}^{1,2}(W^d,H^{\ast})$ such that
$d\bar{g}=\bar{\beta}$ on $U_{r',\varphi}$.
Define $g=\left(T_{-\varphi}^{\ast}\bar{g}\right)\hat{\chi}_{R'}$, where
$R'$ is also a sufficiently large positive number.
Then $g$ satisfies the desired properties.
\end{proof}

To prove Theorem~\ref{a vanishing theorem 0}, we need 
some homotopy arguments on finite dimensional space.
Let $U$ be a bounded open subset of $\RR^{n+m}$.
Let us write $z=(x,y)\in \RR^{n+m}$, where
$x\in \RR^n$ and $y\in \RR^m$.
Let $A$ be the image of the projection of
$U$ with respect to the first variable $x$.
Clearly, $A$ is also an open subset.
For $x\in A$, set $U_x=\{y\in \RR^m~|~(x,y)\in U\}$
which is also an open subset.
Using the notation above, we prepare the following.
The proof of this result is easy and we omit it.

\begin{lem}\label{retraction formula}
Suppose that $U_x$ is a convex set and contains $0$.
Let $\alpha$ be a $C^{\infty}$ $1$-form on $U$.
We write
\begin{equation}
\alpha(z)=\sum_{i=1}^n\beta_i(x,y)dx^i+\sum_{j=1}^m\gamma_j(x,y)dy^j.
\end{equation}
Let $\pi : U\to A$ be the projection and define
$s : A\to U$ by
$s(x)=(x,0)\in U$ for $x\in A$.
Let
\begin{equation}
\left(K\alpha\right)(z)=\int_0^1\sum_{j=1}^m\gamma_j(x,ty)y^jdt.
\end{equation}
If $d\alpha=0$ on $U$,
then it holds that $s^{\ast}\alpha$ is a closed form
on $A$ and
\begin{equation}
\alpha=\pi^{\ast}s^{\ast}\alpha+dK\alpha.
\end{equation}
\end{lem}

Needless to say, if $H^1(A,{\mathbb R})=0$, then
there exists a smooth function $g$ on $A$ such that
$dg=s^{\ast}\alpha$.
Therefore we have
$\alpha=d\left(\pi^{\ast}g+K\alpha\right)$.
We use this in the proof of Theorem~\ref{a vanishing theorem 0}.

\begin{proof}[Proof of Theorem~$\ref{a vanishing theorem 0}$]
Let $N\in {\mathbb N}$ and set
\begin{eqnarray}
R_N&=&
\Bigl\{\eta=(\eta^1,\ldots,
\eta^{d})~\in \Omega_N^{\perp}~ \Big |~
\max_{1\le i\le d}\|\eta^{i}\|_{m,\theta'/2}<r/4,\nonumber\\
& &\max_{1\le i<j\le d}\|C(\eta^{i},\eta^{j})\|_{m,\theta}
<r/4,
\max_{1\le i\le j\le d}\|C(\varphi^i,\eta^{j})
\|_{m,\theta}<r/4,\nonumber\\
& &\qquad\max_{1\le i\le j\le d}\|C(\eta^{i},
\varphi^j)\|_{m,\theta}<r/4
\Bigr\}.
\end{eqnarray}
For $\eta\in \Omega_N^{\perp}$, define
\begin{eqnarray}
U_{r,\varphi}(\eta)&=&
\Bigl\{
\xi=(\xi^1,\ldots,\xi^d)\in \Omega_N~\Big |~
\xi+\eta\in U_{r,\varphi},
\max_{1\le i<j\le d}\|C(\xi^i,\eta^{j})\|_{m,\theta}
<r/4,\nonumber\\
& &~
\max_{1\le i<j\le d}\|C(\eta^{i},\xi^j)\|_{m,\theta}
<r/4
\Bigr\}.
\end{eqnarray}
This set can be identified with
a bounded open subset of the Euclidean space
of dimension $2^N d$.
Using this, we define an approximate set of
$U_{r,\varphi}$ as follows.
\begin{eqnarray}
U_{r,\varphi,N}&=&
\left\{w\in \Omega~|~w(N)\in U_{r,\varphi}(w(N)^{\perp}),
w(N)^{\perp}\in R_N
\right\}.
\end{eqnarray}
Since $\Omega$ is isomorphic to the product space
$\Omega_N\times \Omega_N^{\perp}$,
$U_{r,\varphi,N}$ is thought as a subset of this product 
space.
Thus any function $g$ on $U_{r,\varphi,N}$ can be identified with
a function of $(\xi,\eta)$ where
$\xi\in U_{r,\varphi}(\eta), \eta\in R_N$.

Using Lemma~\ref{general PI} and Lemma~\ref{positivity} and an
induction, 
we prove the following Claims.

\medskip

\noindent
{\bf Claim 1}~Let $\eta\in R_N$.
Poincar\'e's inequality holds on
$U_{r,\varphi}(\eta)$ in the following form:
\begin{eqnarray}
\Var(g; U_{r,\varphi}(\eta))&\le&
C\int_{U_{r,\varphi}(\eta)}
|Dg(\xi)|_H^2d\mu_{N,U_{r,\varphi}(\eta)}(\xi),
\label{PI on DN}
\end{eqnarray}
where $C$ is a positive constant which depends only on
$r,d,\varphi,m,\theta,\theta'$ and
$\mu_{N,U_{r,\varphi}(\eta)}$
is a normalized probability measure on $U_{r,\varphi}(\eta)$.

\medskip

\noindent
{\bf Claim 2}~There exists a measurable function $g_N$
on $U_{r,\varphi,N}$ such that
for $\mu_N^{\perp}$-almost all  $\eta\in R_N$,
the function $\xi\in U_{r,\varphi}(\eta)\to
g_N(\xi,\eta)$
is a $C^{\infty}$ function with
\begin{eqnarray}
\sup_{\xi\in U_{r,\varphi}(\eta)}|g_N(\xi,\eta)|&<&\infty\\
\int_{U_{r,\varphi}(\eta)}g_N(\xi,\eta)d\mu_N(\xi)&=&0
\end{eqnarray}
and
$d_Ng_N=\beta_N$ holds on $U_{r,\varphi,N}$.
Here $d_Ng_N$ is the exterior differential of $g_N$ with respect to the
variable $\xi$ and $\beta_N=P_N\beta$ which is the projection
of $\beta$ onto $(\Omega_N\cap H)^{\ast}$.

\medskip

To prove these claims, we introduce the following sets.
First, we fix $\eta\in R_N$.
Let
\begin{eqnarray}
B_{d,N}(\eta)&=&
\Bigl\{\xi^d~|~
\|\xi^d+\eta^d\|_{m,\theta'/2}<r,
\max_{1\le i\le d}\|C(\varphi^i,\xi^d+\eta^d)\|_{m,\theta}<r,\nonumber\\
& &\quad \|C(\xi^d+\eta^d,\varphi^d)\|_{m,\theta}<r,
\max_{1\le l<d}\|C(\eta^{l},\xi^d)\|_{m,\theta}<r/4
\Bigr\}.
\end{eqnarray}
For $1\le k\le d-1$,
taking $\xi^i\in
B_{i,N}(\xi^{i+1},\ldots, \xi^d, \eta)$~$(k+1\le i\le d)$
inductively, we define
\begin{eqnarray}
\lefteqn{B_{k,N}(\xi^{k+1},\ldots,\xi^d,\eta)}\nonumber\\
& &=
\Bigl\{
\xi^k \Big |~
\|\xi^k+\eta^{k}\|_{m,\theta'/2}<r,~
\max_{l>k}\|C(\xi^k+\eta^k,\xi^l+\eta^l)\|_{m,\theta}<r,\nonumber\\
& &\quad
\max_{1\le i\le k}\|C(\varphi^i,\xi^k+\eta^k)\|_{m,\theta}<r,
\max_{l\ge k}\|C(\xi^k+\eta^k,\varphi^l)\|_{m,\theta}<r,\nonumber\\
& &\quad
\max_{l>k}\|C(\xi^k,\eta^{l})\|_{m,\theta}<r/4,
\max_{1\le j<k}\|C(\eta^{j},\xi^k)\|_{m,\theta}<r/4
\Bigr\}.
\end{eqnarray}
Note that $0\in B_{k,N}(\xi^{k+1},\ldots,\xi^d,\eta)$.
We denote all elements
$(\xi^{k+1},\ldots,\xi^d)$ which can be obtained in this way by
$S_{k+1,d}(\eta)$.

Now we define a sequence of subsets
inductively.
First set $U_{d}(\eta)=U_{r,\varphi}(\eta)$.
Inductively, for $1\le k\le d-1$ and $\left(\xi^{k+1},\ldots,\xi^d\right)
\in S_{k+1,d}(\eta)$
define
\begin{eqnarray}
\lefteqn{U_k(\xi^{k+1},\ldots,\xi^d, \eta)}\nonumber\\
& &=
\Biggl\{
(\xi^1,\ldots,\xi^k)~\Big |~
\max_{1\le i\le k}\|\xi^i+\eta^i\|_{m,\theta'/2}<r,\nonumber\\
& &\qquad \max_{1\le i<j\le k}
\|C(\xi^i+\eta^i,\xi^j+\eta^j)\|_{m,\theta}<r,
\max_{1\le i\le k<l\le d}
\|C(\xi^i+\eta^i,\xi^l+\eta^l)\|_{m,\theta}<r,\nonumber\\
& &\qquad
\max_{1\le i\le j\le k}\|C(\varphi^i,\xi^j+\eta^j)\|_{m,\theta}<r,
\max_{1\le i\le k, i\le j}\|C(\xi^i+\eta^i,
\varphi^j)\|_{m,\theta}<r,\nonumber\\
& &\qquad 
\max_{1\le i\le k, i<j\le d}\|C(\xi^i,\eta^{j})
\|_{m,\theta}<r/4,
\max_{1\le i<j, 1<j\le k}
\|C(\eta^{i},\xi^j)\|_{m,\theta}<r/4
\Biggr\}.\nonumber\\
& &
\end{eqnarray}
Then 
\begin{equation}
B_{k,N}(\xi^{k+1},\ldots,\xi^d, \eta)
=\left\{
\xi^k~\Big |~
U_k(\xi^{k+1},\ldots,\xi^d,\eta)^{\xi^k}\ne\emptyset\right\}
\end{equation}
and for $\xi^k\in B_{k,N}(\xi^{k+1},\ldots,\xi^d,\eta)$,
\begin{eqnarray}
U_k(\xi^{k+1},\ldots,\xi^d,\eta)^{\xi^k}
&=&
U_{k-1}(\xi^k,\ldots,\xi^d,\eta).\label{induction step}
\end{eqnarray}
In the above and below,  $U_k(\cdots)^{\xi^k}$,
$U_k(\cdots)_{(\xi^1,\ldots,\xi^{k-1})}$
denote the sections as in Lemma~\ref{general PI}.
Also 
\begin{eqnarray}
\lefteqn{U_{k-1}(0,\xi^{k+1},\ldots,\xi^d, \eta)}\nonumber\\
& &=
\Bigl\{(\xi^1,\ldots,\xi^{k-1})~\Big|~
U_k(\xi^{k+1},\ldots,\xi^{d},\eta)_{
(\xi^1,\ldots,\xi^{k-1})}\ne \emptyset\Bigr\}\nonumber
\end{eqnarray}
and for $(\xi^{1},\ldots,\xi^{k-1})\in U_{k-1}
(0,\xi^{k+1},\ldots,\xi^{d},\eta)$,
\begin{eqnarray}
\lefteqn{U_k(\xi^{k+1},\ldots,\xi^d, \eta)
_{(\xi^1,\ldots,\xi^{k-1})}}\nonumber\\
& &=
\Bigl\{
\xi^k~\Big|~
\|\xi^{k}+\eta^{k}\|_{m,\theta'/2}<r,
\max_{1\le i<k}\|C(\xi^i+\eta^i,\xi^k+\eta^k)\|_{m,\theta}<r,\nonumber\\
& &\qquad \max_{l>k}\|C(\xi^k+\eta^k,\xi^l+\eta^l)\|_{m,\theta}<r\nonumber\\
& &\qquad \max_{1\le i\le k}\|C(\varphi^i,\xi^k+\eta^k)\|_{m,\theta}<r,
~\max_{l\ge k}\|C(\xi^k+\eta^k,\varphi^l)\|_{m,\theta}<r,\nonumber\\
& &\qquad
\max_{l>k}\|C(\xi^k,\eta^{l})\|_{m,\theta}<r/4,
\max_{1\le i<k}\|C(\eta^{i},\xi^k)\|_{m,\theta}<r/4
\Bigr\}.\label{convex-section}
\end{eqnarray}
Note that $U_k(\xi^{k+1},\ldots,\xi^d,\eta)
_{(\xi^1,\ldots,\xi^{k-1})}$ is a
convex set of $\RR^{2^N}$ and contains $0$.
Further, by Lemma~\ref{positivity}, we have
for all $1\le k\le d-1$,
\begin{multline}
\inf\Bigl\{\mu\Bigl(U_k(\xi^{k+1},\ldots,\xi^d,\eta)
_{x}
\cap
U_k(\xi^{k+1},\ldots,\xi^d,\eta)
_{y}\Bigr)~\Big |~
x,y\in U_{k-1}\left(0,\xi^{k+1},\ldots,\xi^d,\eta\right),\\
(\xi^{k+1},\ldots,\xi^d)\in S_{k+1,d}(\eta),~
\eta\in R_N
\Bigr\}>0
\end{multline}
and the lower bound is given by the inverse of products of 
$C(l,r/4,r,m,\theta,\theta')$.
Hence in order to check Claim 1,
by (\ref{induction step}) and Lemma~\ref{general PI}, we need to prove
Poincar\'e's inequality with a Poincar\'e constant
which is independent of $\xi^k,\ldots,\xi^d,\eta$
on $U_{k-1}(\xi^k,\ldots,\xi^d,\eta)$.
This is checked by using Lemma~\ref{PI on convex domain}.
Thus 
we see that 
Claim 1
holds with the constant
$C$ which depends only on the inverse of products of 
$C(l,r/4,r,m,\theta,\theta')$.

We prove Claim 2.
Let $\eta\in R_N$.
Then $\beta_N(\cdot,\eta)\in 
\wedge^1T^{\ast}U_{r,\varphi}(\eta)$
is also a closed $C^{\infty}$-differential form
and the supremum norm of all derivatives are finite 
for almost all $\eta$ by the Sobolev embedding theorem.
By Lemma~\ref{retraction formula} and using inductive argument,
we can construct a bounded function
$u_N(\cdot,\eta)\in 
C^{\infty}(U_{r,\varphi}(\eta))$ explicitly such that
$d_Nu_N=\beta_N$ and
$u_N(\xi,\eta)$ is a measurable function on $U_{r,\varphi,N}$.
Using $u_N$, we see that
$$
g_N=u_N-
\frac{1}{\mu_N\left(U_{r,\varphi}(\eta)\right)}
\int_{U_{r,\varphi}(\eta)}u_N(\xi,\eta)d\mu_N(\xi)
$$
is the desired function.

Now, we prove the existence of $g$ which satisfies the
desired property in the Theorem.
Let $g_N$ be the function in the Claim 2.
Then by the Poincar\'e inequality established in the Claim 1,
it holds that
\begin{equation}
\|g_N\|_{L^2(U_{r,\varphi,N})}^2\le
C\|\beta_N\|_{L^2(U_{r,\varphi,N})}^2
\le C\|\beta\|_{L^2(U_{r,\varphi})}^2.
\end{equation}
Let $\hat{g}_N(w)=g_N(w)1_{U_{r,\varphi,N}}(w)$.
Let us choose a positive numbers $r_1,r_2$
such that $0<r'<r_1<r_2<r$.
Let $\rho$ be a smooth function on $\RR^{3d(d+1)/2}$ such that
$
\max_y|\rho(y)-\max_i|y^i||
$
is sufficiently small.
It is easy to see the existence of such a function using
a mollifier.
Then there exists a small positive number
$\ep$ such that for any $r_1\le s\le r_2$,
\begin{multline}
\left\{x=(x^i)\in \RR^{3d(d+1)/2}~|~
\max_i|x^i|<r'+\ep\right\}
\subset
\left\{x=(x^i)\in \RR^{3d(d+1)/2}~|~\rho(x^{(m)})<s^m\right\}\\
\subset \left\{x=(x^i)\in \RR^{3d(d+1)/2}~|~
\max_i|x^i|<r
\right\},
\end{multline}
where $x^{(m)}=((x^1)^m,\ldots,(x^{3d(d+1)/2})^m)$.
Note that the index $j$ of $(x^i)^j$ 
is the power and $i$ stands for the $i$-th element.
Let $\hat{\rho}(w)$ be the composition of
$\rho$ and the $3d(d+1)/2$ random variables 
\begin{multline}
\|w^i\|_{m,\theta'/2}^m ~(1\le i\le d),
\|C(w^j,w^k)\|_{m,\theta}^m ~(1\le j<k\le d)\\
\|C(\varphi^i,w^j)\|_{m,\theta}^m ~(1\le i\le j\le d),
\|C(w^i,\varphi^j)\|_{m,\theta}^m ~(1\le i\le j\le d).
\end{multline}
Let $\chi$ be the smooth decreasing function such that
$\chi(u)=1$ for $u\le (r/6)^m$
$\chi(u)=0$ for $u\ge (r/5)^m$ and set
\begin{eqnarray*}
\hat{\chi}_N(w)&=&
\chi\Bigl(
\sum_{i=1}^d\|w(N)^{\perp,i}\|_{m,\theta'/2}^m+
\sum_{1\le j<k\le d}
\|C(w(N)^{\perp,j},w(N)^{\perp,k})\|_{m,\theta}^m\nonumber\\
& &
+\sum_{1\le i\le j\le d}\|C(\varphi^i,w(N)^{\perp,j})\|_{m,\theta}^m
+\sum_{1\le i\le j\le d}\|C(w(N)^{\perp,i},\varphi^j)\|_{m,\theta}^m
\nonumber\\
& &
+\sum_{1\le i<j\le d}\|C(w(N)^i,w(N)^{\perp,j})\|_{m,\theta}^m
+\sum_{1\le i<j\le d}\|C(w(N)^{\perp,i},w(N)^j)\|_{m,\theta}^m
\Bigr).
\end{eqnarray*}
Let $\psi$ be the smooth decreasing function such that
$\psi(u)=1$ for $u\le \frac{r_1^m+r_2^m}{2}$ 
and $\psi(u)=0$ for $u\ge \frac{r_1^m+2r_2^m}{3}$.
Let $h_N(w)=\hat{g}_N(w)\psi\left(\hat{\rho}(w)\right)
\hat{\chi}_N(w)$.
Since $\sup_N\|\hat{g}_N\|_{L^2(W^d,\mu)}<\infty$,
there exists
a subsequence $\hat{g}_{N(k)}$~$(N(1)<N(2)<\cdots)$
such that $\hat{g}_{N(k)}$ converges weakly to some 
$\hat{g}_{\infty}\in L^2(W^d,\mu)$.
Noting that $\|\hat{\chi}_N\|_{\infty}\le 1$ and
$\lim_{N\to\infty}\hat{\chi}_N(w)=1$ for all $w\in \Omega$,
we see that $\hat{g}_{N(k)}(w)\psi(\hat{\rho}(w))\hat{\chi}_{N(k)}(w)$
also converges weakly to
$\hat{g}_{\infty}(w)\psi(\hat{\rho}(w))$
which we denote by $h_{\infty}(w)$.
We calculate the weak derivative of $h_{\infty}$.
Fix a natural number $N_0$ and
let $\theta\in {\mathbb D}^{\infty}(W^d\to P_{N_0}H^{\ast})$.
Then
\begin{eqnarray}
\int_{W^d}h_{\infty}(w)D^{\ast}\theta(w)d\mu(w)
&=&\lim_{k\to\infty}\int_{W^d}
h_{N(k)}(w)D^{\ast}\theta(w)d\mu(w)\nonumber\\
&=&\lim_{k\to\infty}\int_{W^d}\left(d_{N(k)}h_{N(k)}(w),\theta(w)\right)
d\mu(w).
\end{eqnarray}
Here
\begin{eqnarray}
d_{N(k)}\left(\hat{g}_{N(k)}\psi(\hat{\rho})\hat{\chi}_{N(k)}\right)
&=&\beta_{N(k)}\psi(\hat{\rho})\hat{\chi}_{N(k)}
+\hat{g}_{N(k)}d_{N(k)}\left(\psi(\hat{\rho}(w))\right)
\hat{\chi}_{N(k)}(w)\nonumber\\
& &+
\hat{g}_{N(k)}\psi(\hat{\rho}(w))
d_{N(k)}\hat{\chi}_{N(k)}(w).
\end{eqnarray}
Noting that
\begin{eqnarray}
\lim_{k\to\infty}\|d_{N(k)}\left(\psi(\hat{\rho})\right)-
d\left(\psi(\hat{\rho})\right)\|_{L^4(\mu)}
&=&0,\\
\lim_{k\to\infty}
\|d_{N(k)}\hat{\chi}_{N(k)}\|_{L^4(\mu)}
&=&0,
\end{eqnarray}
we get
\begin{eqnarray}
\lefteqn{\int_{W^d}h_{\infty}(w)D^{\ast}\theta(w)d\mu(w)}\nonumber\\
& &=
\int_{W^d}
\Bigl(\beta(w)\psi(\hat{\rho}(w))
+\hat{g}_{\infty}(w)d\left(\psi(\hat{\rho}(w))\right),
\theta(w)
\Bigr)d\mu(w).
\end{eqnarray}
This implies $dh_{\infty}=\beta\psi(\hat{\rho})
+\hat{g}_{\infty}d\left(\psi(\hat{\rho})\right)$ in weak sense.
By Lemma~\ref{pointwise estimate} and Lemma~\ref{H-derivative},
$d\left(\psi(\hat{\rho})\right)$
is a bounded function.
Hence $dh_{\infty}\in L^2(W^d,\mu)$ which implies 
$h_{\infty}\in {\mathbb D}^{1,2}(W^d,\RR)$.
Also $h_{\infty}$ satisfies that 
$dh_{\infty}=\beta$ on $U_{r',\varphi}$.
Finally we need to show the regularity of the higher order
derivatives of $h_{\infty}$.
Choosing a smooth function $\psi_1$ on $\RR$ such that
$\psi_1(u)=1$ for $u\le \frac{r_1^m+3r_2^m}{4}$
and $\psi_1(u)=0$ for $u\ge \frac{r_1^m+4r_2^m}{5}$,
we have 
$$
\hat{g}_{\infty}\psi_1(\hat{\rho})d\left(\psi(\hat{\rho})\right)
=\hat{g}_{\infty}d\left(\psi(\hat{\rho})\right).
$$
We see that $\hat{g}_{\infty}\psi_1(\hat{\rho})
\in {\mathbb D}^{1,2}(W^d,\RR)$
by the same argument as the above.
Hence $h_{\infty}\in {\mathbb D}^{2,q}(W^d,{\mathbb R})$.
Iterating this procedure, we get 
$h_{\infty}\in {\mathbb D}^{\infty,q}(W^d,\RR)$.
\end{proof}

\begin{rem}
In the same way as the proof of Claim 1,
we can prove that for any $g\in {\mathbb D}^{1,2}(W^d)$,
\begin{eqnarray}
\Var(g; U_{r,\varphi})&\le&
C\int_{U_{r,\varphi}}
|Dg(w)|_H^2d\mu_{U_{r,\varphi}}(w),
\end{eqnarray}
where $\mu_{U_{r,\varphi}}$ denotes the normalized probability measure
on $U_{r,\varphi}$ and $\Var$ denotes the variance with respect to the measure.
We may define a local Sobolev space $W^1(U_{r,\varphi})$.
It is not clear that $W^1(U_{r,\varphi})$ coincides with the restriction
of ${\mathbb D}^{1,2}(W^d)$ to $U_{r,\varphi}$ at the moment.
Note that the extension property of functions on convex sets were studied in
{\rm \cite{hino}}.
See {\rm \cite{hino2}} for more recent results.
\end{rem}

Let $B_{\ep}(e)=\{a\in G~|~d(a,e)<\ep\}$.
We assume that $\ep$ is sufficiently small
and $B_{\ep}(e)$ is diffeomorphic to a standard ball
in a Euclidean space.
Let
$$
{\cal D}_{\ep}=\left\{w\in\Omega~|~X(1,e,w)\in B_{\ep}(e)\right\}.
$$
This set is formally homotopy equivalent to 
$S=\{w\in \Omega~|~X(1,e,w)=e\}$ and $L_e(G)$.
We construct a covering of ${\cal D}_{\ep}$ by
a countable family of $U_{r}(\varphi)$ in the next section.
This covering is vital for the proof of
the existence of $f$ satisfying 
$df=\alpha$.

\section{A covering lemma for ${\mathcal D}_{\ep}$}

For $K\in {\mathbb N}$ and $0<\kappa<1$, let
\begin{eqnarray}
A_{K}&=&\left\{w\in \Omega~|~d_{\Omega}(0,w)<K \right\}\\
B_{N,\kappa}&=&
\Bigl\{
w\in \Omega~\Big|~
\max_{i}\|w(N)^{\perp,i}\|_{m,\theta'/2}<\kappa,
\max_{1\le i<j\le
d}\|C(w(N)^{\perp,i},w(N)^{\perp,j})\|_{m,\theta}<\kappa,
\nonumber\\
& &\quad \max_{1\le i\le j\le d}\|C(w(N)^{i},w(N)^{\perp,j})\|_{m,\theta}
<\kappa,
\max_{1\le i\le j\le d}
\|C(w(N)^{\perp,i},w(N)^{j})\|_{m,\theta}<\kappa
\Bigr\}.\nonumber\\
& &
\end{eqnarray}
Note that $A_K=U_K(0)$, $B_{N,\kappa}=\{w\in \Omega~|~w\in 
U_{\kappa}(w(N))\}$.
For $w\in A_K\cap B_{N,\kappa}$, $\max_i\|w(N)^i\|_{m,\theta'/2}<K+1$.
Let $\ep_n=\ep(1-\frac{1}{n})$~$(n=1,2,\ldots)$ and
\begin{eqnarray}
{\cal D}_{\ep_n,K,N,\kappa}&=&
{\cal D}_{\ep_n}\cap A_K\cap B_{N,\kappa}\\
\end{eqnarray}
For any $\kappa>0, n,K$, we have
\begin{equation}
\liminf_{N\to\infty}{\cal D}_{\ep_n,K,N,\kappa}
={\cal D}_{\ep_n}\cap A_K.\label{key-covering 1}
\end{equation}

For fixed $n$ and $K$,  we can find a positive number
$\kappa(n,K)$ such that there exists a
finite cover of ${\cal D}_{\ep_n,K,N,\kappa(n,K)}$
by $U_{r}(\varphi)$ which satisfies
$U_{r}(\varphi)\subset {\cal D}_{\ep_{2n}}$.
Since (\ref{key-covering 1}) holds,
this implies that there exists a countable cover
of ${\cal D}_{\ep_n}\cap A_K$ by $U_r(\varphi)$
which are included in ${\cal D}_{\ep_{2n}}$ and
so does for ${\cal D}_{\ep}$ too.
More precisely we prove the following.

\begin{lem}\label{covering}
$(1)$~Let $R_{m,\theta}=\max(M_{m,\theta}^2, N_{m,\theta})$.
See Lemma~$\ref{pointwise estimate}$\,$(1)$ and Lemma~$\ref{Hoelder and Besov}$
for the constants $M_{m,\theta}, N_{m,\theta}$.
Let 
\begin{equation}
\kappa<
\min\left(
\frac{\ep}{48nR_{m,\theta}(K+1)
F\left(K+18R_{m,\theta}\left(
K+1)\right)\right)}, \frac{1}{2}
\right),
\label{kappa}
\end{equation}
where $F$ is a function which appeared in Proposition~$\ref{solution of sde}$.
Let $w\in {\cal D}_{\ep_n,K,N,\kappa}$.
We take $\varphi\in H$ such that
\begin{equation}
\|\varphi-w(N)\|_H\le
\frac{\kappa}{
3\left(6\kappa+2K+5\right)
}.
\end{equation}
Then 
\begin{equation}
w\in U_{4\kappa/3}(\varphi)
\subset U_{\sqrt{2}\kappa}(\varphi)
\subset {\cal D}_{\ep_{2n}}.
\end{equation}

\noindent
$(2)$~Let $\kappa$ be a positive number satisfying 
$(\ref{kappa})$.
Then for any $N\in {\mathbb N}$,
there exists $L=L(n,K,N,\kappa)$ and a finite number of
piecewise linear paths
$\{\varphi_i\}_{i=1}^L\subset \Omega_N$
such that 
\begin{equation}
{\cal D}_{\ep_n,K,N,\kappa}\subset
\cup_{i=1}^LU_{4\kappa/3}(\varphi_i)\subset 
\cup_{i=1}^LU_{\sqrt{2}\kappa}(\varphi_i)\subset 
{\cal D}_{\ep_{2n}}.
\end{equation}

\noindent
$(3)$~Let
$\{\kappa_i,\varphi_i\}_{i=1}^{\infty}$
be countable positive numbers and
piecewise linear paths 
which are obtained in $(2)$ when $N, K, n$ take all values of natural
numbers.
Then it holds that
\begin{equation}
{\cal D}_{\ep}=\cup_{i=1}^{\infty}
U_{4\kappa_i/3}(\varphi_i)=
\cup_{i=1}^{\infty}
U_{\sqrt{2}\kappa_i}(\varphi_i).\label{covering of Dep}
\end{equation}
\end{lem}

We need a lemma to prove
the above.

For $z\in \Omega$, let us define
\begin{eqnarray}
V_r(z)&=&\left\{w\in \Omega~|~d_{\Omega}(w,z)<r\right\}.\label{ball}
\end{eqnarray}

\begin{lem}\label{U and V}
Let $r>0$.

\noindent
$(1)$~Let $\varphi_1=(\varphi_1^1,\ldots,\varphi_1^d), 
\varphi_2=(\varphi_2^1,\ldots,\varphi_2^d)\in H$.
Let $0<\delta<1$.
If 
\begin{eqnarray}
\max_i\|\varphi_1^i-\varphi_2^i\|_H
&\le&
\frac{\delta r}{
1+3r+2\max_i\|\varphi_1^i\|_{m,\theta/2}}
\label{inclusion1}
\end{eqnarray}
then 
$U_{r}(\varphi_1)\subset U_{(1+\delta)r}(\varphi_2)$.

If the stronger assumption
\begin{eqnarray}
\max_i\|\varphi_1^i-\varphi_2^i\|_H
&\le&
\frac{\delta r}{1+6r+2\max
_{i}\left(\|\varphi_1^i\|_{m,\theta/2}, \|\varphi_2^i\|_{m,\theta/2}
\right)}
\label{inclusion2}
\end{eqnarray}
holds,
then we have
$$
U_r(\varphi_1)\subset U_{(1+\delta)r}(\varphi_2)
\subset U_{(1+\delta)^2r}(\varphi_1).
$$

\noindent
$(2)$~
Let $0<r<1$ and $\varphi\in H$.
Then
\begin{equation}
U_r(\varphi)\subset
V_{R_{m,\theta}(5+6\|\varphi\|_{m,\theta/2})r}(\varphi).\label{inclusion3}
\end{equation}
\end{lem}

\begin{proof}
(1)~Let $\ep=\max_i\|\varphi_1^i-\varphi_2^i\|_H$.
Let $w\in U_r(\varphi_1)$.
Then we have
\begin{eqnarray}
\|w^i-\varphi^i_2\|_{m,\theta'/2}&\le&
\|w^i-\varphi^i_1\|_{m,\theta'/2}+
\|\varphi^i_1-\varphi^i_2\|_{m,\theta'/2}
<r+\ep,
\end{eqnarray}
\begin{eqnarray}
\lefteqn{\|C(w^j-\varphi^j_2,w^k-\varphi^k_2)\|_{m,\theta}}\nonumber\\
& &=
\Bigr\|C(w^j-\varphi_1^j,w^k-\varphi_1^k)
+C(\varphi_1^j-\varphi_2^j,w^k-\varphi_1^k)+
C(w^j-\varphi_1^j,\varphi_1^k-\varphi_2^k)\nonumber\\
& &
\qquad +C(\varphi_1^j-\varphi_2^j,\varphi_1^k-\varphi^k_2)
\Bigl\|_{m,\theta}\nonumber\\
& &<
r+3\ep r+\ep^2,
\end{eqnarray}
\begin{eqnarray}
\|C(\varphi_2^i,w^j-\varphi^j_2)\|_{m,\theta}
&=&
\Bigl\|
C(\varphi_1^i,w^j-\varphi_1^j)+
C(\varphi_1^i,\varphi_1^j-\varphi_2^j)+
C(\varphi_2^i-\varphi^i_1,w^j-\varphi^j_1)\nonumber\\
& &\quad+C(\varphi_2^i-\varphi_1^i,\varphi_1^j-\varphi_2^j)
\Bigr\|_{m,\theta}\nonumber\\
&<&
r+\ep\|\varphi_1^i\|_{m,\theta/2}+2\ep r+\ep^2.
\end{eqnarray}
In the above, we have used Lemma~\ref{pointwise estimate}~(2).
Similarly,
\begin{eqnarray}
\|C(w^i-\varphi_2^i,\varphi_2^j)\|_{m,\theta}
&<&
r+\ep r+2\ep\|\varphi_1^j\|_{m,\theta/2}+\ep^2.
\end{eqnarray}
Therefore if
$$
\ep\left(3r+1+2\max_i\|\varphi_1^i\|_{m,\theta/2}
\right)
\le \delta r,
$$
then
$w\in U_{r(1+\delta)}(\varphi_2)$ which proves the first statement.
The second statement follows from the first one.

\noindent
(2)~Assume $w\in U_r(\varphi)$. Let $i<j$.
Since $C(w^i,w^j)-C(\varphi^i,\varphi^j)=
C(w^i-\varphi^i,w^j-\varphi^j)+
C(\varphi^i,w^j-\varphi^j)+
C(w^i-\varphi^i,\varphi^j)$,
noting Lemma~\ref{Hoelder and Besov}, we have
$$
\|C(w^i,w^j)-C(\varphi^i,\varphi^j)\|_{H,\theta}
< 4N_{m,\theta}r(1+\|\varphi^i\|_{m,\theta/2}).
$$
Note that $C(w^i-\varphi^i,w^j-\varphi^j)$
is a limit of iterated integrals of smooth paths and so
we can still apply Lemma~\ref{Hoelder and Besov}.
Let us consider the case where $i=j$.
Since 
\begin{eqnarray}
\lefteqn{C(w^i,w^i)_{s,t}-C(\varphi^i,\varphi^i)_{s,t}}\nonumber\\
& &=\frac{1}{2}\left\{(w^i-\varphi^i)_t-(w^i-\varphi^i)_s\right\}^2
+C(\varphi^i,w^i-\varphi^i)_{s,t}+
C(w^i-\varphi^i,\varphi^i)_{s,t},
\end{eqnarray}
\begin{eqnarray}
\lefteqn{\|C(w^i,w^i)-C(\varphi^i,\varphi^i)\|_{H,\theta}}\nonumber\\
& &\le \frac{1}{2}\|w^i-\varphi^i\|_{H,\theta/2}^2
+\|C(\varphi^i,w^i-\varphi^i)\|_{H,\theta}+
\|C(w^i-\varphi^i,\varphi^i)\|_{H,\theta}\nonumber\\
& &\le \frac{1}{2}M_{m,\theta}^2r^2+2N_{m,\theta}(1+\|\varphi^i\|_{m,\theta/2})r.
\end{eqnarray}
Let $i>j$.
Using (\ref{ibp for rough path}),
we have
\begin{eqnarray}
\lefteqn{C(w^i,w^j)_{s,t}-C(\varphi^i,\varphi^j)_{s,t}}\nonumber\\
& &=
C(\varphi^j,\varphi^i)_{s,t}-C(w^j,w^i)_{s,t}
+\left\{(w^i-\varphi^i)_t-(w^i-\varphi^i)_s\right\}
\left\{(w^j-\varphi^j)_t-(w^j-\varphi^j)_s\right\}\nonumber\\
& &+(\varphi^i_t-\varphi^i_s)
\left\{(w^j-\varphi^j)_t-(w^j-\varphi^j)_s\right\}
+\left\{(w^i-\varphi^i)_t-(w^i-\varphi^i)_s\right\}
(\varphi^j_t-\varphi^j_s).
\end{eqnarray}
Hence
\begin{eqnarray*}
\|C(w^i,w^j)-C(\varphi^i,\varphi^j)\|_{m,\theta}
&\le&
4N_{m,\theta}r(1+\|\varphi^i\|_{m,\theta/2})+
M_{m,\theta}^2r^2+
2rM_{m,\theta}^2\max_i\|\varphi^i\|_{m,\theta/2}
\end{eqnarray*}
which completes the proof of (\ref{inclusion3}).
\end{proof}

\begin{proof}[Proof of Lemma~$\ref{covering}$]
(1)~Suppose that $w\in {\cal D}_{\ep_n,K,N,\kappa}$.
Then $\|w(N)\|_{m,\theta'/2}<K+1$.
By Lemma~\ref{U and V}~(2),
$d_{\Omega}(w(N),w)< 6R_{m,\theta}(K+1)\kappa$.
Hence
$
d_{\Omega}(w(N),0)\le K+
6R_{m,\theta}(K+1)\kappa.
$
By Proposition~\ref{solution of sde}~(2),
\begin{eqnarray}
d(X(1,e,w(N)),e)&\le&
d(X(1,e,w(N)),X(1,e,w))+d(X(1,e,w),e)\nonumber\\
&<& 6R_{m,\theta}(K+1)\kappa F\left(K+6R_{m,\theta}(K+1)\kappa\right)
+\ep_n.
\end{eqnarray}
Hence, if
$$
\kappa<
\kappa(n,p,K,\ep):=
\min\left(\frac{\ep}
{6np R_{m,\theta}(K+1)F\left(K+6R_{m,\theta}(K+1)\right)}
,1\right),
$$
then
$X(1,e,w(N))\in B_{\ep(1-\frac{1}{n}(1-\frac{1}{p}))}(e)$.
Now assume that $\kappa<1/2$.
Let $z\in U_{2\kappa}(w(N))$.
Then
$d_{\Omega}(w(N),z)< 12R_{m,\theta}\left(K+1\right)\kappa$.
Thus
$
d_{\Omega}(0,z)<
18R_{m,\theta}(K+1)\kappa.
$
Therefore
\begin{eqnarray}
\lefteqn{d\left(X(1,e,z),e\right)}\nonumber\\
& &\le
d(X(1,e,z),X(1,e,w(N)))+d(X(1,e,w(N)),e)\nonumber\\
& &<
12R_{m,\theta}(K+1)\kappa F\left(K+18R_{m,\theta}(K+1)\kappa\right)
+
\ep\left(1-\frac{1}{n}(1-\frac{1}{p})\right).
\end{eqnarray}
Consequently if 
$$
\kappa<
\min\left(\frac{1}{2},\kappa(n,p,K,\ep),
\frac{\ep}{12nqF\left(K+
K+18R_{m,\theta}(K+1)\right)
R_{m,\theta}(K+1)
}\right),
$$
$
d(X(1,e,z),e)<\ep\left(1-\frac{1}{n}(1-\frac{1}{p}-\frac{1}{q})\right)
$
holds.
Now we set $p=q=4$ and $\kappa$ to be a positive number such that
\begin{equation}
\kappa<
\min\left(\frac{\ep}{48nF\left(K+18R_{m,\theta}(K+1)\right)
R_{m,\theta}(K+1)}, \frac{1}{2}\right).
\end{equation}
For such a $\kappa$, it holds that
if $w\in {\cal D}_{\ep_n,K,N,\kappa}$ then
$z\in {\cal D}_{\ep_{2n}}$
for any $z\in U_{2\kappa}(w(N))$.
That is, $w\in U_{\kappa}(w(N))\subset U_{2\kappa}(w(N))\subset
{\cal D}_{\ep_{2n}}$.
Applying Lemma~\ref{U and V}~(1) to the case where
$\varphi_1=w(N)$, $\varphi_2=\varphi$, 
$r=\kappa$, $\delta=\sqrt{2}-1, 1/3$,
we have
if
$$
\|\varphi-w(N)\|_H<\frac{\kappa}{
3\left(6\kappa+1+2(K+2)\right)
}
$$
then
$$
w\in U_{\kappa}(w(N))\subset
U_{4\kappa/3}(\varphi)\subset
U_{\sqrt{2}\kappa}(\varphi)
\subset
U_{2\kappa}(w(N))\subset
{\cal D}_{\ep_{2n}}.
$$
This completes the proof of (1) from which
follow
(2) and (3).
\end{proof}

\section{$H$-simply connected set in a Wiener space}

We introduce the following notions.

\begin{dfi}
Let $D$ be an $H$-open and  measurable subset of
$\Omega$ with $\mu(D)>0$.
Here $D$ is said to be $H$-open if for any $w\in D$,
there exists $\ep>0$ such that
$w+\{h\in H~|~\|h\|_H<\ep\}\subset D$.

\noindent
$(1)$~$D$ is called an $H$-connected set if,
whenever $w, w+h\in D$, there exists a $C^{\infty}$ curve 
$h : [0,1]\to H$ such that
$h(0)=0$ and $h(1)=h$ and $w+h(\tau)\in D$ for all $0\le \tau\le 1$.

\noindent
$(2)$~$D$ is called an $H$-simply connected set if
the following holds:
Let us fix any point $w$ of $D$.
Let $\{h(0,\tau)~|~0\le \tau\le 1\}$
and $\{h(1,\tau)~|~0\le \tau\le 1\}$
be $C^{\infty}$ curves on $H$ such that
$h(0,0)=h(1,0)=0$, $h(0,1)=h(1,1)$ and
$\{w+h(i,\tau)~|~0\le \tau\le 1\}\subset D$ for
$i=0,1$. Then there exists
a $C^{\infty}$ map ${\cal H} :[0,1]^2
\to H$ which may depend on $w$ such that
\begin{itemize}
\item[{\rm (i)}]~${\cal H}(0,\tau)=h(0,\tau)$,
${\cal H}(1,\tau)=h(1,\tau)$ for all $0\le \tau\le 1$,
\item[{\rm (ii)}]~${\cal H}(\sigma,0)=0$ and
${\cal H}(\sigma,1)=h(0,1)=h(1,1)$ for all $\sigma$,
\item[{\rm (iii)}]~
$w+{\cal H}(\sigma,\tau)\in D$ holds
for any $(\sigma,\tau)\in [0,1]^2$.
\end{itemize}
\end{dfi}

The ball like set $U_r(\varphi)$ is $H$-connected.
We need the following lemma to prove this statement.
Also this lemma will be used in the proof of
Proposition~\ref{covering2}~(2).

\begin{lem}\label{Urvarphi and H}
Let $\varphi_i\in H$ and $r_i>0$~$(i=1,2)$.
The following three conditions {\rm (i), (ii), (iii)}
are equivalent.
\begin{itemize}
\item[{\rm (i)}]~$\mu\left(U_{r_1}(\varphi_1)\cap U_{r_2}(\varphi_2)\right)>0$.
\item[{\rm (ii)}]~$U_{r_1}(\varphi_1)\cap U_{r_2}(\varphi_2)\ne \emptyset$.
\item[{\rm (iii)}]~
$U_{r_1}(\varphi_1)\cap U_{r_2}(\varphi_2)\cap H
\ne \emptyset$.
\end{itemize}
\end{lem}

\begin{proof} 
It is trivial that (i) implies (ii).
The implication (ii) $\Longrightarrow$ (iii) follows 
from that\\
$\lim_{N\to\infty}d_{\Omega}(w(N),w)=0$ for any
$w\in \Omega$.
We prove (iii) implies (i).
By the assumption, there exists $h\in U_{r_1}(\varphi_1)
\cap U_{r_2}(\varphi_2)\cap H$.
Let $\ep$ be a sufficiently small positive number.
Let $w\in U_{\ep}(0)$.
Then $w+h\in U_{r_1}(\varphi_1)\cap U_{r_2}(\varphi_2)$
and $\mu\left(U_{\ep}(0)+h\right)>0$.
This proves (i).
\end{proof}

\begin{lem}\label{H-connectivity}
Let $D_i=U_{r_i}(\varphi_i)$~$(1\le i\le n)$.
Assume that $\left(\cup_{i=1}^kD_i\right)\cap D_{k+1}\ne \emptyset$.
Then $D=\cup_{i=1}^nD_i$ is an $H$-connected set.
\end{lem}

\begin{proof} 
Clearly, $D_i, D$ are $H$-open sets.
Let $w,w+h\in D$.
Without loss of generality, we may assume that
$w\in D_1$, $w+h\in D_i$ and
$D_{k}\cap D_{k+1}\ne \emptyset$ for all
$1\le k\le i-1$.
Let $\psi_k\in D_k\cap D_{k+1}\cap H$.
Let $\varphi_{k,w(N)^{\perp}}=\varphi_k+w(N)^{\perp}$
and $\psi_{k,w(N)^{\perp}}=\psi_k+w(N)^{\perp}$
Then for sufficiently large $N$, it holds that 
\begin{eqnarray}
\{(1-\tau)\varphi_{k,w(N)^{\perp}}+\tau\psi_{k,w(N)^{\perp}}
~|~0\le \tau\le 1\}&\subset& D_k\quad (k=1,\ldots,i-1),\\
\{(1-\tau)\psi_{k-1,w(N)^{\perp}}+\tau\varphi_{k,w(N)^{\perp}}~
|~0\le \tau\le 1\}
&\subset& D_k \quad (k=2,\ldots,i)\\
\{(1-\tau)w+\tau\varphi_{1,w(N)^{\perp}}~|~0\le \tau\le 1\}&\subset& D_1,\\
\{(1-\tau)(w+h)+\tau\varphi_{i,w(N)^{\perp}}~|~0\le \tau\le 1\}&\subset&
D_i.
\end{eqnarray}
This follows from Theorem~\ref{Omega}.
Hence, we have proved the existence of a piecewise linear path
$h=h(\tau)$~$(0\le \tau\le 1)$ such that
$h(0)=0$, $h(1)=h$ and $w+h(\tau)\subset D$ for all $0\le \tau\le 1$.
Note that if $\sup_{\tau}\|\tilde{h}(\tau)-h(\tau)\|_H$ is sufficiently small,
then $\{w+\tilde{h}(\tau)~|~0\le \tau\le i+1\}\subset D$.
Thus we see the existence of a smooth path connecting $w$ and $w+h$.
\end{proof}

The space of mapping, $H^1([0,1]\to G)$, is a 
$C^{\infty}$-Hilbert manifold naturally.
In the lemma below, we use this differentiable structure.

\begin{lem}\label{simply connectedness}
Assume that $G$ is a simply
connected compact Lie group.
Let $V$ be an open set of $G$ which is diffeomorphic to
a ball in a Euclidean space.
Let
$$
H^1_{V}=\{\gamma \in H^1([0,1]\to G)~|~
\gamma_0=e, \gamma_1\in V\}.
$$
Let $\{\gamma(i,\tau)~|~0\le \tau\le 1\}\subset H^1_V$ ~$(i=0,1)$
be two $C^{\infty}$-curves with the same
starting point and end point in $H_{V}^1$, 
that is, we assume
$$
\gamma(0,0)=\gamma(1,0)\in H^1_V, \gamma(0,1)=\gamma(1,1)\in H_V^1.
$$
Then there exists a $C^{\infty}$-homotopy map ${\cal M} :
(\sigma,\tau)(\in [0,1]^2)\mapsto {\cal M}(\sigma,\tau)\in H_V^1$
such that 
\begin{itemize}
\item[{\rm (i)}]~${\cal M}(0,\tau)=\gamma(0,\tau)$ and
${\cal M}(1,\tau)=\gamma(1,\tau)$ for all $\tau$,
\item[{\rm (ii)}]~${\cal M}(\sigma,0)=\gamma(0,0)=\gamma(1,0)$
and ${\cal M}(\sigma,1)=\gamma(0,1)=\gamma(1,1)$ for all 
$\sigma$.
\end{itemize}
\end{lem}

\begin{proof}
This follows from that $\pi_2(G)=0$ and
so $\pi_1(L_e(G))=0$.
See \cite{bott} and \cite{segal}.
This is the result in continuous category.
In the case of $H^1$-paths, it suffices to approximate the
continuous homotopy by a smooth homotopy.
\end{proof}

\begin{pro}\label{covering2}
Assume that $G$ is a simply connected
compact Lie group.

\noindent
$(1)$~The subset ${\cal D}_{\ep}$ is an $H$-connected and
$H$-simply connected set for sufficiently small $\ep$.

\noindent
$(2)$~Let $\{U_{4\kappa_i/3}(\varphi_i), i=1,2,\ldots\}$
be the sets which are defined in Lemma~$\ref{covering}$~$(3)$.
Then if necessary, by changing the order of the sets,
we have
$$
\mu\left(\left(\cup_{i=1}^nU_{4\kappa_{i}/3}(\varphi_{i})\right)
\cap U_{4\kappa_{n+1}/3}(\varphi_{n+1})\right)>0
\qquad \mbox{for all $n\ge 1$}.
$$
\end{pro}

\begin{proof}~(1)~
First we prove that ${\cal D}_{\ep}$ is an $H$-connected set.
Assume that $w,w+h\in {\cal D}_{\ep}$.
Then $X(1,e,w+h), X(1,e,w)\in B_{\ep}(e)$.
Let $Z(t,h,w)$ be the $H^1$-path in Proposition~\ref{solution of sde}.
Since $X(1,e,w+h)=X(1,Z(1,h,w),w)$,
$t\mapsto Z(t,h,w)$ is a $H^1$-curve on $G$
starting at $e$ and $Z(1,h,w)\in X^{-1}(1,\cdot,w)(B_{\ep}(e))$.
Also $e\in X^{-1}(1,\cdot,w)(B_{\ep}(e))$ holds.
Since $G$ is simply connected and $X^{-1}(1,\cdot,w)\left(B_{\ep}(e)\right)$ 
is a contractive set, 
there exists a map $(\tau,t)\in [0,1]^2\mapsto \gamma^{h,w}(\tau)_{t}\in G$ 
such that
\begin{itemize}
\item[(i)]~$\gamma^{h,w}(0)_{t}=e$ and $\gamma^{h,w}(1)_{t}=Z(t,h,w)$
for all $0\le t\le1$,
\item[(ii)]~$\tau\in [0,1]\mapsto \gamma^{h,w}(\tau)$ is a 
$C^{\infty}$-map with values in
$H^1_{X^{-1}(1,\cdot,w)(B_{\ep}(e))}$.
\end{itemize}
Now we define $h(\tau)=\zeta(\gamma^{h,w}(\tau),w)$.
See Proposition~\ref{solution of sde} for the definition of $\zeta$.
The mapping $\tau(\in [0,1])\mapsto h(\tau)$ is a $C^{\infty}$-curve on $H$.
Also $X(t,\gamma^{h,w}(\tau)_{t},w)
=X(t,e,w+h(\tau))$~$((\tau,t)\in [0,1]^2)$ holds
by the definition.
Therefore $h(0)=0$, $h(1)=h$ and $X(1,e,w+h(\tau))\in B_{\ep}(e)$
for all $0\le \tau \le 1$.
This proves that ${\cal D}_{\ep}$ is an $H$-connected set.
Next we prove the $H$-simply connectedness of ${\cal D}_{\ep}$.
Let $\tau\in [0,1]\mapsto h(i,\tau)
\in H$~$(i=0,1)$ be $C^{\infty}$-curves on $H$ such that
\begin{itemize}
\item[(i)]~$w+h(i,\tau)\in {\cal D}_{\ep}$ for all $0\le \tau\le 1$ and
	  $i=0,1$.
\item[(ii)]~$h(0,0)=h(1,0)=0$, $h(0,1)=h(1,1)$.
\end{itemize}
Then 
$Z(t,h(0,0),w)=Z(t,h(1,0),w)=e$ and
$Z(t,h(0,1),w)=Z(t,h(1,1),w)$ hold for all
$0\le t\le 1$.
Also $t\mapsto Z(t,h(i,\tau),w)$ is a $H^1$-curve on $G$
starting at $e$ and the end point
$Z(1,h(i,\tau),w)\in X^{-1}(1,\cdot,w)(B_{\ep}(e))$
for all $0\le \tau\le 1$ and $i=0,1$.
Therefore $\tau\mapsto Z(\cdot,h(i,\tau),w)$ is a $C^1$-map
from $[0,1]$ to $H^1_{X^{-1}(1,\cdot,w)(B_{\ep}(e))}$.
Since $B_{\ep}(e)$ is a contractive set,
$H^1_{X^{-1}(1,\cdot,w)(B_{\ep}(e))}$ is
also a simply connected set by Lemma~\ref{simply connectedness}.
Therefore
there exists a $C^{\infty}$ homotopy map 
\begin{equation}
(\sigma,\tau)(\in [0,1]^2)\mapsto {\cal M}^{h,w}(\sigma,\tau)\in 
H^1_{X^{-1}(1,\cdot,w)(B_{\ep}(e))}
\end{equation}
such that
\begin{itemize}
\item[(i)]~${\cal M}^{h,w}(i,\tau)_t=Z(t,h(i,\tau),w)$ 
for all $0\le \tau, t\le 1$ and $i=0,1$,
\item[(ii)]~${\cal M}^{h,w}(\sigma,0)_t=Z(t,h(0,0),w)=Z(t,h(1,0),w)=e$
and ${\cal M}^{h,w}(\sigma,1)_t=Z(t,h(0,1),w)=Z(t,h(1,1),w)$
for all $0\le \sigma\le 1$.
\end{itemize}
Let
\begin{equation}
{\cal H}(\sigma,\tau)=\zeta\left({\cal M}^{h,w}(\sigma,\tau),w\right).
\end{equation}
Then 
\begin{itemize}
\item[(i)]~
${\cal H}(i,\tau)=h(i,\tau)$
for all $0\le \tau\le 1$ and $i=0,1$,
\item[(ii)]~For all $\sigma$,
${\cal H}(\sigma,0)=0$ and ${\cal H}(\sigma,1)=h(0,1)=h(1,1)$,
\item[(iii)]~The mapping
$(\sigma,\tau)\in [0,1]^2\mapsto {\cal H}(\sigma,\tau)\in H$ is
$C^{\infty}$,
\item[(iv)]~
$w+{\cal H}(\sigma,\tau)\in {\cal D}_{\ep}$ for all $(\sigma,\tau)$.
\end{itemize}
These complete the proof.

\noindent
(2)~Since the map $h(\in H)\mapsto X(\cdot,e,h)(\in H^1([0,1]\to G~|~\gamma(0)=e))$
is a diffeomorphism, ${\cal D}_{\ep}\cap H$ is diffeomorphic to
$H^1_{B_{\ep}(e)}$.
Hence, ${\cal D}_{\ep}\cap H$ is an open connected subset
of $H$. Since $U_{4\kappa_i/3}(\varphi_i)\cap H$ is an open subset of
$H$ and ${\cal D}_{\ep}\cap H=\cup_{i=1}^{\infty}
\left(U_{4\kappa_i/3}(\varphi_i)\cap H\right)$,
it is an easy exercise to show that if necessary, by changing the order of the sets,
we have 
$$
\cup_{i=1}^n\left(U_{4\kappa_i/3}(\varphi_i)\cap H\right)\cap 
U_{4\kappa_{n+1}/3}(\varphi_{n+1})\ne \emptyset\qquad \mbox{for all $n=1,2,\ldots$}.
$$
Thus, by Lemma~\ref{Urvarphi and H}, we complete the proof.
\end{proof}

\begin{lem}[Stokes theorem in $H$-direction]\label{stokes}
~$(1)$~Let $f\in {\mathbb D}^{1,q}(W^d)$, where $q>1$.
Then for any $C^1$-curve $h=h(\tau)$~$(0\le \tau\le 1)$ on $H$,
we have
\begin{equation}
f(w+h(1))=f(w+h(0))+\int_0^1\left((Df)(w+h(t)),\dot{h}(t)\right)_Hdt
\quad \mbox{$\mu$-almost all $w$}.
\end{equation}

\noindent
$(2)$~
Let $\beta\in {\mathbb D}^{1,q}(W^d,H^{\ast})$, where
$q>1$.
Let $\calh=\calh(\sigma,\tau)$~$((\sigma,\tau)\in [0,1]^2)$ be a
$C^2$-map with values in $H$. We assume that $\calh(\sigma,0)=\calh(0,0)$ and
$\calh(\sigma,1)=\calh(0,1)$ for all $0\le \sigma\le 1$.
Then it holds that
\begin{eqnarray}
\lefteqn{\int_{0}^1\left(\beta(w+\calh(1,\tau)), \partial_{\tau} 
\calh(1,\tau)\right)d\tau
-\int_0^1\left(\beta(w+\calh(0,\tau)),\partial_{\tau}\calh(0,\tau)\right)
d\tau}
\nonumber\\
& &
=\iint_{(\sigma,\tau)\in [0,1]^2}
(d\beta)(w+\calh(\sigma,\tau))
\left(\partial_{\sigma}\calh(\sigma,\tau),\partial_{\tau}
\calh(\sigma,\tau)\right)
d\sigma d\tau
~~\mbox{$\mu$-almost all $w$}.
\end{eqnarray}
\end{lem}

\begin{proof} 
(1)~This is trivial for $f\in \FC(W^d)$.
General cases follow from a limiting argument.

\noindent
(2)~First we assume that $\beta\in \FC(W^d,H^{\ast})$.
By the definition of the exterior differential,
we have
$$
d\beta(w)(X,Y)=\left((D\beta)(w)[X],Y\right)-
\left((D\beta)(w)[Y],X\right),
$$
where $X,Y\in H$.
Here $(D\beta)(w)[X]$ denotes the derivative in the direction to
$X$.
Let
$\phi(\sigma)=\int_0^1\left(\beta(w+\calh(\sigma,\tau)),
\partial_{\tau}\calh(\sigma,\tau)\right)d\tau$.
We have
\begin{eqnarray}
\dot{\phi}(\sigma)
&=&
\int_0^1\left((D\beta)(w+\calh(\sigma,\tau))
[\partial_{\sigma}\calh(\sigma,\tau)],
\partial_{\tau}\calh(\sigma,\tau)\right)d\tau+
\int_0^1\left(\beta(w+\calh(\sigma,\tau)),
\partial_{\sigma}\partial_{\tau}\calh(\sigma,\tau)\right)d\tau\nonumber\\
&=&\int_0^1(d\beta)(w+\calh(\sigma,\tau))
(\partial_{\sigma}\calh(\sigma,\tau),\partial_{\tau}\calh(\sigma,\tau))d\tau
\nonumber\\
& &+\int_0^1\left((D\beta)
(w+\calh(\sigma,\tau))[\partial_{\tau}\calh(\sigma,\tau)],
\partial_{\sigma}\calh(\sigma,\tau)\right)
d\tau+\int_0^1\left(\beta(w+\calh(\sigma,\tau)),
\partial_{\sigma}\partial_{\tau}\calh(\sigma,\tau)\right)d\tau\nonumber
\end{eqnarray}
and
\begin{multline*}
\int_0^1
\left((D\beta)(w+\calh(\sigma,\tau))[\partial_{\tau}\calh(\sigma,\tau)],
\partial_{\sigma}\calh(\sigma,\tau)\right)
d\tau+\int_0^1\left(\beta(w+\calh(\sigma,\tau)),
\partial_{\sigma}\partial_{\tau}\calh(\sigma,\tau)\right)d\tau
\\
=\left(\beta(w+\calh(\sigma,1)),\partial_{\sigma}\calh(\sigma,1)\right)
-\left(\beta(w+\calh(\sigma,0)),\partial_{\sigma}\calh(\sigma,0)\right)=0.
\end{multline*}
Therefore we get
\begin{eqnarray}
\phi(1)-\phi(0)&=&
\iint_{(\sigma,\tau)\in [0,1]^2}
(d\beta)(w+\calh(\sigma,\tau))
\left(\partial_{\sigma}\calh(\sigma,\tau),\partial_{\tau}\calh(\sigma,\tau)
\right)d\sigma d\tau.
\end{eqnarray}
By the limiting argument, we complete the proof.
\end{proof}

\section{A retraction map in a Wiener space}

Let $X(t,a,w)$ be the solution of the SDE
which is defined in Proposition~\ref{solution of sde}.
In this section, we construct a retraction map
from a tubular neighborhood of the submanifold $S$
to $S$.
Recall that $S$ is defined by
$$
S=\left\{w\in \Omega~|~X(1,e,w)=e\right\}.
$$

By Proposition~\ref{solution of sde},
it is easy to see that $w\mapsto X(t,e,w)$ is $H$-differentiable map and
$$
(R_{X(t,e,w)})_{\ast}^{-1}DX(t,e,w)[h]=
\int_0^tAd\left(X(s,e,w)\right)\dot{h}(s)ds.
$$

Note that the differential form
$\alpha\in {\mathbb D}^{k,q}(\wedge^pT^{\ast}L_e(G))$ is a
measurable map from $L_e(G)$ to $\wedge^pH_0^{\ast}$.
For $\alpha\in {\mathbb D}^{k,q}(\wedge^pT^{\ast}L_e(G))$,
define the pull-back of $\alpha$ by $X$ as follows:
$$
(X^{\ast}\alpha)(w)
=\alpha(X(w))(U(w),\cdots,U(w)),
$$
where
$U(w)h=\int_0^tAd\left(X(s,e,w)\right)\dot{h}(s)ds$.
Since $X_{\ast}\mu_e=\nu_e$, $X^{\ast}\alpha\in 
L^p(\wedge^pT^{\ast}S)$.
In fact, the map $X^{\ast}$ gives isomorphisms
between Sobolev spaces as follows.

\begin{pro}\label{isomorphism}
$(1)$~Let $k$ be a non-negative integer and $q>1$.
The mapping $X^{\ast}$ is a bijective linear isometry
from ${\mathbb D}^{k,q}(\wedge^pT^{\ast}L_e(G))$ to
${\mathbb D}^{k,q}
(\wedge^pT^{\ast}S)$.

\noindent
$(2)$~For any $\alpha\in {\mathbb D}^{k,q}(\wedge^pT^{\ast}L_e(G))$,
we have $d_SX^{\ast}\alpha=X^{\ast}d\alpha$.
\end{pro}

\begin{proof}
~(1)~The surjectivity follows from 
the denseness of $X^{\ast}\FC(L_e(G))$ in
${\mathbb D}^{\infty}(S)$.
See Lemma~3.3 in \cite{aida-sobolev}.
In the case of tensors, the proof of the bijectivity can be found in 
Proposition~3.6
in \cite{aida-sobolev}.
The same proof works in the case of differential forms.

\noindent
(2)~This follows from a direct calculation.
\end{proof}

Let $\ep$ be a sufficiently small positive number.
For $a\in B_{\ep}(e)$, let 
$$
\psi_{\ep}(a,w)=-\int_0^{\cdot}
Ad\left(X(s,e,w)^{-1}\right)(\log a)ds\in H.
$$
Here $\log$ is the inverse mapping of $\exp : {\mathfrak g}\to G$.
Using this, we define 
\begin{equation}
\Psi_{\ep}(w)=w+\psi_{\ep}(X(1,e,w),w)
\qquad w\in {\cal D}_{\ep}.
\end{equation}
By Proposition~\ref{solution of sde},
$\Psi_{\ep}(w)\in S$ for all $w\in {\cal D}_{\ep}$.
Note that $\sup_{w\in \Omega}\|D\Psi_{\ep}(w)\|_{L(H,H)}<\infty$.
We define the pull-back of $\theta \in \FC(\wedge^pT^{\ast}S)$
by $\Psi_{\ep}$ as follows:
$$
(\Psi_{\ep}^{\ast}\theta)(w)=
\theta(\Psi_{\ep}(w))(D\Psi_{\ep}(w),\ldots,D\Psi_{\ep}(w)).
$$

The statement (5) in the following proposition
which follows from the result in rough path analysis
is important in the proof of our main results.

\begin{pro}\label{retraction map}
$(1)$~Let $q>1$.
For any $\eta\in {\mathbb D}^{\infty}(W^d)$,
it holds that
\begin{eqnarray}
\lefteqn{\int_{{\cal
 D}_{\ep}}|\Psi_{\ep}^{\ast}\theta(w)|^q\eta(w)d\mu(w)}\nonumber\\
& &=
\int_{B_{\ep}(e)}da\int_{S}d\mu_e(w)
|\theta(w)\left((D\Psi_{\ep})(w+\psi_{\ep}(a^{-1},w))\cdot,\cdots,
(D\Psi_{\ep})(w+\psi_{\ep}(a^{-1},w))\cdot\right)|^q\nonumber\\
& &\qquad \times \eta\left(w+\psi_{\ep}(a^{-1},w)\right)
\exp\left(-(\log a, b(1,w))-\frac{1}{2}|\log a|^2
\right),
\end{eqnarray}
where
$b(1,w)=\int_0^1Ad(X(t,e,w))\circ dw(t)$.
In particular $\|\Psi_{\ep}^{\ast}\theta\|_{L^q({\cal D}_{\ep},\mu)}
\le C_{q,r}\|\theta\|_{L^{r}(S,\mu_e)}$
for any $1<q<r$.

\noindent
$(2)$~Let $\chi$ be a smooth function on ${\mathbb R}$
such that $\chi=1$ in a neighborhood of $0$ and 
${\rm supp}~\chi\subset (-\infty, \ep^2)$.
Set $\hat{\chi}(w)=\chi\left(d(X(1,e,w),e)^2\right)$.
Define 
$T_{\chi,\ep}\theta=\hat{\chi}\Psi_{\ep}^{\ast}\theta$ for
$\theta\in \FC(\wedge^pT^{\ast}S)$.
Then $T_{\chi,\ep}$ can be extended uniquely
to a bounded linear operator
from ${\mathbb D}^{k,r}(\wedge^pT^{\ast}S)$ to ${\mathbb D}^{k,q}
(\wedge^pH^{\ast})$
for any $1<q<r$ and $k\in {\mathbb N}\cup \{0\}$.
Moreover it holds that
\begin{equation}
dT_{\chi,\ep}\theta=d\hat{\chi}\wedge \Psi_{\ep}^{\ast}\theta
+\hat{\chi}\Psi_{\ep}^{\ast}d_S\theta.\label{pull-back}
\end{equation}

\noindent
$(3)$~
The pull-back $\iota^{\ast}\beta\in {\mathbb D}^{k,q}(\wedge^pT^{\ast}S)$ 
is well-defined for
$p$-form $\beta$ on 
$W^d$ with $\|\beta\|_{k,r}<\infty$
for sufficiently large $k$ and any $1<q<r$.
Moreover it holds that
\begin{equation}
d_S\iota^{\ast}\beta=\iota^{\ast} d\beta.
\end{equation}

\noindent
$(4)$~
For sufficiently large $k$ and 
$q>1$, it holds that for any $\theta\in {\mathbb D}^{k,q}(\wedge^pT^{\ast}S)$
\begin{equation}
\iota^{\ast}T_{\chi,\ep}\theta=\theta.
\end{equation}

\noindent
$(5)$~
Let $\varphi\in H$ and
$U_{r}(\varphi)\subset {\cal D}_{\ep}$.
Then there exists a constant $C$ which depends only on
$r,\varphi,\ep$ such that
\begin{equation}
\|\Psi_{\ep}^{\ast}\theta\|_{L^2(U_r(\varphi))}\le
C\|\theta\|_{L^2(\mu_e)}.
\end{equation}
\end{pro}

\begin{proof} Noting that $X(t,e,w+\psi_{\ep}(a,w))=e^{-t\log a}X(t,e,w)$,
(1) follows from the quasi-invariance of $\nu_e$.
See \cite{gross}. 
The extension property of (2) follows from (1).
One can check the identity (\ref{pull-back}) by a direct calculation
when $\theta$ is a smooth cylindrical form.
General cases follow from an approximation argument.
Part (3) is easy to check when $\beta$ is a smooth cylindrical form.
General cases follows from a limiting argument.
Part (4) follows from $D\Psi_{\ep}(w)=P(w)$ on $S$, where
$P(w)$ is a projection operator from $H$ onto
the tangent space of $S$ at $w$.
Part (5) follows from (1) and Proposition~\ref{solution of sde} (2).
\end{proof}

\section{Proof of the main theorem}

The following immediate consequence of the ergodicity of the
Wiener measure under translations by $H$
is used to construct $f$ in Theorem~\ref{main theorem 1}
by the local data on $U_{r}(\varphi)$.

\begin{lem}\label{ergodicity}
Let $A, B$ be measurable subsets of $W^d$ with
$\mu(A)>0$ and $\mu(B)>0$.
Then there exists $h\in H$ and a measurable subset
$A_0\subset A$ such that $\mu(A_0)>0$
and $A_0+h\subset B$.
\end{lem}

Let $\chi$ be a smooth nonnegative function
such that $\chi(u)=1$ for $u\le 4\ep^2/9$
and $\chi(u)=0$ for $u\ge 9\ep^2/16$.
Let $\hat{\chi}(w)=\chi\left(d(X(1,e,w),e)^2\right)$.

\begin{lem}\label{final lemma}
Let $\theta$ be the $1$-form on $S$ 
in Theorem~$\ref{main theorem 1}$.
Let $\ep$ be a sufficiently small positive number.
Let $\beta=\Psi_{\ep}^{\ast}\theta$.
Let $1<q<p$.
Then there exists a measurable function $F$ on 
${\cal D}_{\ep}$ and
$\rho_n$~$(n\in {\mathbb N})$ on $\Omega$ such that
the following hold.

\noindent
$(1)$~
The function $\rho_n$ is a bounded
non-negative $\infty$-quasi-continuous function
and $\rho_n\in {\mathbb D}^{\infty}(W^d)$ holds.

\noindent
$(2)$~For any $r>1$ and $k\in {\mathbb N}$,
$\lim_{n\to\infty}C^k_r(\{w\in \Omega~|~\rho_n(w)=1\}^c)=0$ and
$\lim_{n\to\infty}\|\rho_n-1\|_{r,k}=0$.

\noindent
$(3)$~
There exists $F_n\in {\mathbb D}^{1,2}(W^d)\cap 
{\mathbb D}^{\infty,q}(W^d)$ such that
$F(w)=F_n(w)$ and
$dF_n(w)=\beta(w)$ for $\mu$-almost all $w$ of
$\{w\in \Omega~|~\rho_n(w)\ne 0\}\cap {\cal D}_{\ep/2}$.

\noindent
$(4)$~
Let $\hat{F}_n=\tilde{F}_n\rho_n\hat{\chi}$,
where $\tilde{F}_n$ is a $(q,\infty)$-quasi-continuous version of
$F_n$.
It holds that
$\hat{F}_n\in {\mathbb D}^{1,2}(W^d)\cap 
{\mathbb D}^{\infty,q}(W^d)$ and
\begin{equation}
d\hat{F}_n=
\beta\rho_n\hat{\chi}+
\tilde{F}_nd\rho_n\hat{\chi}+
\tilde{F}_n\rho_nd\hat{\chi}.
\end{equation}
\end{lem}

\begin{proof} Let $\chi_0$ be a smooth decreasing function 
on ${\mathbb R}$ such that $\chi_0(u)=1$
for $u\le 9\ep^2/4$ and ${\rm supp}~\chi_0\subset (-\infty,4\ep^2)$.
Let $\gamma=T_{\chi_0,2\ep}\theta$.
Then $\gamma\in {\mathbb D}^{\infty,q}(W^d,H^{\ast})$.
Also note that $\gamma=\beta$  and $d\gamma=0$ on ${\cal D}_{\ep}$.
The latter result follows from Proposition~\ref{retraction map}~(2).
Let $U_{\sqrt{2}\kappa_i}(\varphi_i)$~$(i=1,2,\ldots)$
be the covering of ${\cal D}_{\ep}$ in 
Lemma~\ref{covering}~(3) and Proposition~\ref{covering2}~(2).
Let us choose $r_i$ such that
$4\kappa_i/3<r_i<\sqrt{2}\kappa_i$.
Since $d\gamma=0$ on $U_{\sqrt{2}\kappa_i}(\varphi_i)$
and $\gamma\in 
L^2(U_{\sqrt{2}\kappa_i}(\varphi_i))$,
by Theorem~\ref{a vanishing theorem}, we see that there exist
$g_i\in {\mathbb D}^{\infty,q}(W^d)\cap {\mathbb D}^{1,2}(W^d)$
such that $dg_i=\gamma$ on $U_{r_i}(\varphi_i)$.
However $g_i$ on $U_{r_i}(\varphi_i)$ is not determined uniquely,
in fact, there is an ambiguity of additive constant.
Actually we prove that there are constants $c_i$ 
and a measurable function $F$ on ${\cal D}_{\ep}$
such that $F(w)=g_i(w)+c_i$ almost all $w\in U_{r_i}(\varphi_i)$
for any $i$ and $r_i$.
First set $c_1=0$.
We define $c_i$~$(i\ge 2)$ inductively in the following way.
Suppose that there exist $c_1,\ldots,c_i$ and a measurable function
$G_i$ on $\cup_{j=1}^iU_{r_j}(\varphi_j)$ such that
$G_i(w)=g_j(w)+c_j$ almost all $w\in U_{r_j}(\varphi_j)$
for all $1\le j\le i$.
By Theorem~\ref{a vanishing theorem}, there exist
$G_{i,j}\in {\mathbb D}^{1,2}(W^d)\cap {\mathbb D}^{\infty,q}(W^d)$ 
such that $G_{i,j}(w)=G_i(w)$ on $U_{r_j}(\varphi_j)$.
We prove that 
for any $\{r_j'\}$ with $4\kappa_j/3<r_j'<r_j$~$(1\le j\le i)$
there exists $H_i\in {\mathbb D}^{1,2}(W^d)\cap 
{\mathbb D}^{\infty,q}(W^d)$ such that
$H_i=G_i$ and $dH_i=\beta$ on $\cup_{j=1}^iU_{r_j'}(\varphi_j)$.

Note that there exist $\phi_j\in {\mathbb D}^{\infty}(W^d)$
~$(1\le j\le i+2)$ such that the following identity holds.
For $1\le j\le i$
$$
\phi_j(w)=
\begin{cases}
1\qquad~w\in U_{r_j'+\ep_j}(\varphi_j),\\
0\qquad~w\in U_{r_j'+\ep_j'}(\varphi_j)^c\\
\end{cases}
$$
and
$$
\phi_{i+1}(w)=
\begin{cases}
0\qquad~w\in \cup_{j=1}^iU_{r_j'+\ep_j-\delta_j'}(\varphi_j),\\
1\qquad~w\in 
\left(\cup_{j=1}^iU_{r_j'+\ep_j-\delta_j}(\varphi_j)\right)^c,\\
\end{cases}
$$
$$
\phi_{i+2}(w)=
\begin{cases}
1\qquad~w\in \cup_{j=1}^iU_{r_j'+\ep_j-\delta_j'-\tau_j'}(\varphi_j),\\
0\qquad~w\in 
\left(\cup_{j=1}^iU_{r_j'+\ep_j-\delta_j'-\tau_j}(\varphi_j)\right)^c.\\
\end{cases}
$$
Here we choose positive numbers such that
$0<\delta_j<\delta_j'<\ep_j<\ep_j'$,
$\ep_j-\delta_j'-\tau_j'>0$,
$0<\tau_j<\tau_j'$ and $r_j'+\ep_j'<r_j$.
These functions can be constructed explicitly in a similar way to
$\tilde{\rho}(w)$ in the proof of Theorem~\ref{a vanishing theorem} using 
mollifiers.
Since $\sum_{j=1}^{i+1}\phi_j(w)\ge 1$ for any $w\in \Omega$,
$$
\tilde{\phi}_j(w)=\frac{\phi_j(w)}{\sum_{j=1}^{i+1}\phi_j(w)}
$$
belongs to ${\mathbb D}^{1,2}(W^d)\cap {\mathbb D}^{\infty,q}(W^d)$
and $\sum_{j=1}^{i+1}\tilde{\phi}_j(w)=1$ for all $w\in \Omega$.
This is a partition of unity associated with
the covering of $\Omega$:
$$
U_{r_j'+\ep_j'}(\varphi_j)~(1\le j\le i),\quad
\left(\cup_{j=1}^iU_{r_j'+\ep_j-\delta_j'}(\varphi_j)\right)^c
$$
Since $\phi_{i+2}(w)\phi_{i+1}(w)=0$ for all $w\in \Omega$,
we have
\begin{eqnarray}
G_i(w)\phi_{i+2}(w)&=&
\sum_{j=1}^{i+1}G_i(w)\phi_{i+2}(w)\tilde{\phi}_j(w)\nonumber\\
&=&\sum_{j=1}^iG_{i,j}(w)\phi_{i+2}(w)\tilde{\phi}_j(w).
\end{eqnarray}
Therefore $H_i=G_{i}\phi_{i+2}$ is the desired function.

By using the existence of
$H_i$ and the $H$-simply connectedness of ${\cal D}_{\ep}$,
we next prove the existence of a measurable function
$G_{i+1}$ on $\cup_{j=1}^{i+1}U_{r_j'}(\varphi_j)$
and a constant $c_{i+1}$
such that $G_{i+1}(w)=G_i(w)$ for almost all 
$w\in \cup_{j=1}^{i}U_{r_j'}(\varphi_j)$ and 
$G_{i+1}(w)=g_{i+1}(w)+c_{i+1}$ for almost all 
$w\in U_{r_{i+1}'}(\varphi_{i+1})$.
Since $\mu\left((\cup_{j=1}^iU_{r_j'}(\varphi_j))\cap
U_{r_{i+1}'}(\varphi_{i+1})\right)>0$,
there exists a piecewise linear path $\varphi\in H$, 
$\delta>0$ and $1\le i_0\le i$
such that
$U_{\delta}(\varphi)\subset U_{r_{i+1}'}(\varphi_{i+1})\cap 
U_{r_{i_0}'}(\varphi_{i_0})$.
Because $d(g_{i+1}-g_{i_0})=0$ on $U_{\delta}(\varphi)$,
$g_{i+1}(w)-g_{i_0}(w)$ is equal to a constant almost all
$w$ on $U_{\delta}(\varphi)$.
We choose $c_{i+1}$ such that 
$g_{i+1}(w)+c_{i+1}=g_{i_0}(w)+c_{i_0}(=G_i(w))$
almost all $w\in U_{\delta}(\varphi)$.
It suffices to prove that 
\begin{equation}
g_{i+1}(w)+c_{i+1}=G_i(w)
\qquad \mbox{for almost all 
$
w\in\left(\cup_{j=1}^iU_{r_j'}(\varphi_j)\right)\cap
U_{r_{i+1}'}(\varphi_{i+1})
$}.\label{extension of g}
\end{equation}
Suppose that there exists a set
$B\subset U_{r_{i_1}'}(\varphi_{i_1})\cap 
U_{r_{i+1}'}(\varphi_{i+1})$ of positive measure for some $1\le i_1\le i$
and $c'>0$
such that 
$$
|g_{i+1}(w)+c_{i+1}-G_i(w)|>c' 
\qquad \mbox{for all
$w\in B$}.
$$
By the ergodicity of the Wiener measure,
there exists a subset
$A\subset U_{\delta}(\varphi)$ with positive measure 
and $h\in H$ such that
$A+h\subset B$.
Choose a point $\eta\in A$ such that
$\mu(V_r(\eta)\cap A)>0$ for all $r>0$,
where $V_r(\eta)$ is defined by
(\ref{ball}).
By the $H$-connectivity of
$\cup_{j=1}^iU_{r_j'}(\varphi_j)$ and
$U_{r_{i+1}}(\varphi_{i+1})$,
there exists two $C^{\infty}$-curves 
$h(i,\tau)$~$(0\le \tau\le 1)$ on $H$
such that $h(i,0)=0$, $h(i,1)=h$~$(i=0,1)$
and
$\eta+h(0,\tau)\subset \cup_{j=1}^iU_{r_j'}(\varphi_j)$
$\eta+h(1,\tau)\subset U_{r_{i+1}'}(\varphi_{i+1})$
for all $0\le \tau\le 1$.
By choosing $\delta$ to be a sufficiently small positive number,
we have for all $0\le \tau\le 1$,
\begin{eqnarray}
V_{\delta}(\eta)+h(0,\tau)&\subset& \cup_{j=1}^iU_{r_{i}'}
(\varphi_i)\\
V_{\delta}(\eta)+h(1,\tau)&\subset& U_{r_{i+1}'}(\varphi_{i+1})
\end{eqnarray}
By the $H$-simply connectedness of ${\cal D}_{\ep}$,
there exists a $C^{\infty}$-map 
${\cal H}={\cal H}(\sigma,\tau)$~$(0\le \sigma,\tau\le 1)$ such that
${\cal H}(0,\tau)=h(0,\tau)$, ${\cal H}(1,\tau)=h(1,\tau)$ and
$\eta+{\cal H}(\sigma,\tau)\subset {\cal D}_{\ep}$ for all 
$(\sigma,\tau)\in [0,1]^2$.
Using the continuity of $X(1,e,\cdot)$ in the topology
of $d_{\Omega}$, we see that there exists $0<\delta'<\delta$ 
such that for all $0\le \sigma,\tau \le 1$
$V_{\delta'}(\eta)+{\cal H}(\sigma,\tau)\subset {\cal D}_{\ep}$.
Note that $dg_{i+1}=\beta$ on $U_{r_{i+1}'}(\varphi_{i+1})$
and $dH_i=\beta$ on $\cup_{j=1}^iU_{r_i'}(\varphi_i)$.
By applying Lemma~\ref{stokes} and noting that
$d\beta=0$ on ${\cal D}_{\ep}$,
we obtain
$$
\left(g_{i+1}(w+h)+c_{i+1}\right)-
\left(g_{i+1}(w)+c_{i+1}\right)
=G_{i}(w+h)-G_{i}(w)
\qquad\mbox{for almost all $w\in A\cap V_{\delta'}(\eta)$}.
$$
This is a contradiction.
This implies (\ref{extension of g}).
Inductively,
we obtain a measurable function $F$
on ${\cal D}_{\ep}$ such that for any $i$
$F(w)=g_i(w)+c_i$ for some $c_i$ and there exists
$H_i\in {\mathbb D}^{1,2}(W^d)\cap {\mathbb D}^{\infty,q}(W^d)$
such that $F(w)=H_i(w)$
for almost all $w\in \cup_{j=1}^iU_{r_j}(\varphi_j)$.
Let $\chi_1$ be a non-negative smooth non-increasing function such that
$\chi_1(u)=1$ for $u\le (1/2)^m$ and $\chi_1(u)=0$ for
$u\ge (2/3)^m$.
Let
\begin{eqnarray*}
\chi_{n,2}(w)&=&\chi_1\left(n^{-m}
\left(\sum_{1\le i,j\le d}\|C(w^i,w^j)\|_{m,\theta}^m+\sum_{1\le k\le d}
\|w^k\|_{m,\theta'/2}^m\right)
\right),\\
\chi_{\kappa,N,3}(w)&=&
\chi_1\Biggl(
\kappa^{-m}\Biggl(
\sum_{k=1}^n\|w(N)^{\perp,k}\|_{m,\theta'/2}^m+
\sum_{1\le i<j\le d}
\|C(w(N)^{\perp,i},w(N)^{\perp,j})\|_{m,\theta}^m\nonumber\\
& &+\sum_{1\le i\le j\le d}
\|C(w(N)^{i},w(N)^{\perp,j})\|_{m,\theta}^m+
\sum_{1\le i\le j\le d}
\|C(w(N)^{\perp,i},w(N)^{j})\|_{m,\theta}^m
\Biggr)
\Biggr),
\end{eqnarray*}
and set $\chi_{n,\kappa,N,4}(w)=\chi_{n,2}(w)\chi_{\kappa,N,3}(w)$.
Then we have
$
\left\{\chi_{n,\kappa,N,4}(w)\ne 0\right\}\cap 
{\cal D}_{\ep_2}
\subset
{\cal D}_{\ep_2,n,N,\kappa}.
$
Now choosing $\kappa=\kappa(n)$ to be sufficiently small according to $n$
as in Lemma~\ref{covering}, we have
for sufficiently large $L_0\in {\mathbb N}$,
$$
{\cal D}_{\ep_2,n,N,\kappa(n)}
\subset
\cup_{i=1}^{L_0}U_{4\kappa_i/3}(\varphi_i).
$$
Therefore letting $N=a(\kappa(n))$ to be a sufficiently large 
natural number
according to $\kappa=\kappa(n)$, 
we see that
$\rho_n(w)=\chi_{n,\kappa(n), a(\kappa(n)),4}
(w)$ satisfies the properties (1), (2).
As for (3), it suffices to set $F_n=H_i$ for
sufficiently large $i$.
Part (4) follows from (3).
\end{proof}

We now can prove the main theorems.

\begin{proof}
[Proof of Theorem~$\ref{main theorem 1}$]
Let $\rho_n$ be the function in Lemma~\ref{final lemma}.
Then (1) holds.
Let $f_n=\tilde{F}_n$.
We construct $f$ on $S$.
Let $C_n=\{\rho_n\ne 0\}\cap {\cal D}_{\ep/2}$.
By Lemma~\ref{final lemma}~(2),
$\lim_{n\to\infty}\mu_e(C_n^c)=0$.
For $n,n'\in {\mathbb N}$, we have
\begin{equation}
\tilde{F}_n(w)=\tilde{F}_{n'}(w)=F(w)
\quad\mbox{for $\mu$-almost all $w$ of $C_n\cap C_{n'}$}.
\end{equation}
Hence there exists a Borel measurable subset $B_{n,n'}$ such that
$C^k_q(B_{n,n'})=0$ and
\begin{equation}
\tilde{F}_n(w)=\tilde{F}_{n'}(w) \qquad 
\mbox{for all $w\in C_n\cap C_{n'}\cap B_{n,n'}^c$}.
\end{equation}
This implies that
$\tilde{F}_n(w)=\tilde{F}_{n'}(w)$
for $\mu_e$-almost all $w\in C_n\cap C_{n'}\cap S$.
Therefore there exists a measurable function
$f$ on $S$
\begin{equation}
f(w)=\tilde{F}_n(w)\qquad
\mbox{for $\mu_e$-almost all $w\in C_n\cap S$.}
\end{equation}
For this $f$ and $f_n$, (2)~(i), (ii) holds.
We prove (ii).
Lemma~\ref{final lemma}~(3) shows that $dF_n=\beta=T_{\chi_0,2\ep}$ on
$C_n$.
Hence, using Proposition~\ref{retraction map}~(3) and (4),
we can conclude that $d_S(\iota^{\ast}F_n)=\theta$
on $\{\rho_n\ne 0\}\cap S$
which implies $d_Sf_n=\theta$ on $\{\rho_n\ne 0\}\cap S$.
We prove (2)~(iii).
Note that $f\rho_n\eta=f_n\rho_n\eta\in {\mathbb D}^{\infty,q-}(W^d)$.
Hence by Theorem~4.3 in \cite{sugita}, we have
$f\rho_n\eta\in L^1(S,\mu_e)$.
The equation in (2)~(iv) is equivalent to
$$
\int_Sf_n\rho_nd_S^{\ast}\left(\rho_n\eta\right)d\mu_e=
\int_S\left(d_S\left(f_n\rho_n\right), \rho_n\eta\right)d\mu_e
$$
which follows from the integration by parts formula on $S$.
We prove (2)~(v).
By the integration by parts formula on $S$,
we have
\begin{equation}
\int_S\psi_K'(\hat{F}_n(w))\left(d_S\hat{F}_n(w),\eta(w)\right)
d\mu_e(w)=
\int_S\psi_K(\hat{F}_n(w))d_S^{\ast}\eta(w)
d\mu_e(w).\label{ibp on S}
\end{equation}
By Lemma~\ref{final lemma}~(4),
we get
\begin{equation}
d_S\hat{F}_n=\theta\rho_n+\tilde{F}_nd\rho_n.\label{dSF}
\end{equation}
Substituting (\ref{dSF}) into (\ref{ibp on S})
and replacing $\eta$ by $\rho_n\eta$, we have
\begin{multline}
\int_S\psi'_K\left(f(w)\rho_n(w)\right)
\Bigl(\theta(w)\rho_n(w)+f(w)d\rho_n(w), \rho_n(w)\eta(w)\Bigr)
d\mu_e(w)\\
=
\int_S\psi_K\left(f(w)\rho_n(w)\right)
d_S^{\ast}\left(\rho_n\eta\right)(w)
d\mu_e(w).
\end{multline}
Here we have used that $f(w)=\tilde{F}_n(w)$ 
$\mu_e$-almost all $w$ on $\{\rho_n\ne 0\}$.
Letting 
$n\to\infty$, we obtain
\begin{eqnarray}
\int_S\psi'_K(f(w))(\theta(w),\eta(w))d\mu_e(w)
&=&\int_S\psi_K(f(w))d_S^{\ast}\eta(w)d\mu_e(w).
\end{eqnarray}
This implies that the weak derivative of $\psi_K(f)$ is
$\psi_K'(f)\theta$.
Since $\left(d_S^{\ast}d_S, \FC(W^d)\right)$ is essentially self-adjoint
(see \cite{aida-ou}, \cite{aida-sobolev}),
$\psi_K(f)\in {\mathbb D}^{1,2}(S)$ and
$d_S\psi_K(f)=\psi_K'(f)\theta$.
\end{proof}

We prove Theorem~\ref{main theorem on loop}.

\begin{proof}
[Proof of Theorem~$\ref{main theorem on loop}$]
Let $\bar{\alpha}=X^{\ast}\alpha$.
Then $\bar{\alpha}\in L^2(\wedge^1T^{\ast}S)\cap {\mathbb D}^{\infty,p}
(\wedge^1T^{\ast}S)$ and $d_S\bar{\alpha}=0$ on $S$.
By Theorem~\ref{main theorem 1}, there exists a measurable function
$g$ on $S$ such that $d_Sg=\bar{\alpha}$.
By using Proposition~\ref{isomorphism}~(1),
we see that there exists a measurable function 
$f$ on $L_e(G)$ such that $X^{\ast}f=g$ for $\mu_e$-almost all
$w$.Hence $X^{\ast}f^K=g^K$.
By Proposition~\ref{isomorphism} and Theorem~\ref{main theorem 1},
we have $f^K\in {\mathbb D}^{1,2}(L_e(G))$
and $df^K=\psi'_K(f)\alpha$ which proves (1).
Since $df^K=\psi'_K(f)\alpha$, using a similar argument to
the proof of Lemma~14 in \cite{aida-irred}, we have
$$
f^K(e^{\ep h}\gamma)-f^K(\gamma)=
\int_0^{\ep}\left(\psi'_K(f(\gamma))\alpha(e^{sh}\gamma),h\right)ds.
$$
Letting $K\to\infty$, we complete the proof of (2).
Part (3) follows from (2).
\end{proof}

We need the Weitzenb\"ock formula for $\square$
to prove Theorem~\ref{main theorem 2}.
It will be proved below.

\begin{lem}\label{final lemma2}
Let ${\Casimir}=\sum_{i=1}^d({\rm ad \ep_i})^2$,
where $\{\ep_i\}$ denotes an orthonormal system of ${\mathfrak g}$.
Then
\begin{eqnarray}
\left(\square\alpha,h\right)&=&
\left(\nabla_{\nu_{e}}^{\ast}\nabla\alpha+
\alpha+T_{b(1)}\alpha,h\right)\nonumber\\
& &+\int_0^1\left((\Casimir \alpha)_t,h_t\right)dt
-\int_0^1\int_0^1\left(\Casimir \alpha_t,h_s\right)dtds,
\end{eqnarray}
where
$(T_v\alpha)_t=\int_0^t[\alpha_s,v]ds-t\int_0^1[\alpha_s,v]ds$
~$(v\in {\mathfrak g})$,
$b(t,\gamma)=\int_0^t(R_{{\gamma_s}})_{\ast}^{-1}\circ d\gamma_s
\in {\mathfrak g}$.
Here $[\cdot,\cdot]$ denotes the Lie bracket.
Also $(\square\alpha,h)$ denotes the coupling of
$\square\alpha(\gamma)\in H_0^{\ast}$ and $h\in H_0$.
\end{lem}

For simplicity we denote
$$
\square=\nabla_{\mu_e}^{\ast}\nabla+I+T_{b(1)}+T_2+T_3,
$$
where $T_2,T_3$ are 0-order operators acting on 
$1$-forms corresponding to the terms
$\int_0^1\left((\Casimir \alpha)_t,h_t\right)dt$
and
$-\int_0^1\int_0^1\left(\Casimir \alpha_t,h_s\right)dtds$
respectively.

\bigskip

\begin{proof}
[Proof of Theorem~$\ref{main theorem 2}$]
~Let $\alpha\in L^2(\wedge^1T^{\ast}L_e(G))$
and assume that $\square \alpha=0$.
We need to show that $\alpha\in 
\cap_{1<p<2} {\mathbb D}^{\infty,p}(\wedge^1T^{\ast}L_e(G))$.
Let $\theta\in \FC(\wedge^1T^{\ast}L_e(G))$.
Then
\begin{eqnarray}
\left(\alpha,\nabla_{\nu_e}^{\ast}\nabla\theta\right)&=&
\left(\alpha,(\square-I-T_{b(1)}-T_2-T_3)\theta\right)\nonumber\\
&=&-\left((I+T_{b(1)}^{\ast}+T_2^{\ast}+T_3^{\ast})\alpha,\theta\right).
\end{eqnarray}
Since $b(1)\in \cap_{p>1}L^p(L_e(G),d\nu_e)$,
the weak derivative 
$\nabla_{\nu_e}^{\ast}\nabla\alpha$ belongs to
$\cap_{1<p<2}L^p(\wedge^1T^{\ast}L_e(G))$.
Hence
by Theorem~2.16 in \cite{aida-sobolev},
$
\alpha\in \cap_{1<p<2}{\mathbb D}^{2,p}
(\wedge^1T^{\ast}L_e(G))
$
which implies $\alpha\in \cap_{1<p<2}{\mathbb D}^{\infty,p}(\wedge^1L_e(G))$.
Also note that $d\alpha=0$.
Let $f$ and $f^K$ be the function in Theorem~\ref{main theorem on loop}
Then $df^K=\psi_K'(f)\alpha$ on $L_e(G)$.
Note that $\alpha$ satisfies the equation 
$d^{\ast}\alpha=0$ on $L_e(G)$.
Hence we have
\begin{eqnarray*}
\int_{L_e(G)}|\alpha(\gamma)|_{T_{\gamma}L_e(G)}^2d\nu_e(\gamma)&=&
\lim_{K\to\infty}\int_{L_e(G)}\left(\alpha(\gamma),
\psi_K'(f)\alpha(\gamma\right)_{T_{\gamma}L_e(G)}
d\nu_e(w)\nonumber\\
&=&
\lim_{K\to\infty}\int_{L_e(G)}d^{\ast}\alpha(\gamma)f^K(\gamma)
d\nu_e(\gamma)\nonumber\\
&=&0.
\end{eqnarray*}
This implies $\alpha=0$ which proves $\ker \square=\{0\}$.
We prove (\ref{decomposition}).
Let $H_1=\overline{\left\{df~|~f\in \FC(L_e(G))\right\}}$
and 
$H_2=\overline{\left\{
d^{\ast}\alpha~|~\alpha\in \FC(\wedge^2T^{\ast}L_e(G))\right\}}$.
It is easy to see $H_1\cap H_2=\{0\}$.
Let $H_3=(H_1\oplus H_2)^{\perp}$.
Assume there exists a non-zero $\alpha\in H_3$.
Then for any smooth cylindrical $1$-form $\beta$,
$$
\left(\square\beta,\alpha\right)_{L^2(\wedge^1T^{\ast}L_e(G))}
=\left(dd^{\ast}\beta, \alpha\right)+
\left(d^{\ast}d\beta,\alpha\right).
$$
Since $d^{\ast}\beta$ and $d\beta$ can be approximated by
smooth cylindrical functions and $1$-forms respectively,
we obtain $(\square\beta,\alpha)=0$.
This shows $\square\alpha=0$ in weak sense.
By the essential-selfadjointness of $(\square,\FC(\wedge^1T^{\ast}L_e(G)))$
which is due to \cite{shigekawa2},
this implies $\alpha\in {\rm D}(\square)$ and
$\square\alpha=0$.
Hence $\alpha=0$ which completes the proof.
\end{proof}

We give a proof of Weitzenb\"ock formula for the sake of completeness.
The reader may find the proof in \cite{fang-franchi}.
Also we note that this calculation is essentially similar to that of
$\Gamma_2$ of the Dirichlet form in \cite{getzler, shigekawa1}.
First we recall some results in \cite{aida-sobolev}.

\begin{lem}\label{final lemma3}
Let $X_h$ be the right-invariant vector field corresponding to
$h\in H$.

\noindent
$(1)$~We have
\begin{equation*}
\int_{L_e(G)}X_hf\cdot gd\nu_e=
\int_{L_e(G)}f\cdot\left(-X_hg+(h,b)g\right)d\nu_e.
\end{equation*}
Here $(h,b)=\int_0^1(\dot{h}(s),db(s))$.

\noindent
$(2)$~For any $h,k\in H$,
\begin{equation*}
\nabla_{X_h}X_k=X_{-P_0\int_0^{\cdot}[h_s,\dot{k}_s]ds},
\end{equation*}
where $P_0h=h_t-th_1$.

\noindent
$(3)$~For any $h,k\in H_0$,
$$
[X_h,X_k]=X_{[k,h]},
$$
where $[X_h,X_k]$ is the Lie bracket of the vector field on
$L_e(G)$.
\end{lem}

\begin{proof}
[Proof of Lemma~$\ref{final lemma2}$]
 We fix a complete orthonormal system $\{e_i\}$ of $H_0$.
By Lemma~\ref{final lemma3}, for any smooth $1$-form $\alpha$
on $L_e(G)$,
\begin{equation*}
d^{\ast}\alpha=\sum_{i}\left(-X_{e_i}\left(\alpha(e_i)\right)
+(e_i,b)\alpha(e_i)\right),
\end{equation*}
where $\alpha(e_i)$ stands for
the coupling of $\alpha(\gamma)\in H_0^{\ast}$ 
and $e_i\in H_0$.
Let $\beta$ be a smooth $2$-form on $L_e(G)$.
By Lemma~\ref{final lemma3},
\begin{eqnarray*}
(d^{\ast}\beta)(e_k)
&=&-\sum_{i}X_{e_i}\left(\beta(e_i,e_k)\right)+\sum_i(e_i,b)\beta(e_i,e_k)
-\sum_{i<j}\beta(e_i,e_j)\left([e_j,e_i],e_k\right).
\end{eqnarray*}
Using these, we have for $h\in H_0$
\begin{multline*}
\left(\left(d^{\ast}d+dd^{\ast}\right)\alpha\right)(h)=
-\sum_{i}X_{e_i}\left(X_{e_i}(\alpha(h))\right)+
\sum_{i}(e_i,b)X_{e_i}(\alpha(h))
+\alpha(h)\\
+\sum_{i<j}\alpha\left([e_j,e_i]\right)
\left([e_j,e_i],h\right)+\alpha\left(P_0\int_0^{\cdot}[h_s,db_s]\right)
-\sum_i(e_i,b)\alpha([e_i,h])\nonumber\\
+\sum_iX_{[h,e_i]}(\alpha(e_i))+
\sum_iX_{e_i}(\alpha([h,e_i]))-\sum_{i<j}\left(
X_{e_i}(\alpha(e_j))-X_{e_j}(\alpha(e_i))
\right)\left([e_j,e_i],h\right)
\end{multline*}
By the definition of the covariant derivative,
we have
\begin{eqnarray*}
(\nabla_{\nu_e}^{\ast}\nabla\alpha)(h)&=&
-\sum_iX_{e_i}\left(X_{e_i}(\alpha(h))\right)+\sum_{i}(e_i,b)X_{e_i}(\alpha(h))
-\sum_i(e_i,b)\alpha(\nabla_{e_i}h)\nonumber\\
& &+2\sum_iX_{e_i}\left(\alpha(\nabla_{e_i}h)\right)-
\sum_i\alpha(\nabla_{e_i}\nabla_{e_i}h),
\end{eqnarray*}
where $\nabla_hk=-P_0\left(\int_0^{\cdot}[h_s,\dot{k}_s]ds\right)$
for $h,k\in H_0$.
Hence
\begin{eqnarray*}
\left(\left(d^{\ast}d+dd^{\ast}\right)\alpha\right)(h)&=&
\left(\nabla_{\nu_e}^{\ast}\nabla\alpha\right)(h)+
\alpha(h)+\sum_i\alpha(\nabla_{e_i}\nabla_{e_i}h)\nonumber\\
& &+\frac{1}{2}\sum_{i,j}\alpha\left([e_j,e_i]\right)
\left([e_j,e_i],h\right)+I_1+I_2.
\end{eqnarray*}
Here
\begin{eqnarray*}
I_1&=&\alpha\left(P_0\int_0^{\cdot}[h_s,db_s]\right)
-\sum_i(e_i,b)\alpha([e_i,h])+\sum_i(e_i,b)\alpha(\nabla_{e_i}h),
\end{eqnarray*}
\begin{eqnarray*}
I_2&=&
\sum_iX_{[h,e_i]}\left(\alpha(e_i)\right)+
\sum_iX_{e_i}\left(\alpha([h,e_i])\right)-\sum_{i<j}\left(
X_{e_i}(\alpha(e_j))-X_{e_j}(\alpha(e_i))
\right)\left([e_j,e_i],h\right)\nonumber\\
& &-2\sum_iX_{e_i}\left(\alpha(\nabla_{e_i}h)\right).
\end{eqnarray*}
By the explicit calculation,
$
I_1=(T_{b(1)}\alpha)(h)
$
and
$I_2=0$.
We calculate $\frac{1}{2}\sum_{i,j}\alpha\left([e_j,e_i]\right)
\left([e_j,e_i],h\right)$ and
$\sum_i\alpha(\nabla_{e_i}\nabla_{e_i}h)$.
\begin{eqnarray*}
\sum_i\alpha(\nabla_{e_i}\nabla_{e_i}h)
&=&
\sum_i\int_0^1\left(\dot{\alpha}_t, -\left[e_i(t),
[e_i(t),\dot{h}_t]-\int_0^1[e_i(s),\dot{h}_s]ds\right]\right)dt\nonumber\\
&=&
-\sum_i\int_0^1\left([\dot{\alpha}_t,e_i(t)],
		[\dot{h}_t,e_i(t)]\right)dt
\nonumber\\
& &+\sum_i\left(\int_0^1\left[\dot{\alpha}_t, e_i(t)\right]dt,
\int_0^1\left[\dot{h}_s,e_i(s)\right]ds\right).
\end{eqnarray*}
\begin{eqnarray*}
\lefteqn{\frac{1}{2}\sum_{i,j}\alpha\left([e_j,e_i]\right)
\left([e_j,e_i],h\right)}\nonumber\\
& &=
\sum_i\int_0^1\left(
[\dot{\alpha}_t,e_i(t)]-\int_0^1[\dot{\alpha}_t,e_i(t)]dt,
[\dot{h}_t,e_i(t)]-\int_0^1[\dot{h}_t,e_i(t)]dt
\right)dt\nonumber\\
& &-\sum_i\int_0^1\left(
[\dot{\alpha}_t,e_i(t)]-\int_0^1[\dot{\alpha}_t,e_i(t)]dt,
\int_0^t[\dot{h}_s,\dot{e}_i(s)]ds-
\int_0^1\left(\int_0^u[\dot{h}_s,\dot{e}_i(s)]ds\right)du\right)dt.
\nonumber\\
& &
\end{eqnarray*}
Thus
\begin{eqnarray*}
\lefteqn{\sum_i\alpha(\nabla_{e_i}\nabla_{e_i}h)
+\frac{1}{2}\sum_{i,j}\alpha\left([e_j,e_i]\right)
\left([e_j,e_i],h\right)}\nonumber\\
& &=
-\sum_i\int_0^1\left(
[\dot{\alpha}_t,e_i(t)]-\int_0^1[\dot{\alpha}_t,e_i(t)]dt,
\int_0^t[\dot{h}_s,\dot{e}_i(s)]ds\right)dt\nonumber\\
& &=-\sum_{i=1}^d\left(\int_0^1\left[[\alpha_t,\ep_i],\ep_i\right]dt,
\int_0^1h_tdt \right)-
\sum_{i=1}^d\int_0^1\left([\alpha_t,\ep_i],
		     [h_t,\ep_i]\right)dt\nonumber\\
& &=-\left(\int_0^1(\Casimir\alpha)_tdt, \int_0^1h_tdt\right)
+\int_0^1\left(\Casimir\alpha_t,h_t\right)dt.
\end{eqnarray*}
This completes the proof.
\end{proof}

\noindent
{\bf Acknowledgment}~
I would like to thank the referee for 
the careful reading of the manuscript 
and valuable comments 
that helped to improve the paper.

\end{document}